\documentclass[12pt, reqno]{amsart}
\usepackage{txfonts}      
\usepackage{amssymb}
\usepackage{eucal}
\usepackage{amsmath}
\usepackage{amscd}
\usepackage[dvips]{color}
\usepackage{multicol}
\usepackage[all]{xy}           
\usepackage{graphicx}
\usepackage{color}
\usepackage{colordvi}
\usepackage{xspace}
\usepackage{enumerate}
\usepackage{tikz}

\usepackage{ifpdf}
\ifpdf
 \usepackage[colorlinks,final,backref=page,hyperindex]{hyperref}
\else
 \usepackage[colorlinks,final,backref=page,hyperindex,hypertex]{hyperref}
\fi

\begin{document}
\newcommand {\emptycomment}[1]{} 

\newcommand{\mysymbol}{pp-post-Lie\xspace}
\newcommand{\Mysymbol}{Pp-post-Lie\xspace}
\newcommand{\mysymbolnew}{invariant\xspace}
\newcommand{\mysymbolcph}{generalized Hessian\xspace}
\newcommand{\mysymbolcphp}{generalized pseudo-Hessian\xspace}
\newcommand{\mysymbolcphh}{ Hessian\xspace}
\newcommand{\mysymbolnewi}{invariant bilinear form\xspace}

\newcommand{\nc}{\newcommand}
\newcommand{\delete}[1]{}
\nc{\mfootnote}[1]{\footnote{#1}} 
\nc{\todo}[1]{\tred{To do:} #1}

\nc{\mlabel}[1]{\label{#1}}  
\nc{\mcite}[1]{\cite{#1}}  
\nc{\mref}[1]{\ref{#1}}  
\nc{\meqref}[1]{\eqref{#1}} 
\nc{\mbibitem}[1]{\bibitem{#1}} 

\delete{
\nc{\mlabel}[1]{\label{#1}  
{\hfill \hspace{1cm}{\bf{{\ }\hfill(#1)}}}}
\nc{\mcite}[1]{\cite{#1}{{\bf{{\ }(#1)}}}}  
\nc{\mref}[1]{\ref{#1}{{\bf{{\ }(#1)}}}}  
\nc{\meqref}[1]{\eqref{#1}{{\bf{{\ }(#1)}}}} 
\nc{\mbibitem}[1]{\bibitem[\bf #1]{#1}} 
}

\newtheorem{thm}{Theorem}[section]
\newtheorem{lem}[thm]{Lemma}
\newtheorem{cor}[thm]{Corollary}
\newtheorem{pro}[thm]{Proposition}
\newtheorem{conj}[thm]{Conjecture}
\theoremstyle{definition}
\newtheorem{defi}[thm]{Definition}
\newtheorem{ex}[thm]{Example}
\newtheorem{rmk}[thm]{Remark}
\newtheorem{pdef}[thm]{Proposition-Definition}
\newtheorem{condition}[thm]{Condition}

\renewcommand{\labelenumi}{{\rm(\alph{enumi})}}
\renewcommand{\theenumi}{\alph{enumi}}

\nc{\tred}[1]{\textcolor{red}{#1}}
\nc{\tblue}[1]{\textcolor{blue}{#1}}
\nc{\tgreen}[1]{\textcolor{green}{#1}}
\nc{\tpurple}[1]{\textcolor{purple}{#1}}
\nc{\btred}[1]{\textcolor{red}{\bf #1}}
\nc{\btblue}[1]{\textcolor{blue}{\bf #1}}
\nc{\btgreen}[1]{\textcolor{green}{\bf #1}}
\nc{\btpurple}[1]{\textcolor{purple}{\bf #1}}

\nc{\cm}[1]{\textcolor{red}{Chengming:#1}}
\nc{\li}[1]{\textcolor{blue}{#1}}
\nc{\lir}[1]{\textcolor{blue}{Li:#1}}
\nc{\dl}[1]{\textcolor{green}{Dilei:#1}}
\nc{\ldl}[1]{\textcolor{violet}{#1}}


\nc{\twovec}[2]{\left(\begin{array}{c} #1 \\ #2\end{array} \right )}
\nc{\threevec}[3]{\left(\begin{array}{c} #1 \\ #2 \\ #3 \end{array}\right )}
\nc{\twomatrix}[4]{\left(\begin{array}{cc} #1 & #2\\ #3 & #4 \end{array} \right)}
\nc{\threematrix}[9]{{\left(\begin{matrix} #1 & #2 & #3\\ #4 & #5 & #6 \\ #7 & #8 & #9 \end{matrix} \right)}}
\nc{\twodet}[4]{\left|\begin{array}{cc} #1 & #2\\ #3 & #4 \end{array} \right|}

\nc{\rk}{\mathrm{r}}
\newcommand{\g}{\mathfrak g}
\newcommand{\h}{\mathfrak h}
\newcommand{\pf}{\noindent{$Proof$.}\ }
\newcommand{\frkg}{\mathfrak g}
\newcommand{\frkh}{\mathfrak h}
\newcommand{\Id}{\rm{Id}}
\newcommand{\gl}{\mathfrak {gl}}
\newcommand{\ad}{\mathrm{ad}}
\newcommand{\add}{\frka\frkd}
\newcommand{\frka}{\mathfrak a}
\newcommand{\frkb}{\mathfrak b}
\newcommand{\frkc}{\mathfrak c}
\newcommand{\frkd}{\mathfrak d}
\newcommand {\comment}[1]{{\marginpar{*}\scriptsize\textbf{Comments:} #1}}

\nc{\tforall}{\text{ for all }}

\nc{\svec}[2]{{\tiny\left(\begin{matrix}#1\\
#2\end{matrix}\right)\,}}  
\nc{\ssvec}[2]{{\tiny\left(\begin{matrix}#1\\
#2\end{matrix}\right)\,}} 

\nc{\typeI}{local cocycle $3$-Lie bialgebra\xspace}
\nc{\typeIs}{local cocycle $3$-Lie bialgebras\xspace}
\nc{\typeII}{double construction $3$-Lie bialgebra\xspace}
\nc{\typeIIs}{double construction $3$-Lie bialgebras\xspace}

\nc{\bia}{{$\mathcal{P}$-bimodule ${\bf k}$-algebra}\xspace}
\nc{\bias}{{$\mathcal{P}$-bimodule ${\bf k}$-algebras}\xspace}

\nc{\rmi}{{\mathrm{I}}}
\nc{\rmii}{{\mathrm{II}}}
\nc{\rmiii}{{\mathrm{III}}}
\nc{\pr}{{\mathrm{pr}}}
\newcommand{\huaA}{\mathcal{A}}

\nc{\mcdot}{{}}

\nc{\OT}{constant $\theta$-}
\nc{\T}{$\theta$-}
\nc{\IT}{inverse $\theta$-}


\nc{\asi}{ASI\xspace}
\nc{\dualp}{transposed Poisson\xspace}
\nc{\Dualp}{Transposed Poisson\xspace}
\nc{\dualpop}{{\bf TPois}\xspace}
\nc{\ldualp}{derivation-transposed Poisson\xspace}

\nc{\spdualp}{sp-dual Poisson \xspace} \nc{\aybe}{AYBE\xspace}

\nc{\admset}{\{\pm x\}\cup K^{\times} x^{-1}}

\nc{\dualrep}{gives a dual representation\xspace}
\nc{\admt}{admissible to\xspace}

\nc{\ciri}{\circ_{\rm I}}
\nc{\cirii}{\circ_{\rm II}}
\nc{\ciriii}{\circ_{\rm III}}

\nc{\opa}{\cdot_A}
\nc{\opb}{\cdot_B}

\nc{\post}{positive type\xspace}
\nc{\negt}{negative type\xspace}
\nc{\invt}{inverse type\xspace}

\nc{\pll}{\beta}
\nc{\plc}{\epsilon}

\nc{\ass}{{\mathit{Ass}}}
\nc{\comm}{{\mathit{Comm}}}
\nc{\dend}{{\mathit{Dend}}}
\nc{\zinb}{{\mathit{Zinb}}}
\nc{\tdend}{{\mathit{TDend}}}
\nc{\prelie}{{\mathit{preLie}}}
\nc{\postlie}{{\mathit{PostLie}}}
\nc{\quado}{{\mathit{Quad}}}
\nc{\octo}{{\mathit{Octo}}}
\nc{\ldend}{{\mathit{ldend}}}
\nc{\lquad}{{\mathit{LQuad}}}

 \nc{\adec}{\check{;}} \nc{\aop}{\alpha}
\nc{\dftimes}{\widetilde{\otimes}} \nc{\dfl}{\succ} \nc{\dfr}{\prec}
\nc{\dfc}{\circ} \nc{\dfb}{\bullet} \nc{\dft}{\star}
\nc{\dfcf}{{\mathbf k}} \nc{\apr}{\ast} \nc{\spr}{\cdot}
\nc{\twopr}{\circ} \nc{\tspr}{\star} \nc{\sempr}{\ast}
\nc{\disp}[1]{\displaystyle{#1}}
\nc{\bin}[2]{ (_{\stackrel{\scs{#1}}{\scs{#2}}})}  
\nc{\binc}[2]{ \left (\!\! \begin{array}{c} \scs{#1}\\
    \scs{#2} \end{array}\!\! \right )}  
\nc{\bincc}[2]{  \left ( {\scs{#1} \atop
    \vspace{-.5cm}\scs{#2}} \right )}  
\nc{\sarray}[2]{\begin{array}{c}#1 \vspace{.1cm}\\ \hline
    \vspace{-.35cm} \\ #2 \end{array}}
\nc{\bs}{\bar{S}} \nc{\dcup}{\stackrel{\bullet}{\cup}}
\nc{\dbigcup}{\stackrel{\bullet}{\bigcup}} \nc{\etree}{\big |}
\nc{\la}{\longrightarrow} \nc{\fe}{\'{e}} \nc{\rar}{\rightarrow}
\nc{\dar}{\downarrow} \nc{\dap}[1]{\downarrow
\rlap{$\scriptstyle{#1}$}} \nc{\uap}[1]{\uparrow
\rlap{$\scriptstyle{#1}$}} \nc{\defeq}{\stackrel{\rm def}{=}}
\nc{\dis}[1]{\displaystyle{#1}} \nc{\dotcup}{\,
\displaystyle{\bigcup^\bullet}\ } \nc{\sdotcup}{\tiny{
\displaystyle{\bigcup^\bullet}\ }} \nc{\hcm}{\ \hat{,}\ }
\nc{\hcirc}{\hat{\circ}} \nc{\hts}{\hat{\shpr}}
\nc{\lts}{\stackrel{\leftarrow}{\shpr}}
\nc{\rts}{\stackrel{\rightarrow}{\shpr}} \nc{\lleft}{[}
\nc{\lright}{]} \nc{\uni}[1]{\tilde{#1}} \nc{\wor}[1]{\check{#1}}
\nc{\free}[1]{\bar{#1}} \nc{\den}[1]{\check{#1}} \nc{\lrpa}{\wr}
\nc{\curlyl}{\left \{ \begin{array}{c} {} \\ {} \end{array}
    \right .  \!\!\!\!\!\!\!}
\nc{\curlyr}{ \!\!\!\!\!\!\!
    \left . \begin{array}{c} {} \\ {} \end{array}
    \right \} }
\nc{\leaf}{\ell}       
\nc{\longmid}{\left | \begin{array}{c} {} \\ {} \end{array}
    \right . \!\!\!\!\!\!\!}
\nc{\ot}{\otimes} \nc{\sot}{{\scriptstyle{\ot}}}
\nc{\otm}{\overline{\ot}}
\nc{\ora}[1]{\stackrel{#1}{\rar}}
\nc{\ola}[1]{\stackrel{#1}{\la}}
\nc{\pltree}{\calt^\pl}
\nc{\epltree}{\calt^{\pl,\NC}}
\nc{\rbpltree}{\calt^r}
\nc{\scs}[1]{\scriptstyle{#1}} \nc{\mrm}[1]{{\rm #1}}
\nc{\dirlim}{\displaystyle{\lim_{\longrightarrow}}\,}
\nc{\invlim}{\displaystyle{\lim_{\longleftarrow}}\,}
\nc{\mvp}{\vspace{0.5cm}} \nc{\svp}{\vspace{2cm}}
\nc{\vp}{\vspace{8cm}} \nc{\proofbegin}{\noindent{\bf Proof: }}
\nc{\proofend}{$\blacksquare$ \vspace{0.5cm}}
\nc{\freerbpl}{{F^{\mathrm RBPL}}}
\nc{\sha}{{\mbox{\cyr X}}}  
\nc{\ncsha}{{\mbox{\cyr X}^{\mathrm NC}}} \nc{\ncshao}{{\mbox{\cyr
X}^{\mathrm NC,\,0}}}
\nc{\shpr}{\diamond}    
\nc{\shprm}{\overline{\diamond}}    
\nc{\shpro}{\diamond^0}    
\nc{\shprr}{\diamond^r}     
\nc{\shpra}{\overline{\diamond}^r}
\nc{\shpru}{\check{\diamond}} \nc{\catpr}{\diamond_l}
\nc{\rcatpr}{\diamond_r} \nc{\lapr}{\diamond_a}
\nc{\sqcupm}{\ot}
\nc{\lepr}{\diamond_e} \nc{\vep}{\varepsilon} \nc{\labs}{\mid\!}
\nc{\rabs}{\!\mid} \nc{\hsha}{\widehat{\sha}}
\nc{\lsha}{\stackrel{\leftarrow}{\sha}}
\nc{\rsha}{\stackrel{\rightarrow}{\sha}} \nc{\lc}{\lfloor}
\nc{\rc}{\rfloor}
\nc{\tpr}{\sqcup}
\nc{\nctpr}{\vee}
\nc{\plpr}{\star}
\nc{\rbplpr}{\bar{\plpr}}
\nc{\sqmon}[1]{\langle #1\rangle}
\nc{\forest}{\calf}
\nc{\altx}{\Lambda_X} \nc{\vecT}{\vec{T}} \nc{\onetree}{\bullet}
\nc{\Ao}{\check{A}}
\nc{\seta}{\underline{\Ao}}
\nc{\deltaa}{\overline{\delta}}
\nc{\trho}{\tilde{\rho}}

\nc{\rpr}{\circ}
\nc{\dpr}{{\tiny\diamond}}
\nc{\rprpm}{{\rpr}}

\nc{\mmbox}[1]{\mbox{\ #1\ }} \nc{\ann}{\mrm{ann}}
\nc{\Aut}{\mrm{Aut}} \nc{\can}{\mrm{can}}
\nc{\twoalg}{{two-sided algebra}\xspace}
\nc{\colim}{\mrm{colim}}
\nc{\Cont}{\mrm{Cont}} \nc{\rchar}{\mrm{char}}
\nc{\cok}{\mrm{coker}} \nc{\dtf}{{R-{\rm tf}}} \nc{\dtor}{{R-{\rm
tor}}}
\renewcommand{\det}{\mrm{det}}
\nc{\depth}{{\mrm d}}
\nc{\Div}{{\mrm Div}} \nc{\End}{\mrm{End}} \nc{\Ext}{\mrm{Ext}}
\nc{\Fil}{\mrm{Fil}} \nc{\Frob}{\mrm{Frob}} \nc{\Gal}{\mrm{Gal}}
\nc{\GL}{\mrm{GL}} \nc{\Hom}{\mrm{Hom}} \nc{\hsr}{\mrm{H}}
\nc{\hpol}{\mrm{HP}} \nc{\id}{\mrm{id}} \nc{\im}{\mrm{im}}
\nc{\incl}{\mrm{incl}} \nc{\length}{\mrm{length}}
\nc{\LR}{\mrm{LR}} \nc{\mchar}{\rm char} \nc{\NC}{\mrm{NC}}
\nc{\mpart}{\mrm{part}} \nc{\pl}{\mrm{PL}}
\nc{\ql}{{\QQ_\ell}} \nc{\qp}{{\QQ_p}}
\nc{\rank}{\mrm{rank}} \nc{\rba}{\rm{RBA }} \nc{\rbas}{\rm{RBAs }}
\nc{\rbpl}{\mrm{RBPL}}
\nc{\rbw}{\rm{RBW }} \nc{\rbws}{\rm{RBWs }} \nc{\rcot}{\mrm{cot}}
\nc{\rest}{\rm{controlled}\xspace}
\nc{\rdef}{\mrm{def}} \nc{\rdiv}{{\rm div}} \nc{\rtf}{{\rm tf}}
\nc{\rtor}{{\rm tor}} \nc{\res}{\mrm{res}} \nc{\SL}{\mrm{SL}}
\nc{\Spec}{\mrm{Spec}} \nc{\tor}{\mrm{tor}} \nc{\Tr}{\mrm{Tr}}
\nc{\mtr}{\mrm{sk}}

\nc{\ab}{\mathbf{Ab}} \nc{\Alg}{\mathbf{Alg}}
\nc{\Algo}{\mathbf{Alg}^0} \nc{\Bax}{\mathbf{Bax}}
\nc{\Baxo}{\mathbf{Bax}^0} \nc{\RB}{\mathbf{RB}}
\nc{\RBo}{\mathbf{RB}^0} \nc{\BRB}{\mathbf{RB}}
\nc{\Dend}{\mathbf{DD}} \nc{\bfk}{{K}} \nc{\bfone}{{\bf 1}}
\nc{\base}[1]{{a_{#1}}} \nc{\detail}{\marginpar{\bf More detail}
    \noindent{\bf Need more detail!}
    \svp}
\nc{\Diff}{\mathbf{Diff}} \nc{\gap}{\marginpar{\bf
Incomplete}\noindent{\bf Incomplete!!}
    \svp}
\nc{\FMod}{\mathbf{FMod}} \nc{\mset}{\mathbf{MSet}}
\nc{\rb}{\mathrm{RB}} \nc{\Int}{\mathbf{Int}}
\nc{\Mon}{\mathbf{Mon}}
\nc{\remarks}{\noindent{\bf Remarks: }}
\nc{\OS}{\mathbf{OS}} 
\nc{\Rep}{\mathbf{Rep}}
\nc{\Rings}{\mathbf{Rings}} \nc{\Sets}{\mathbf{Sets}}
\nc{\DT}{\mathbf{DT}}

\nc{\BA}{{\mathbb A}} \nc{\CC}{{\mathbb C}} \nc{\DD}{{\mathbb D}}
\nc{\EE}{{\mathbb E}} \nc{\FF}{{\mathbb F}} \nc{\GG}{{\mathbb G}}
\nc{\HH}{{\mathbb H}} \nc{\LL}{{\mathbb L}} \nc{\NN}{{\mathbb N}}
\nc{\QQ}{{\mathbb Q}} \nc{\RR}{{\mathbb R}} \nc{\BS}{{\mathbb{S}}} \nc{\TT}{{\mathbb T}}
\nc{\VV}{{\mathbb V}} \nc{\ZZ}{{\mathbb Z}}


\nc{\calao}{{\mathcal A}} \nc{\cala}{{\mathcal A}}
\nc{\calc}{{\mathcal C}} \nc{\cald}{{\mathcal D}}
\nc{\cale}{{\mathcal E}} \nc{\calf}{{\mathcal F}}
\nc{\calfr}{{{\mathcal F}^{\,r}}} \nc{\calfo}{{\mathcal F}^0}
\nc{\calfro}{{\mathcal F}^{\,r,0}} \nc{\oF}{\overline{F}}
\nc{\calg}{{\mathcal G}} \nc{\calh}{{\mathcal H}}
\nc{\cali}{{\mathcal I}} \nc{\calj}{{\mathcal J}}
\nc{\call}{{\mathcal L}} \nc{\calm}{{\mathcal M}}
\nc{\caln}{{\mathcal N}} \nc{\calo}{{\mathcal O}}
\nc{\calp}{{\mathcal P}} \nc{\calq}{{\mathcal Q}} \nc{\calr}{{\mathcal R}}
\nc{\calt}{{\mathcal T}} \nc{\caltr}{{\mathcal T}^{\,r}}
\nc{\calu}{{\mathcal U}} \nc{\calv}{{\mathcal V}}
\nc{\calw}{{\mathcal W}} \nc{\calx}{{\mathcal X}}
\nc{\CA}{\mathcal{A}}

\nc{\vsa}{\vspace{-.1cm}}
\nc{\vsb}{\vspace{-.2cm}}
\nc{\vsc}{\vspace{-.3cm}}
\nc{\vsd}{\vspace{-.4cm}}
\nc{\vse}{\vspace{-.5cm}}

\nc{\fraka}{{\mathfrak a}} \nc{\frakB}{{\mathfrak B}}
\nc{\frakb}{{\mathfrak b}} \nc{\frakd}{{\mathfrak d}}
\nc{\oD}{\overline{D}}
\nc{\frakF}{{\mathfrak F}} \nc{\frakg}{{\mathfrak g}}
\nc{\frakm}{{\mathfrak m}} \nc{\frakM}{{\mathfrak M}}
\nc{\frakMo}{{\mathfrak M}^0} \nc{\frakp}{{\mathfrak p}}
\nc{\frakS}{{\mathfrak S}} \nc{\frakSo}{{\mathfrak S}^0}
\nc{\fraks}{{\mathfrak s}} \nc{\os}{\overline{\fraks}}
\nc{\frakT}{{\mathfrak T}}
\nc{\oT}{\overline{T}}
\nc{\frakX}{{\mathfrak X}} \nc{\frakXo}{{\mathfrak X}^0}
\nc{\frakx}{{\mathbf x}}
\nc{\frakTx}{\frakT}      
\nc{\frakTa}{\frakT^a}        
\nc{\frakTxo}{\frakTx^0}   
\nc{\caltao}{\calt^{a,0}}   
\nc{\ox}{\overline{\frakx}} \nc{\fraky}{{\mathfrak y}}
\nc{\frakz}{{\mathfrak z}} \nc{\oX}{\overline{X}}

\font\cyr=wncyr10

\nc{\al}{\alpha}
\nc{\lam}{\lambda}
\nc{\lr}{\longrightarrow}

\nc{\dpdop}{dual p-$\mathcal{O}$-operator\xspace}
\nc{\dpdops}{dual p-$\mathcal{O}$-operators\xspace}
\nc{\pdsCYBE}{pp-post-classical Yang-Baxter equation\xspace}
\nc{\PCYBE}{PPP-CYBE\xspace}

\nc{\gphpl}{\mysymbolcphp post-Lie\xspace}


\title[Bialgebra of post-Lie algebras via Manin triples and generalized Hessian Lie groups]{
A bialgebra theory of post-Lie algebras via Manin triples and
generalized Hessian Lie groups}

\author{Dilei Lu}
\address{Chern Institute of Mathematics \& LPMC, Nankai University, Tianjin 300071, China}
\email{ludyray@126.com}

\author{Chengming Bai}
\address{Chern Institute of Mathematics \& LPMC, Nankai University, Tianjin 300071, China}
\email{baicm@nankai.edu.cn}

\author{Li Guo}
\address{Department of Mathematics and Computer Science, Rutgers University, Newark, NJ 07102, USA}
\email{liguo@rutgers.edu}


\begin{abstract}
We develop a bialgebra theory of post-Lie algebras that can be
characterized by Manin triples of post-Lie algebras
associated to a bilinear form satisfying certain invariant
conditions. In the absence of dual representations for adjoint representations of post-Lie algebras, we utilize the geometric interpretation of post-Lie algebras to find the desired invariant condition, by generalizing pseudo-Hessian Lie groups to allow constant torsion for the flat connection.
The resulting notion is a generalized pseudo-Hessian post-Lie algebra, which is a post-Lie algebra equipped with a nondegenerate symmetric invariant bilinear form.
Moreover, generalized pseudo-Hessian post-Lie
algebras are also naturally obtained from quadratic Rota-Baxter
Lie algebras of weight one.
On the other hand, the notion of partial-pre-post-Lie algebra (pp-post-Lie algebras) is introduced as the algebraic structure
underlying generalized pseudo-Hessian post-Lie algebras, by splitting one of the two binary operations of post-Lie algebras. The notion of pp-post-Lie bialgebras is introduced as the equivalent structure of
Manin triples of post-Lie algebras associated to a nondegenerate symmetric invariant bilinear form, thereby establishing a bialgebra theory for post-Lie algebras via the Manin triple approach. We also study the related analogs of the classical Yang-Baxter equation, $\mathcal O$-operators and successors for pp-post-Lie algebras.
In particular, there is a construction of pp-post-Lie bialgebras from
the successors of pp-post-Lie algebras.
\end{abstract}

\subjclass[2020]{\!\!\!\!
17D25,  
17B62,  
22E60,  
17B38,  
58D17,	
53C05,  
16T10,	
17A30,  
17A36 
}

\keywords{Post-Lie algebra, Hessian Lie group, Manin triple, bialgebra, classical Yang-Baxter equation, $\mathcal{O}$-operator, pre-Lie algebra}

\maketitle

\vspace{-1.3cm}

\tableofcontents

\allowdisplaybreaks

\section{Introduction}
This paper develops a bialgebra theory from Manin triples of
post-Lie algebras, accomplished by an invariant bilinear form from
generalized pseudo-Hessian Lie groups and \mysymbol algebras
from partial splitting of operations of post-Lie algebras.

\subsection{Bialgebra theories}
A bialgebra structure consists of an algebra structure and a
coalgebra structure coupled by certain compatibility conditions.
In the case of associative algebras, one of the common
compatibility conditions is that the comultiplication is an
algebra homomorphism, giving the usual notion of an
(associative) bialgebra and a Hopf algebra~\mcite{Sw}; another one
is that the comultiplication behaves like a derivation, leading to
the infinitesimal bialgebra of Joni and Rota~\mcite{JR}. In the
case of Lie algebras, the latter compatibility gives the Lie
bialgebra of Drinfeld~\mcite{CP,Dr}.
In addition to its
broad applications~\mcite{CP,Dr,Dr1}, Lie bialgebra has shown its
significance by the close connections with the important notions
of Manin triples of Lie algebras, the classical Yang-Baxter
equation (CYBE), also called the classical $r$-matrices,
$\calo$-operators (that is, relative Rota-Baxter operators) on Lie algebras, and pre-Lie algebras~\mcite{B3,CP,Ku1,STS}.
These connections are depicted in the diagram
\begin{equation}
        \xymatrix{
            \text{\small pre-Lie}\atop\text{\small  algebras} \ar[r]     & \mathcal{O}\text{\small -operators on}\atop\text{\small Lie algebras}\ar[r] &
            \text{\small solutions of}\atop \text{\small CYBE} \ar[r]  & \text{\small Lie}\atop \text{\small  bialgebras}  \ar@{<->}[r] &
            \text{\small Manin triples of} \atop \text{\small Lie algebras} }
    \mlabel{eq:bigliediag}
\end{equation}
Especially, the Lie bialgebra is characterized by a Manin triple
of Lie algebras associated to a bilinear form satisfying an
invariant condition.

After the establishment of a theory for the infinitesimal bialgebra along the line of \meqref{eq:bigliediag}~\mcite{A1,A2,B0}, a similar approach has been taken to establish the bialgebra theories for quite a few other algebraic structures over the past few years, such as pre-Lie algebras \mcite{B1}, dendriform algebras \mcite{B0}, Leibniz
algebras \mcite{TS}, L-dendriform algebras~\mcite{BHC}, perm algebras \mcite{BYZ}, 3-Lie algebras \mcite{BGS} and Rota-Baxter associative and Lie algebras~\mcite{BGLM,BGM}.

Such a bialgebra theory, once established, is not only rich in
theory with related structures similar to the diagram in
\eqref{eq:bigliediag}. It often plays important roles in many
areas and have connections with other structures arising from
mathematics and physics. For example, Lie bialgebras are the infinitesimal versions of Poisson-Lie groups, and
are crucial in the study of quantum groups
\mcite{CP,Dr}. Pre-Lie bialgebras correspond
to the underlying algebraic structure of para-K\"ahler Lie groups
\mcite{B1}. Novikov bialgebras are used
to construct infinite-dimensional Lie bialgebras \mcite{HBG}.

\subsection{The Manin triple approach}
Other than the associative bialgebra, the rest of the above bialgebras are obtained by the so-called Manin triple approach.
Roughly speaking, for a given algebraic structure, to establish its bialgebra theory with this approach, one constructs triples $(A,A^*,A\oplus A^*)$ with each component equipped with the given algebraic structure and with a bilinear form $\mathcal{B}$ defined on $A\oplus A^*$ satisfying certain ``invariant" conditions. Thus, 
\begin{enumerate}
    \item the algebraic structure on $A$ gives the algebraic component of the desired bialgebra theory,
    \item the algebraic structure on $A^*$, by duality, gives the coalgebra component of the desired bialgebra theory, and most essentially
    \item the interaction of $A$ and $A^*$ in $A\oplus A^*$ through the bilinear form $\mathcal{B}$ gives the compatibility conditions of the desired bialgebra theory.
\end{enumerate}

However, despite its general name, there is no uniform procedure to follow in the Manin triple approach. The critical issue is finding the invariant condition for a bilinear form $\mathcal{B}$ on $A\oplus A^*$ and this has to be done on a case-by-case basis.
The following remarks illustrate the challenges that can be encountered in this approach.

\begin{enumerate}
\item Usually, if there is a
suitable notion of dual representations for a given algebraic
structure, then it can be used to obtain an invariant bilinear
form. This important property of an algebraic structure is
extracted as the {\it properness} in the general discussion
in~\mcite{Ku}. In the absence of the properness property, finding a bilinear form with a suitable invariant
condition becomes difficult. This is the case for transpose Poisson
algebras~\mcite{LB}.
\item Even with a suitable invariant condition of the bilinear
form in place for the Manin triples of a given algebraic
structure, the resulting bialgebra might not be realizable within
the original algebraic structure. For example, a Manin triple of
Lie algebras associated to a (skew-symmetric) symplectic form
gives rise to a pre-Lie
bialgebra~\mcite{B1}, and a Manin triple of pre-Lie algebras
associated to a nondegenerate symmetric 2-cocycle
defines the bialgebra theory of the
L-dendriform algebra~\mcite{NB}.
We note for later relevance that the L-dendriform algebra is obtained from splitting the operation of the pre-Lie algebra.
In the context of the
general procedure of splitting an operad~\mcite{BBGN}, the L-dendriform algebra
is the successor of the pre-Lie algebra, which in turn is the
successor of the Lie algebra.
\end{enumerate}
\vsc

\subsection{The Manin triple approach to a bialgebra theory for post-Lie algebras}

The notion of post-Lie algebras was introduced  by Vallette in
connection with homology of partition posets and the Koszul duality of operads \mcite{Va}. They attract more attentions
in the study of numerical integration on Lie groups and manifolds
\mcite{MS,MW}. In a differential geometric context, a post-Lie
algebra  appears naturally as the algebraic structure of the flat
connection with constant torsion  on a Lie group \mcite{ML}.
Moreover, post-Lie algebras also showed up in numerous other topics in pure and applied mathematics such as
modified classical Yang-Baxter equation \mcite{BGN},
Hopf algebras and factorization theorems  \mcite{EMM}, Rota-Baxter
operators \mcite{GLBJ}, Poincar\'e-Birkhoff-Witt theorems
\mcite{Do} and  regularity structures in stochastic
partial differential equations \mcite{BK}. 

The goal of the present paper is to develop a bialgebra theory for post-Lie algebras via the Main triple approach.
Here we will need to resolve both of the challenges mentioned above.

\subsubsection{Finding the invariant bilinear form from generalized pseudo-Hessian Lie groups}
First we need to find a suitable invariant bilinear form in order to define Manin triples of post-Lie algebras.
Unfortunately, no desired dual representations have been found for the
adjoint representation of a post-Lie algebra.

Instead we obtained the needed invariant bilinear form from the
(pseudo-)metrics in Riemannian
geometry, by generalizing pseudo-Hessian Lie groups to allow the
torsion of the connection to be constant instead of zero,
motivated by the aforementioned Riemannian geometric
interpretation of post-Lie algebras and certain geometric
structures. Explicitly, on the one hand, a post-Lie algebra
corresponds to a Lie group with a left-invariant connection which
is flat and has constant torsion. On the other hand, a
pseudo-Hessian Lie group is a Lie group with a flat and
torsion-free left-invariant connection such that it is compatible
with a left-invariant pseudo-Riemannian metric in the sense that
the  Codazzi equation is satisfied. Combining the two
structures in a natural way leads us to the notion of a
generalized pseudo-Hessian Lie group.
As the algebraic structure corresponding to a generalized
pseudo-Hessian Lie group, the notion of a \mysymbolcphp post-Lie
algebra is introduced, which is a post-Lie algebra equipped with a
nondegenerate symmetric bilinear form satisfying an invariant
condition.
Quite
remarkably, this invariant condition is exactly what we need,
corresponding to the Codazzi equation and the invariant condition
associated to the Lie bracket decided by the constant torsion. It
is also exactly the condition satisfied by the flat left-invariant
connection and the left-invariant pseudo-Riemannian metric on the
generalized pseudo-Hessian Lie group.

There is another significance of \mysymbolcphp post-Lie
algebras. It is known that a Rota-Baxter Lie algebra of weight one
induces a post-Lie algebra \mcite{BGN}. If the Rota-Baxter Lie algebra is quadratic in the sense that the Lie algebra has a nondegenerate symmetric
invariant bilinear form, then the induced post-Lie algebra is also \mysymbolcphp in the above sense.
\vsb

\subsubsection{Finding the underlying algebraic structure for the bialgebra by partially splitting}
The second challenge is also present in our development of a
bialgebra theory of post-Lie algebras. More precisely, the
bialgebra structure arising from Manin triples of post-Lie algebras
associated to the above-found invariant bilinear form cannot be realized in terms of post-Lie algebras. To find the right structure, we draw inspiration from the fact that the Manin triple of pre-Lie algebras associated
to a nondegenerate symmetric 2-cocycle gives rise to the bialgebra theory of L-dendriform algebras~\mcite{NB}, which is the splitting structure (successor) of pre-Lie algebras~\mcite{BBGN}.

Thus we speculated that Manin triples of post-Lie algebras should give the bialgebra theory of a splitting structure (successor) of post-Lie algebras, arriving at the notion of partial-pre-post-Lie algebras (\mysymbol algebras). An interesting twist in this search is that \mysymbol algebras present a new kind of
splitting of operations: rather than splitting both operations of the post-Lie algebra $(A, \circ,[-,-])$ as is usually the case~\mcite{BBGN}, the \mysymbol algebra is obtained from only splitting the operation $\circ$ into the
sum of two operations $\rhd$ and $\lhd$, without changing the Lie bracket $[-,-]$. Furthermore, the
characterization of \mysymbol algebras is given in terms of
post-Lie algebras with the specific representations on the dual
spaces instead of the representations on the underlying vector
spaces of the post-Lie algebras themselves.

\vsb

\subsubsection{\Mysymbol bialgebras, their CYBE and pre-structures}

With both of the challenges under control, we are able to define the notion of \mysymbol bialgebras that can be equivalently characterized by Manin triples of post-Lie
algebras associated to  the invariant bilinear forms.
We also give an analog of the theory for
classical Yang-Baxter equation (CYBE) in Lie algebras to construct
\mysymbol bialgebras. Explicitly, the study of a special class of
\mysymbol bialgebras leads to the introduction of the \pdsCYBE
(\PCYBE) in a  \mysymbol algebra whose antisymmetric solutions
give \mysymbol bialgebras. The \PCYBE is interpreted in terms of
operator forms and hence the notions of $\mathcal{O}$-operators on
\mysymbol algebras and pre-\mysymbol algebras are introduced to
provide antisymmetric solutions of the \PCYBE in certain bigger
\mysymbol algebras. These connections can be illustrated by the following
diagram in analog to \eqref{eq:bigliediag}.
\vsc
\begin{small}
    \begin{equation*}
\begin{gathered}
    \xymatrix@C=1.1cm{
  \txt{pre-\mysymbol\\  algebras}\ar@{->}[r]  & \txt{$\mathcal{O}$-operators on\\ \mysymbol \\ algebras}\ar@{->}[r]  &  \txt{antisymmetric \\ solutions of  \\ the  \PCYBE }   \ar[r]     & \txt{\mysymbol\\ bialgebras} \
  \ar@{<->}[r] & \txt{Manin triples \\of  post-Lie\\ algebras}
}
\end{gathered}
\end{equation*}
 \end{small}
In particular, there is a construction of \mysymbol bialgebras
from pre-\mysymbol algebras. Also note that the operad of
pre-\mysymbol algebras is the successor of the operad of \mysymbol
algebras.

\vsb
\subsection{Outline of the paper}
The paper is organized as follows.

In Section \mref{s:quadpl}, after recalling the geometric
interpretation of post-Lie algebras, we introduce the notion of a
generalized (pseudo)-Hessian Lie group, by generalizing the
torsion-free condition for the flat connection for a
(pseudo)-Hessian Lie group to the constant torsion condition
(Definition~\mref{d:ghlg}). As its corresponding algebraic
structure, the notion of a \mysymbolcphp post-Lie algebra is
introduced (Definition~\mref{d:gphpl}). Finally, refining the
construction of post-Lie algebras from Rota-Baxter Lie algebras of
weight one, there is a natural construction of \mysymbolcphp
post-Lie algebras  from Rota-Baxter Lie algebras of weight one
with a nondegenerate symmetric invariant bilinear form (Proposition~\mref{iblf}).

In Section \mref{s:ldla},  we first introduce the notion of partial-pre-post-Lie algebras (\mysymbol algebras) (Definition~\mref{ldldef}) and then give some
basic properties. They are characterized in terms of
representations of post-Lie algebras, illustrating a ``partial"
splitting of operations of post-Lie algebras
(Proposition~\ref{dpsplrep}). Such a
characterization is further interpreted in terms of an analog of
$\mathcal O$-operators, namely, \dpdops (Theorem~\mref{thm:dual
2}). As a consequence,
 \mysymbol algebras  are obtained from   \mysymbolcphp post-Lie algebras (Proposition~\mref{compaDPSPL1}).

In Section \mref{s:ldlb},
we introduce the notions of a Manin triple of post-Lie algebras
associated to the \mysymbolnew bilinear form
(Definition~\mref{d:mtpl}) and a \mysymbol bialgebra
(Definition~\mref{d:ppplbialg}). The equivalence between these two
notions is interpreted in terms of certain matched pairs of
post-Lie algebras. The study of a special class of \mysymbol
bialgebras leads to the introduction of the \PCYBE in a  \mysymbol
algebra whose antisymmetric solutions give \mysymbol bialgebras
(Corollary~\mref{iartoldlbia}). The \PCYBE is given an operator
form in terms of $\mathcal{O}$-operators on \mysymbol algebras
(Corollary \mref{LDCYBEO}), and the construction of skew-symmetric
solutions of \PCYBE is given in terms of $\mathcal{O}$-operators
on \mysymbol algebras (Theorem~\mref{oggs}). Finally,
pre-\mysymbol algebras are introduced to provide antisymmetric
solutions of the \PCYBE in a bigger \mysymbol algebra
(Corollary~\mref{pretoaldlcYBE}).

\smallskip
\noindent
{\bf Notations}. Unless otherwise specified, all the vector spaces
and algebras are finite-dimensional over a field  $\mathbb{F}$  of
characteristic $0$, although many results and notions remain
valid in the infinite-dimensional case. For a vector space $V$,
let
\vsa
$$\tau:V\otimes V\rightarrow V\otimes V,\quad u\otimes v\mapsto v\otimes u,\;\;\;u,v\in V,
\vsb
$$
be the flip operator. Let $A$ be a vector space with a binary
operation $\circ$. Define linear maps $L_{\circ},
R_{\circ}:A\rightarrow {\rm End}_{\mathbb F}(A)$ respectively by
\vsa
\begin{eqnarray*}
L_{\circ}(a)b:=a\circ b,\;\; R_{\circ}(a)b:=b\circ a, \;\;\;a,
b\in A.
\vsb
\end{eqnarray*}
In particular, when $(A,\circ:=[-,-])$ is a Lie algebra, we use the
usual notion of the adjoint operator $\ad(a)(b):=[a,b]$ for all
$a,b\in A$.
\vsc
\section{Generalized pseudo-Hessian  post-Lie algebras} \mlabel{s:quadpl}

We begin by recalling the notion of post-Lie algebras and their geometric interpretation. We then introduce the notion of a generalized (pseudo-)Hessian Lie group
as a generalization of a (pseudo-)Hessian Lie group to give the \mysymbolcphp post-Lie algebra as the corresponding algebraic structure. Finally, the construction of post-Lie algebras from Rota-Baxter Lie algebras of weight one is enriched to a natural construction of
\mysymbolcphp post-Lie algebras  from Rota-Baxter Lie algebras of
weight one with a nondegenerate symmetric invariant bilinear form.
\vsb

\subsection{Post-Lie algebras and their geometric interpretation}\

We first recall the notion of post-Lie algebras and some basic
facts.
\vsa
\begin{defi}\mcite{Va}\mlabel{def1}
A \textbf{post-Lie algebra} is a triple $(A,\circ,[-,-])$ where $(A,[-,-])$ is a Lie algebra  and  $\circ: A \otimes A
\rightarrow A$ is a binary operation satisfying the following
equalities:
\vsb
\begin{align}
x \circ[y, z] &= [x \circ y, z]+[y, x \circ z],\mlabel{pl1}\\
 (x \circ y- y \circ x +[x, y]) \circ z &= x \circ(y \circ z)-y \circ(x \circ z), \quad \forall x, y,z \in A.\mlabel{pl2}
\end{align}
\end{defi}
\vsd
\begin{rmk}\begin{enumerate}
\item Recall \mcite{Bu} that a {\bf pre-Lie algebra} is a pair
$(A,\circ)$ in which $A$ is a vector space and $\circ: A \otimes A
\rightarrow A$ is a binary operation satisfying the following
equation:
\vsb
\begin{align}
(x \circ y) \circ z -x \circ (y \circ z) = (y \circ x) \circ z - y
\circ (x \circ z),\;\;\forall x,y,z\in A.
\end{align}
Obviously,  $(A,\circ)$ is a pre-Lie algebra if and only if
$(A,\circ,[-,-]=0)$ is a post-Lie algebra.
 \item   Let $(\mathfrak{g}, [-,-])$ be a
Lie algebra. Then both $(\mathfrak{g},\circ=0,[-,-])$ and
$(\mathfrak{g}, [-,-]^{\mathrm{op}}, [-,-])$ are post-Lie
algebras, where $[x,y]^{\mathrm{op}} = [y,x]$ for all $x,y \in
\mathfrak{g}$.
\end{enumerate}
\end{rmk}

\begin{pro}\mlabel{newLie} \mcite{Va}
Let  $(A, \circ,[-,-])$  be a post-Lie algebra. The binary
operation
\vsa
\begin{equation}
\{x, y\}:=x \circ y-y \circ x+[x, y], \quad \forall x, y \in A,
\mlabel{eq5}
\vsb
\end{equation}
defines a Lie algebra $(\mathfrak{g}(A),\{-,-\})$, called
the \textbf{sub-adjacent Lie algebra} of $(A, \circ,[-,-])$, and
$(A, \circ,[-,-])$ is also called a  \textbf{compatible post-Lie
algebra} structure on the Lie algebra $(\mathfrak{g}(A),\{-,-\})$.
\end{pro}
\vsc
\begin{pro}\mlabel{pro:new}
\mcite{ML} Let  $(A, \circ,[-,-])$  be a post-Lie algebra. Then the
two binary operations
\vsa
\begin{equation}
x * y =x \circ y +[x, y],\quad [x,y]^{\mathrm{op}} = [y,x], \quad
\forall x, y \in A, \mlabel{oppostLie}
\vsb
\end{equation}
define a  post-Lie algebra $(A, * , [-,-]^{\mathrm{op}})$.
Moreover both  $(A, \circ,[-,-])$ and  $(A, * ,
[-,-]^{\mathrm{op}})$ have the same sub-adjacent Lie algebra
$(\mathfrak{g}(A),\{-,-\})$, where $\{-,-\}$ is defined by
Eq.~\eqref{eq5}.
\vsb
\end{pro}

Next we recall the geometric description of post-Lie algebras
\mcite{ML}.   Let $G$ be a   Lie group equipped with a
left-invariant affine connection $\nabla$.  Let $(\mathfrak{g},[-,-])$  be the associated Lie algebra of $G$, consisting of all left-invariant vector fields on $G$. Since  $\nabla$
  is left-invariant, for all $X, Y \in \mathfrak{g}$,  the covariant derivative   $\nabla_{X}Y$ of  $Y $
  in the direction of  $X$ is also left-invariant and thus defines an  $\mathbb{R}$-linear, non-associative binary operation  $\circ$  on  $\mathfrak{g}$
  by
\vsb
\begin{align}\mlabel{eq:pro}
  X \circ Y := \nabla_{X} Y.
\vsb
  \end{align}
The torsion  $\mathrm{T}$  is given by
\vsc
\begin{align}
\mathrm{T}(X, Y) :=X \circ  Y-Y \circ X-[ X, Y ], \;\; \forall
X,Y\in \mathfrak{g}, \mlabel{torsion}
\vsb
\end{align}
and it admits a covariant differential  $\nabla \mathrm{T}$.
Furthermore, the curvature  $\mathrm{R}$ is given by \vsb
\begin{align}
\mathrm{R}(X, Y) Z:  = X \circ(Y \circ Z)-Y \circ(X \circ Z)-[X,
Y] \circ Z,\;\; \forall X,Y,Z\in \mathfrak{g}. \mlabel{curvature}
\end{align}

When the connection $\nabla$ is flat and has constant torsion,
i.e.,  $\mathrm{R}=   0=\nabla\mathrm{T}$, then
$(\mathfrak{g}, \circ,-\mathrm{T}(-, -))$ is a  compatible
post-Lie algebra structure on $(\mathfrak{g},[-,-])$. Here $-\mathrm{T}(-, -)$ is a Lie bracket on $\mathfrak{g}$ because of the  antisymmetry of $\mathrm{T}$ and
the first Bianchi identity:
\vsb
\begin{align}
\mathfrak{S}\Big( \mathrm{T}(\mathrm{T}(X, Y), Z)+\left(\nabla_{X} {\rm
T}\right)(Y, Z)-\mathrm{R}(X, Y) Z \Big)=0, \;\;\forall X,Y,Z\in \mathfrak{g},
\end{align}
where   $\mathfrak{S}$  denotes the sum over the three cyclic permutations of  $(X, Y, Z)$.
Moreover, the flatness of $\nabla$ is equivalent to Eq.~(\mref{pl2}) as can be seen by inserting
Eq.~(\mref{torsion}) into the equation  $\mathrm{R}=0$:
\vsb
\begin{align*}
 (X \circ Y- Y \circ  X - \mathrm{T}(X, Y))\circ Z &=   [ X, Y ]\circ Z= X \circ (Y \circ Z)-Y \circ (X \circ
 Z), \;\;\forall X,Y,Z\in \mathfrak{g},
 \vse
\end{align*}
whereas  Eq.~(\mref{pl1}) follows from the definition of the covariant differential of  $\mathrm{T}$:
\vsb
\begin{align*}
0=(\nabla_X\mathrm{T})(Y, Z)=X \circ \mathrm{T}(Y,
Z)-\mathrm{T}(X \circ Y, Z)-\mathrm{T}(Y, X
\circ Z),  \;\;\forall X,Y,Z\in \mathfrak{g}.
\vse
\end{align*}

Moreover, by the bijection between left-invariant affine connections on a Lie group  and  binary operations on its Lie algebra \cite[page 70]{MMP},  there is a bijection  between  flat left-invariant affine connections  with constant torsion on a Lie group and   compatible post-Lie algebra structures on its Lie algebra.
In particular, $(\mathfrak{g}, \circ)$ is a pre-Lie
algebra when the  left-invariant affine connection
$\nabla$  is flat and  torsion-free, i.e.,  $\mathrm{R}=  0 =
\mathrm{T}$. Recall that giving an (left-invariant) {\bf  affine structure}  on a Lie group $G$
means equipping $G$ with a flat and torsion-free (left-invariant)
affine connection. Thus  there is a natural bijection between left-invariant affine  structures on a Lie
group  and  compatible pre-Lie algebra structures on its Lie algebra
\mcite{Bu}.

\vsc

\subsection{Generalized pseudo-Hessian Lie groups and \mysymbolcphp post-Lie algebras}\

Recall that a {\bf Hessian  (resp. pseudo-Hessian) Lie group } \mcite{CM,Sh1,Sh2} is a triple $(G,\nabla$, $\langle -,- \rangle)$ where $G$ is a connected Lie group, $\nabla$ is a flat and torsion-free left-invariant affine connection and $\langle-,-\rangle $ is a left-invariant  Riemannian (resp. pseudo-Riemannian) metric  on $G$, satisfying the Codazzi equation:
\vsc
\begin{align}
\langle \nabla_X Y, Z \rangle -  \langle  X, \nabla_Y Z \rangle = \langle \nabla_Y X, Z \rangle -  \langle  Y, \nabla_X  Z \rangle, \quad \forall X,Y,Z \in \mathfrak{g}. \mlabel{Ce}
\vsb
\end{align}

\newpage

We introduce the notion of a  \mysymbolcph Lie group whose
connection has constant torsion instead of the torsion-free
condition, extending the Hessian structures on Lie groups.
\begin{defi} \mlabel{d:ghlg}
A {\bf \mysymbolcph (resp. pseudo-Hessian)  Lie group } is a triple $(G,\nabla, \langle - , - \rangle)$ where $G$  is a connected Lie group,  $\nabla$ is a flat left-invariant affine  connection  with
constant torsion $\mathrm{T}$ and $\langle-,-\rangle $ is a left-invariant Riemannian (resp. pseudo-Riemannian)   metric
satisfying the  Codazzi equation \eqref{Ce} and
\vsb
     \begin{align}
 \langle \mathrm{T}(X,Y), Z \rangle =  \langle X, \mathrm{T}(Y,Z) \rangle, \quad \forall X,Y,Z \in \mathfrak{g}. \mlabel{cmcond}
 \end{align}
\end{defi}
\vsd

\begin{rmk}
When the connection is torsion-free, a generalized (pseudo-)Hessian  Lie group  reduces  to a (pseudo-)Hessian Lie group.
\end{rmk}
\vsb
The following conclusion is obvious.
\vsb
\begin{pro}
Let $(G,\nabla, \langle - , - \rangle)$ be a  \mysymbolcph Lie
group. Let $(\frak g, [-,-])$ be the Lie algebra of $G$. Define
$\circ$ and $T$ by Eqs.~{\rm (\mref{eq:pro})} and {\rm
(\mref{torsion})} respectively. Then
$(\mathfrak{g},\circ,-\mathrm{T}(-, -))$ is a post-Lie algebra and
$\langle - , - \rangle $ is an inner product on $\mathfrak{g}$
satisfying Eqs.~\eqref{Ce} and~\eqref{cmcond}.
\end{pro}
\vsb
This result motivates us to introduce the following notion.

\begin{defi}
A  bilinear form $\mathcal B$ on a post-Lie algebra $(A,\circ,
[-,-])$ is called {\bf invariant} if
\vsb
\begin{align}
    \mathcal{B}([x,y], z)    &= \mathcal{B}(x,[y,z]),  \mlabel{lieinv}\\
\mathcal{B}(x \circ y, z)  - \mathcal{B}( x, y \circ z) &=
\mathcal{B}(y \circ x, z)  - \mathcal{B}( y, x \circ z),\quad
\forall x,y,z \in A. \mlabel{2-cocycle}
\end{align}
A {\bf \mysymbolcphp post-Lie algebra} is a quadruple $(A,\circ,
[-,-],\mathcal{B})$ where $(A$, $\circ$, $[-,-])$  is a
post-Lie algebra and $\mathcal{B}$ is a nondegenerate symmetric
invariant bilinear form on $(A,\circ, [-,-])$.
\mlabel{d:gphpl}
\end{defi}
\vsd
\begin{rmk}
Recall the following two notions.

\begin{enumerate}
\item A bilinear form  $\mathcal{B}$  on a Lie algebra
$(\mathfrak{g}, [-,-])$  is called \textbf{invariant} if
$\mathcal{B}$   satisfies Eq.~\eqref{lieinv}.   A  Lie algebra
$(\mathfrak{g},[-,-])$ equipped with a  nondegenerate symmetric
invariant bilinear form $\mathcal{B}$ is called a {\bf quadratic
Lie algebra} and denoted by $(\mathfrak{g},[-,-],\mathcal{B})$.
\item A {\bf pseudo-Hessian pre-Lie algebra} \mcite{NB} is a
triple $(A,\circ,\mathcal{B})$  where $(A,\circ)$ is a pre-Lie
algebra and $\mathcal B$ is a   nondegenerate symmetric bilinear
form satisfying Eq.~(\mref{2-cocycle}). In fact, a bilinear form
$\mathcal B$ on a pre-Lie algebra $(A,\circ)$ satisfying
Eq.~(\mref{2-cocycle}) is called a {\bf 2-cocycle} on $(A,\circ)$ in the
sense of the cohomology theory of pre-Lie algebras \mcite{B1}.
\end{enumerate}
 Thus a \gphpl algebra $(A,\circ,
[-,-])$  reduces to a pseudo-Hessian pre-Lie algebra if $[-,-] =
0$, whereas it reduces to  a quadratic Lie algebra if $\circ =
0$.
\end{rmk}

It is well known that a connected Lie group admits a bi-invariant  Riemannian metric if and only if  its Lie algebra admits an  invariant inner product \mcite{Kn}.

\vsb
\begin{ex}\mlabel{ECpostLie}
Let $G$ be an $n$-dimensional connected Lie group $G$ and
$(\mathfrak{g},[-,-])$ be its Lie algebra. Let
$\{X_1,\cdots,X_n\}$ be a basis of $(\mathfrak{g},[-,-])$. A
left-invariant affine connection $\nabla$ on $G$ is therefore
uniquely determined by the $n^{3}$ Christoffel coefficients $\Gamma_{i
j}^{k}$  defined by
\vsc
\begin{align*}
\nabla_{X_{i}} X_{j}=\sum_{k=1}^{n} \Gamma_{i j}^{k} X_{k}, \quad
\forall X_{i} \in \mathfrak{g},\quad  1 \le i \le n.
\end{align*}
\vsb
This identifies $\nabla$ with a linear map
$\Lambda: \mathfrak{g} \rightarrow \mathfrak{g l}(\mathfrak{g})$,
under which the torsion and curvature become
$$
\mathrm{T}(X, Y)  =\Lambda_{X} Y-\Lambda_{Y} X-[X, Y], \quad
\mathrm{R}(X, Y) Z  =\left[\Lambda_{X}, \Lambda_{Y}\right]
Z-\Lambda_{[X, Y]} Z, \quad \forall  X,Y,Z \in \mathfrak{g}.
$$
\begin{enumerate}
\item  The case when $\Lambda=0$ corresponds to a flat
left-invariant affine connection $\nabla^-$ with constant torsion given by
$\mathrm{T}(X, Y)= -[X, Y]$  such that   $\nabla^-_XY=0$ for all
$X,Y \in \mathfrak{g}$, which is called the {\bf ($-$)-connection}
of $G$. It defines the trivial post-Lie algebra $(\mathfrak{g},
\circ = 0,[-,-])$.

\item  The case when $\Lambda = \ad$ gives
\begin{align*}
\mathrm{T}(X, Y) & =
[X, Y] - [Y, X]-[X, Y]= [X, Y], \\
\mathrm{R}(X, Y) Z 
&=  [X,[[Y,Z]] - [Y,[X,Z]] - [[X, Y],Z] = 0, \; \forall  X,Y,Z \in
\mathfrak{g}.
\end{align*}
The corresponding connection
$\nabla^+$, called the {\bf (+)-connection} of $G$, again has constant torsion and $\nabla^+_XY=[X, Y]$ for all $X,Y \in \mathfrak{g}$. It defines the
post-Lie algebra $(\mathfrak{g}, \circ =
[-,-],[-,-]^{\mathrm{op}})$.
\end{enumerate}

If there is a bi-invariant Riemannian metric $\langle-,-\rangle $
on $G$, then $\langle-,-\rangle$ is an invariant inner product on
$(\mathfrak{g},[-,-])$.  Thus $(G,\nabla^-, \langle - , -
\rangle)$ and $(G,\nabla^+, \langle - , - \rangle)$ are
\mysymbolcph Lie groups and $\langle-,-\rangle$ is an invariant
inner product on the post-Lie algebra $(\mathfrak{g}, 0,[-,-])$
and $(\mathfrak{g}, [-,-],[-,-]^{\mathrm{op}})$ respectively. Both
the ($-$) and ($+$)-connections  on Lie groups were first
introduced by Cartan and Schouten  \mcite{ECJA} and are  called
{\bf Cartan-Schouten connections}.  See \mcite{Ca,No,MMP} for
more details.
\end{ex}

\begin{pro}
Let $(A,\circ,[-,-])$ be a post-Lie algebra and $\mathcal{B}$ be
an \mysymbolnewi on $(A,\circ,[-,-])$. Define a bilinear form
$\omega$ by
\begin{align*}
\omega(x,y): = \mathcal{B}(x,y) - \mathcal{B}(y,x), \quad \forall x,y \in A.
\end{align*}
Then $\omega$ is a 2-cocycle of the sub-adjacent Lie algebra $(\mathfrak{g}(A),\{-,-\})$, that is, $\omega$ satisfies
\begin{align*}
 \omega(\{x,y\},z)+ \omega(\{y,z\},x)+\omega(\{z,x\},y) = 0, \quad \forall x,y,z \in A.
\end{align*}
\end{pro}
\begin{proof}
The $2$-cycle condition is verified directly applying
\meqref{eq5}.
\end{proof}

Moreover, from a \gphpl algebra, we  naturally get
another one.
\begin{pro}
Let $(A,\circ,[-,-],\mathcal{B})$ be a \gphpl algebra. Then $\mathcal{B}$ is   an \mysymbolnewi on the post-Lie algebra
 $(A,*,[-,-]^{\mathrm{op}}) $ given in Proposition \mref{pro:new}, such that $(A,*,[-,-]^{\mathrm{op}}, \mathcal{B}) $ is a  \gphpl algebra.
\end{pro}
\begin{proof}
For all $x,y,z \in A$, we have
\begin{eqnarray*}
&&\mathcal{B}(x * y, z)  - \mathcal{B}( x, y * z) - \mathcal{B}(y * x, z)  + \mathcal{B}( y, x * z)\\
&&\overset{\eqref{oppostLie}}{=}\mathcal{B}(x \circ y + [x,y], z)
- \mathcal{B}( x, y \circ z + [y,z]) - \mathcal{B}(y \circ x +
[y,x], z)  + \mathcal{B}( y, x \circ z + [x,z])= 0.
\end{eqnarray*}
This completes the proof.
\end{proof}

\subsection{Generalized pseudo-Hessian post-Lie algebras  from quadratic  Rota-Baxter   Lie  algebras of weight one}\

\begin{defi} Let $(\mathfrak{g},[-,-])$  be a Lie algebra. A {\bf Rota-Baxter operator of weight  $\lambda \in \mathbb{F}$}  on the  Lie algebra  $(\mathfrak{g},[-,-])$  is a linear map  $P: \mathfrak{g} \rightarrow \mathfrak{g}$  satisfying
\begin{align}
[P(x), P(y)]=P([P(x), y]+[x, P(y)]+\lambda[x, y]), \quad \forall x, y \in \mathfrak{g}. \mlabel{LieRB}
\end{align}
A Lie algebra  $(\mathfrak{g},[-,-])$ with a Rota-Baxter operator  $P$  of weight  $\lambda $ is called a \textbf{Rota-Baxter Lie algebra of weight  $\lambda$}  and denoted by  $(\mathfrak{g},[-,-], P)$.
\end{defi}

\begin{pro}\mcite{BGN}\mlabel{ipl1} Let  $(\mathfrak{g},[-,-], P)$  be a Rota-Baxter Lie algebra of weight one. Define
\begin{align}
x \circ y :=  [P(x), y], \quad \forall x, y \in \mathfrak{g}.\mlabel{ipl}
\end{align}
Then  $(\mathfrak{g}, \circ,[-,-])$  is a post-Lie algebra, called the \textbf{induced post-Lie algebra} from the Rota-Baxter Lie algebra  $(\mathfrak{g},[-,-], P)$.
\end{pro}

\begin{pro}\mlabel{iblf} Let  $(\mathfrak{g},[-,-], P)$ be a Rota-Baxter Lie algebra of weight one and  $(\mathfrak{g},\circ$, $[-,-])$  be the induced post-Lie algebra. If there is an invariant
bilinear form  $\mathcal{B}$  on the Lie algebra
$(\mathfrak{g},[-,-])$, then  $\mathcal{B}$  is \mysymbolnew  on
the post-Lie algebra  $(\mathfrak{g},\circ,[-,-])$. In particular,
if  $(\mathfrak{g},[-,-], P,\mathcal B)$ is a {\bf quadratic
Rota-Baxter Lie algebra of weight one} in the sense that
$(\mathfrak{g},[-,-], P)$ is Rota-Baxter Lie algebra of weight one
and $(\mathfrak g, [-,-],\mathcal B)$ is a quadratic Lie algebra,
then $(\mathfrak{g},\circ,[-,-],\mathcal B)$ is a \mysymbolcphp
post-Lie algebra.
\end{pro}

\begin{proof} Let  $x, y, z \in \mathfrak{g}$. Then we have
\begin{align}
\mathcal{B}(x \circ y, z) &= \mathcal{B}([P(x), y], z) = -\mathcal{B}(y,[P(x), z])=-\mathcal{B}(y, x \circ z).
\mlabel{left-invariant}
\end{align}
Then Eq.~(\mref{2-cocycle}) follows, giving the conclusion.
\end{proof}

\begin{rmk}
A bilinear form  $\mathcal{B}$  on a post-Lie algebra $(A,\!\circ,
[-,-])$  is called \textbf{left-invariant} if $\mathcal{B}$
satisfies Eqs.~\eqref{lieinv} and~\eqref{left-invariant}. In the
case of pre-Lie algebras, a left-invariant bilinear form is automatically a
2-cocycle. Moreover, there is a natural bijection
between the set of pre-Lie algebras with a nondegenerate symmetric
left-invariant bilinear form and the set of connected and simply
connected Lie groups with a left-invariant flat pseudo-metric
\mcite{BGLM,BHC,Mil}.
\end{rmk}

\begin{ex}\mlabel{splex}
Let  $\mathfrak{s l}(2, \mathbb{C})$ be the 3-dimensional complex simple Lie algebra   with a basis:
{\small
\begin{align*}
\left\{e_{1}=\frac{1}{2}\left(\begin{array}{cc}
0 & 1 \\
-1 & 0
\end{array}\right), \ e_{2}=\frac{1}{2 \mathrm{i}}\left(\begin{array}{ll}
0 & 1 \\
1 & 0
\end{array}\right), \ e_{3}=\frac{1}{2 \mathrm{i}}\left(\begin{array}{cc}
1 & 0 \\
0 & -1
\end{array}\right)\right\}, \quad \mathrm{i} = \sqrt{-1},
\end{align*}
}
whose products are given by
$[e_1, e_2]=e_3, [e_2, e_3]=e_1, [e_3, e_1]=e_2.
$
The Killing form $\kappa$ on  $\mathfrak{s l}(2, \mathbb{C})$ defined by
\begin{align*}
\kappa(e_i,e_j): = \mathrm{Tr}(\mathrm{ad}(e_i)\mathrm{ad}(e_j))
=-2\delta_{ij}, \quad   i,j = 1,2,3,
\end{align*}
is a  nondegenerate symmetric invariant bilinear form. The
linear operator $P$ on $\mathfrak{sl}(2,\mathbb{C})$ given by
\begin{align*}
P(e_1)=e_1,\quad P(e_2)=-\frac{1}{2} e_2 + \frac{\mathrm{i}}{2}  e_3,\quad  P(e_3)=-\frac{\mathrm{i}}{2} e_2 - \frac{1}{2}  e_3,
\end{align*}
is a Rota-Baxter operator of weight one~\mcite{PLBG}. By
Proposition \mref{iblf}, $\kappa$ is also invariant on the induced
post-Lie algebra  $(\mathfrak{s l}(2, \mathbb{C}), \circ,[-,-])$ with $\circ$ explicitly given by
{\small
\begin{eqnarray*} e_{1} \circ e_{1} &=& 0, \quad
e_{2} \circ e_{1}= \frac{\mathrm{i}}{2} e_{2}+\frac{1}{2} e_{3},
\quad e_{3} \circ e_{1}=-\frac{1}{2} e_{2}+\frac{\mathrm{i}}{2}
e_{3},\quad
e_{1} \circ e_{2} = e_{3},\\
 e_{2} \circ e_{2}&=&-\frac{\mathrm{i}}{2} e_{1}, \quad e_{3} \circ e_{2}=\frac{1}{2}
 e_{1},\quad
e_{1} \circ e_{3} = -e_{2}, \quad e_{2} \circ e_{3}=-\frac{1}{2}
e_{1}, \quad e_{3} \circ e_{3}=-\frac{\mathrm{i}}{2} e_{1}.
\end{eqnarray*}
}
Then  $(\mathfrak{s l}(2, \mathbb{C}), \circ,[-,-],\kappa)$ is a
\gphpl algebra.
\end{ex}

\begin{ex}Recall \mcite{GLS} that a Rota-Baxter operator $P: G \to G$ on a Lie group $G$  is a smooth map  satisfying
 \begin{align*}
 P(g_1)P(g_2) = P\Big(g_1 P(g_1) g_2 \big(P(g_1)\big)^{-1}\Big),\quad \forall g_1,g_2 \in G.
 \end{align*}
Then the map $P_{*e}$, which is the tangent  of $P$ at the
identity $e$, is a Rota-Baxter operator of weight one on the Lie
algebra $(\mathfrak{g},[-,-])$. Thus there is an induced
post-Lie algebra $(\mathfrak{g},\circ,[-,-])$ where the
operation $\circ$ is defined by Eq.~\eqref{ipl}.     Moreover, if
$G$ is a connected  Lie group equipped with a Rota-Baxter operator
$P$ and a bi-invariant pseudo-Riemannian metric
$\langle-,-\rangle$, then
$(\mathfrak{g},\circ,[-,-],\langle-,-\rangle)$ is a
\gphpl algebra  by Proposition \mref{iblf}.
 \end{ex}

\section{Partial-pre-post-Lie  algebras}
\mlabel{s:ldla}
We introduce the notion of partial-pre-post-Lie  algebras (\mysymbol algebras) and then study some
basic properties. They are characterized in terms of
representations of post-Lie algebras, illustrating a ``partial" splitting of operations of post-Lie algebras. Such a
characterization is further interpreted in terms of an analog of
$\mathcal O$-operators, namely, \dpdops. As a consequence,
 \mysymbol algebras  are obtained from   \gphpl algebras.

\vsb
\subsection{Partial-pre-post-Lie  algebras and post-Lie algebras}
\vsa
\begin{defi} \mlabel{ldldef}
A {\bf partial-pre-post-Lie  algebra (\mysymbol algebra)} is a
quadruple $(A$,  $\rhd$, $\lhd$, $[-,-])$, where  $(A,[-,-])$
is a Lie algebra  and $\rhd,\lhd:   A \otimes A \rightarrow  A$
are binary operations  satisfying
\begin{align}
 x \lhd  [y,z]  &=  [x,y]\lhd z +  [z,x] \lhd y, \mlabel{DPSPL1}\\
  [x,y \lhd  z + z \lhd  y]&=   [x,z] \lhd  y + y \lhd [x,z] = 0,\mlabel{DPSPL2}\\
 x \rhd [y,z]  - [y,z]  \lhd x &=  [x \rhd y + x \lhd y,z]+[y, x \rhd z - z \lhd x  ],\mlabel{DPSPL3}\\
x \rhd( y \lhd z) &= (x \rhd y - y \lhd x)  \lhd z+ y \lhd (x \rhd z + x \lhd z )  - [x,y \lhd z], \mlabel{DPSPL4}\\
\{x,y\}\rhd z &=    x \rhd ( y \rhd z  ) \!-\!y \rhd (x \rhd z )  +  [y,x \lhd z] -  [x, y \lhd z]-  [x,y]  \lhd z, \mlabel{DPSPL5}
\end{align}
where \vsb
\begin{equation}
\{x,y\}: = x \rhd y + x \lhd y  - y \rhd x - y \lhd x + [x,y], \quad x, y, z \in A.
\mlabel{eq:curbra}
\end{equation}
\end{defi}
\vsd

\begin{rmk} Let $(A,  \rhd,\lhd,[-,-])$ be a  \mysymbol algebra.
\begin{enumerate}
\item  If the Lie bracket $[-,-]$ is abelian,
then $(A, \rhd,\lhd)$ is an L-dendriform algebra. Here recall \mcite{BLN} that an {\bf L-dendriform algebra} is a
triple $(A, \rhd,\lhd)$, where $A$ is a vector space and
$\rhd,\lhd:   A \otimes A \rightarrow  A$ are binary operations
satisfying the following equations.
\begin{align}
  ( x \rhd y - y \lhd x) \lhd z  &= x \rhd (y \lhd z) - y \lhd (x \rhd z + x \lhd z),  \\
 (x \rhd y + x \lhd y - y \rhd x - y \lhd x) \rhd z &=    x \rhd ( y \rhd z  ) -y \rhd (x \rhd z ), \; \forall x, y, z \in A.
\end{align}

\item  If the operation $\lhd $ is trivial, then  $(A,
\rhd,[-,-])$ is a post-Lie algebra.
\end{enumerate}
\end{rmk}
\vsd

\begin{lem}\label{lem:equ}
Let  $(A,[-,-])$  be a Lie algebra and $\rhd,\lhd$ be binary
operations on  $A$. Define a binary operation $\{-,-\}$ by
Eq.~\eqref{eq:curbra} and a binary operation $\circ$ by
\begin{equation}
x \circ  y :=x \rhd y + x \lhd y, \quad \forall x,y\in A.
\mlabel{hpla}
\end{equation}
Then the following conclusions hold.
\begin{enumerate}
\item \label{it:t1} If Eq.~\eqref{DPSPL3} holds, then
Eq.~\eqref{DPSPL2} holds if and only if Eq.~\eqref{pl1} holds.

\item \label{it:t2} If Eq.~\eqref{DPSPL4} holds, then
Eq.~\eqref{DPSPL5} holds if and only if Eq.~\eqref{pl2} holds.
\end{enumerate}
\end{lem}

\begin{proof}
(\ref{it:t1}). Let $x,y,z \in A$. Suppose that Eq.~\eqref{DPSPL2}
holds. Then we have
\begin{align*}
 x \circ [y, z]  &\hspace{0.1cm}=   x \rhd [y,z]  + x \lhd [y,z]  \overset{\eqref{DPSPL2}}{=}   x \rhd [y,z]  - [y,z]  \lhd x   \\
 &\overset{\eqref{DPSPL3}}{=}   [x \rhd y + x \lhd y,z]+[y, x \rhd z - z \lhd x  ] \overset{\eqref{DPSPL2}}{=}   [x \circ y, z] +   [y, x \circ  z].
\end{align*}
Thus  Eq.~(\mref{pl1}) holds. Conversely, suppose that
Eq.~\eqref{pl1} holds. Then we have
\begin{eqnarray*}
 x \rhd [y,z]  +  x \lhd [y,z]  &\overset{\eqref{pl1}}{=}&  [x \rhd y + x \lhd y,z]+[y, x \rhd z +  x \lhd  z ]\\
 &\overset{\eqref{DPSPL3}}{=}& x \rhd [y,z]  - [y,z]  \lhd x -[y, x \rhd z - z \lhd x  ]+[y, x \rhd z +  x \lhd  z ]\\
  &=& x \rhd [y,z]  - [y,z]  \lhd x +[y,  z \lhd x +  x \lhd  z  ].
\end{eqnarray*}
Thus we obtain
\begin{align*}
x \lhd [y,z] +  [y,z]\lhd x =  [y,    x \lhd  z  + z \lhd  x ].
\end{align*}
Therefore we have \vsc
\begin{eqnarray*}
&&x \lhd [y,z] +  [y,z]\lhd x=z \lhd [y,x] + [y,x]\lhd z,\\
&&x \lhd [z,y] +  [z,y]\lhd x=y \lhd [z,x] + [z,x]\lhd y,\\
&& y \lhd [x,z] + [x,z]\lhd y=z \lhd [x,y] + [x,y]\lhd z.
\end{eqnarray*}
Adding the above equations, we have
 $2(y \lhd [x,z] +  [x,z]\lhd y) = 0$ and hence
Eq.~(\mref{DPSPL2}) holds.

(\ref{it:t2}). Let $x,y,z \in A$. By Eq.~\eqref{DPSPL4}, we have
\vsb
\begin{eqnarray*}
x \rhd( y \lhd z)- y \rhd( x \lhd z)&=& (x \rhd y - y \lhd x-y\rhd
x+x\lhd y)\lhd z+y\lhd (x\circ z)-x\lhd (y\circ z)\\
&\mbox{}& - [x,y \lhd z] +[y,x\lhd z]\\
&=&(\{x,y\}-[x,y])\lhd z+y\lhd (x\circ z)-x\lhd (y\circ z) - [x,y
\lhd z] +[y,x\lhd z].
\vsd
\end{eqnarray*}
\vsb
Hence we have 
{\small\begin{eqnarray*} &&\{x,y\} \rhd  z - x \rhd
( y \rhd z  ) +y \rhd (x \rhd z )  -
[y,x \lhd z] +  [x, y \lhd z] + [x,y]\lhd z\\
&&=\{x,y\} \rhd  z - x \rhd ( y \rhd z  ) +y \rhd (x \rhd z )-(x
\rhd( y \lhd z)- y \rhd( x \lhd z)-\{x,y\}\lhd z-y\lhd (x\circ
z)+x\lhd (y\circ z))\\
&&=\{x,y\}\circ z-x\rhd (y\circ z)+y\rhd (x\circ z)+y\lhd (x\circ
z)-x\lhd (y\circ z)\\
&&=\{x,y\}\circ z -   x \circ ( y \circ z) +y \circ (x \circ  z).
\end{eqnarray*}}
Thus Eq.~(\mref{DPSPL5}) holds if and only if Eq.~(\mref{pl2})
holds.
\end{proof}

Therefore we have the following consequence.
\vsb
\begin{cor}\label{co:aa} Let  $(A, \rhd,\lhd,[-,-])$ be a   \mysymbol algebra.
Define a binary operation $\circ$ by Eq.~\eqref{hpla}. Then
$(A,\circ,[-,-])$ is a  post-Lie algebra. It is   called the
associated  \textbf{horizontal post-Lie algebra} of  $(A,
\rhd,\lhd,[-,-])$ and  $(A, \rhd,\lhd,[-,-])$ is called a
\textbf{compatible \mysymbol algebra} structure on the post-Lie
algebra $(A,\circ,[-,-])$.
\end{cor}
\vse

\begin{pro}\mlabel{pspleuqvi} Let  $(A, \rhd,\lhd,[-,-])$ be a   \mysymbol algebra.
\begin{enumerate}
 \item\mlabel{it:b2} Define a binary
operation $\bullet$ by
\vsb
\begin{equation}
x \bullet   y  :=x \rhd y - y \lhd x, \quad \forall x,y\in A. \mlabel{vpla}
\end{equation}
Then  $(A,\bullet,[-,-])$ is a  post-Lie algebra. It is   called
the associated  \textbf{vertical post-Lie algebra} of  $(A,
\rhd,\lhd,[-,-])$ and  $(A, \rhd,\lhd,[-,-])$ is called a
\textbf{compatible \mysymbol algebra} structure on the post-Lie
algebra $(A,\bullet,[-,-])$.
\item \mlabel{it:b3} The associated horizontal post-Lie algebra  $(A, \circ,[-,-])$ and vertical post-Lie algebra   $(A, \bullet,[-,-])$ have the same sub-adjacent Lie algebra $(\mathfrak{g}(A),\{-,-\})$ with
    $\{-,-\}$ defined by Eq.~\eqref{eq:curbra}.
It is called the {\bf sub-adjacent Lie algebra} of $(A, \rhd,\lhd,[-,-])$.
\end{enumerate}
\end{pro}

\begin{proof}
(\mref{it:b2}). It follows from a similar proof as the one for the
sufficient part of Lemma~\ref{lem:equ}.

(\mref{it:b3}). It follows from Proposition~\mref{newLie},
Corollary~\ref{co:aa}, Item~(\ref{it:b2})
 and the equation
 \vsb
    \begin{align*}
\{x,y\} &= x \circ y   - y \circ  x + [x,y]  =  x \bullet y   - y \bullet  x + [x,y]  \\
&= x \rhd y + x \lhd y  - y \rhd x - y \lhd x + [x,y]\quad \forall x,y\in A.
\qedhere
\end{align*}
\end{proof}
\vsc
\begin{ex}\mlabel{ex:ldl}
There exists a  compatible \mysymbol algebra  structure
$(\mathfrak{s l}(2, \mathbb{C}),\rhd,\lhd,[-,-])$ on the post-Lie
algebra  $(\mathfrak{s l}(2, \mathbb{C}),\circ,[-,-])$ in Example
\mref{splex} where  $\rhd$ and $\lhd$ are given by
\small{
\begin{eqnarray*} e_{1} \rhd  e_{1} &=&
0,\;\; e_{2} \rhd  e_{1}= e_3, \;\; e_{3} \rhd  e_{1}=-e_2, \;\;
e_{1} \rhd  e_{2} =  \frac{\mathrm{i}}{2} e_2 + \frac{1}{2}  e_3,\;\; e_{2} \rhd  e_{2}=-\frac{\mathrm{i}}{2} e_1,\\
e_{3} \rhd  e_{2}&=& \frac{5}{2}e_{1},\;\; e_{1} \rhd  e_{3} =
-\frac{1}{2} e_2 + \frac{\mathrm{i}}{2} e_3, \;\; e_{2} \rhd
e_{3}=-\frac{5}{2}e_{1}, \;\; e_{3} \rhd e_{3}=-
\frac{\mathrm{i}}{2} e_1,
e_{1} \lhd  e_{1} = 0,\\
 e_{2} \lhd  e_{1}&=&\frac{\mathrm{i}}{2} e_2 - \frac{1}{2}  e_3, \;\; e_{3} \lhd  e_{1}=\frac{1}{2} e_2 + \frac{\mathrm{i}}{2}  e_3,
 \;\;
e_{1} \lhd  e_{2} = -\frac{\mathrm{i}}{2} e_2 + \frac{1}{2}
e_3,\;\; e_{2} \lhd  e_{2}=0, \;\;
 e_{3} \lhd  e_{2}=-2e_{1},\\ e_{1}
\lhd e_{3} & = &-\frac{1}{2} e_2 - \frac{\mathrm{i}}{2}  e_3, \;\;
e_{2} \lhd e_{3}=2e_{1}, \;\; e_{3} \lhd  e_{3}=0.
\end{eqnarray*}}

Note that both the associated horizontal and vertical post-Lie algebras of
$(\mathfrak{s l}(2, \mathbb{C}),\rhd,\lhd$, $[-,-])$ are
$(\mathfrak{s l}(2, \mathbb{C}),\circ,[-,-])$.
\end{ex}

\begin{pro}\mlabel{transposeDPSPL}
 Let  $(A, \rhd,\lhd,[-,-])$ be a  \mysymbol algebra. Define
\vsb
\begin{align*}
\rhd^{t}, \lhd^{t}: A \otimes A \rightarrow A, \quad x \rhd^{t} y=x \rhd y, \quad x \lhd^{t} y=-y \lhd x, \quad \forall x, y \in A .
\end{align*}
Then  $(A, \rhd^{t}, \lhd^{t},[-,-]) $ is a  \mysymbol algebra, which is called the {\bf transpose} of  $(A,  \!\rhd, \lhd,\![-,-])$. The associated   horizontal post-Lie algebra of  $(A, \rhd^{t}, \lhd^{t},[-,-])$  is the associated  vertical post-Lie algebra  $(A, \bullet,[-,-])$ of  $(A, \rhd, \lhd,[-,-])$  and the associated  vertical post-Lie algebra of  $(A, \rhd^{t}, \lhd^{t},$ $[-,-])$  is the associated horizontal post-Lie algebra  $(A, \circ,[-,-])$  of  $(A, \rhd, \lhd,[-,-])$, that is,
\vsb
\begin{align*}
\bullet^{t}=\circ, \quad \circ^{t}=\bullet \text {. }
\end{align*}
\end{pro}
\begin{proof}
 It is straightforward.
 \end{proof}
\vsb
Thus it is
enough to only consider the associated horizontal post-Lie algebra of a \mysymbol algebra.
\vsb

\subsection{Partial splitting of operations of post-Lie algebras
and \mysymbol algebras}

We interpret \mysymbol algebras in terms of representations of
post-Lie algebras.

Recall that a {\bf representation of a Lie algebra}  $(\mathfrak{g},[-,-])$  is a pair  $(V; \rho )$, where $V$  is a vector space and  $\rho: \mathfrak{g} \rightarrow \mathfrak{gl}(V)$ is a linear map satisfying  $\rho([x, y])=[\rho(x), \rho(y)]$ for all $x, y \in \mathfrak{g}$.

\begin{defi}\mcite{TBGS} Let  $(A, \circ,[-,-])$  be a post-Lie algebra. A \textbf{representation} of $(A, \circ,[-,-])$  is a quadruple $(V;l, r, \rho)$ where $V$ is a vector space and $l,r, \rho : A\rightarrow {\rm End}_{\mathbb F}(V)$ are linear maps, such that $(V;\rho)$ is a representation of the Lie algebra $(A,[-,-])$ and $l, r$ satisfy the following equations.
\begin{align}
\rho(x \circ y) &=  l(x) \rho(y) - \rho(y) l(x) , \mlabel{eq6}\\
r([x , y])&= \rho(x)r (y) - \rho(y)r (x), \mlabel{eq7}\\
r(x \circ y) &= l(x) r(y) - r(y)(l(x) -r(x)   +\rho(x)), \mlabel{eq8}\\
l(\{x,y\})&=l(x) l(y) - l(y) l(x),\quad \forall x, y \in A, \mlabel{eq9}
\end{align}
where $\{-,-\}$ is defined by
Eq.~\eqref{eq5}. Representations  $(V_1;l_1, r_1, \rho_1)$  and $(V_2;l_2, r_2,
\rho_2)$ of a post-Lie algebra $(A,\circ,[-,-])$ are
\textbf{equivalent} if there exists a linear isomorphism $\varphi:
V_{1} \rightarrow V_{2}$  such that
\begin{align*}
\varphi\left(l_{1}(x) v\right)=l_{2}(x) \varphi(v),\;
\varphi\left(r_{1}(x) v\right)=r_{2}(x) \varphi(v),\;
\varphi\left(\rho_{1}(x) v\right)=\rho_{2}(x) \varphi(v),
\;\;\forall x \in A, v \in V_{1}.
\end{align*}
\end{defi}

\begin{ex}
Let  $(A, \circ,[-,-])$  be a post-Lie algebra. Then  $(A;
L_{\circ},R_{\circ},\mathrm{ad})$  is a representation of  $(A,
\circ,[-,-])$, called the \textbf{adjoint representation} of  $(A,
\circ,[-,-])$.
\end{ex}

\begin{pro} Let  $(A, \circ, [-,-])$  be a post-Lie algebra and $V$ be a
vector space. Let $l,r,\rho :A\rightarrow {\rm End}_{\mathbb
F}(V)$ be
linear maps. Then the quadruple $(V;l,r,\rho)$ is a representation
of $(A, \circ, [-,-])$ if and only if there is a post-Lie algebra
structure on $A\oplus V$ with the binary operations
$\circ_{A\oplus V}$ and $[-,-]_{A\oplus V}$ respectively defined
by
\begin{align}
\left(x_{1}+v_{1}\right) \circ_{A \oplus V} \left(x_{2}+v_{2}\right) &= x_{1} \circ x_{2}+l\left(x_{1}\right) v_{2}+r\left(x_{2}\right) v_{1}, \mlabel{pleq11}\\
\left[x_{1}+v_{1}, x_{2}+v_{2}\right]_{A \oplus V} &=
[x_{1},x_{2}]+\rho(x_1)v_2 - \rho(x_2)v_1,\quad \forall x_{1},
x_{2} \in A, v_{1}, v_{2} \in V. \mlabel{pleq10}
\end{align}
We denote this post-Lie algebra structure on $A\oplus V$ by
$A\ltimes_{l,r,\rho}V$.
\end{pro}

\begin{proof}
 It is a special case of the matched pair of post-Lie algebras in Proposition
 ~\mref{mppstl} when $B = V$ is equipped with the zero multiplications.
\end{proof}

Let $ V$ and $A$ be vector spaces. For a linear map  $\rho: A
\rightarrow \operatorname{End}_{\mathbb{F}}(V)$, define a linear
map $ \rho^{*}: A \rightarrow
\operatorname{End}_{\mathbb{F}}\left(V^{*}\right)$  by
\begin{align}
\left\langle\rho^{*}(x) v^{*}, u\right\rangle=-\left\langle v^{*},
\rho(x) u\right\rangle, \quad \forall x \in A, u \in V, v^{*} \in
V^{*},\mlabel{dualmap}
\end{align}
where $\left\langle -, -\right\rangle$ is the standard  pairing
between $V$ and $V^*$.

\begin{pro}\mlabel{dpsplrep}
Let  $(A,[-,-])$  be a Lie algebra and $\rhd,\lhd$ be binary operations on  $A$.  Define a
binary operation $\circ$ by Eq.~\eqref{hpla}. Then $(A,
\rhd,\lhd,[-,-])$  is a  \mysymbol algebra if and only if
 $(A, \circ,[-,-])$ is a post-Lie algebra of which  $(A^{*}; L_{\rhd}^{*}-R_{\lhd}^{*},-R_{\lhd}^{*}, \mathrm{ad}^{*})$ is a representation.
\end{pro}

\begin{proof} It is known that $(A^*; \ad^*)$ is a representation of the Lie algebra  $(A,[-,-])$. We rewrite  Eqs.~(\mref{eq6})--\eqref{eq9} in the case that $l=L_{\rhd}^{*}-R_{\lhd}^{*}, r=-R_{\lhd}^{*}$ and
$\rho=\mathrm{ad}^{*}$, which will be labeled by
Eqs.~(\mref{eq6}$^\prime$)--(\ref{eq9}$^\prime$) respectively.
Define a binary operation $\bullet$ by Eq.~\eqref{vpla}.  Let
$x,y,z \in A$, $a^* \in A^*$. Then we have
\begin{align*}
\langle (\mathrm{ad}^{*}(x \circ y)-[(L_{\rhd}^{*}-R_{\lhd}^{*})(x),\mathrm{ad}^{*}(y)])a^*, z \rangle &=  \langle a^*, ( [(L_{\rhd}-R_{\lhd})(x),\mathrm{ad}(y)]-\mathrm{ad}(x \circ y))z \rangle \\
&=   \langle a^*, -[x \circ y, z] + x \bullet [y,z] -  [y,x \bullet z] \rangle;\\
\langle (-R_{\lhd}^{*}([x , y])+\mathrm{ad}^{*}(x) R_{\lhd}^{*} (y) - \mathrm{ad}^{*}(y)R_{\lhd}^{*} (x))a^*, z \rangle
 &=   \langle a^*, ( R_{\lhd} ([x , y])+R_{\lhd} (y) \mathrm{ad} (x) - R_{\lhd}(x)\mathrm{ad} (y))z \rangle \\
&=  \langle a^*, z \lhd [x, y ] +   [x,z]\lhd y -  [y, z] \lhd x \rangle;
\end{align*}
\begin{eqnarray*}
&&\langle \big(-R_{\lhd}^{*}(x \circ y) + (L_{\rhd}^{*}-R_{\lhd}^{*})(x) R_{\lhd}^{*}(y) - R_{\lhd}^{*}(y)(L_{\rhd}^{*}(x)   +\mathrm{ad}^{*}(x))\big)a^*, z \rangle \\
 &&=  \langle a^*, \big( R_{\lhd} (x \circ y) +  R_{\lhd} (y)(L_{\rhd} -R_{\lhd} )(x) - (L_{\rhd} (x)   +\mathrm{ad} (x))R_{\lhd}(y)\big)z
 \rangle\\
&&=  \langle a^*,  z \lhd (x \circ y)    + (x \bullet z) \lhd y
-x \rhd (z \lhd y) - [x,z \lhd y]\rangle.
\end{eqnarray*}
Hence Eq.~\eqref{DPSPL3} $\Longleftrightarrow$
Eq.~(\ref{eq6}$^\prime$); Eq.~\eqref{DPSPL1} $\Longleftrightarrow$
Eq.~(\ref{eq7}$^\prime$); Eq.~\eqref{DPSPL4} $\Longleftrightarrow$
Eq.~(\ref{eq8}$^\prime$).

$(\Longleftarrow)$ Suppose that $(A, \circ,[-,-])$ is a post-Lie
algebra and (\mref{eq6}$^\prime$)--(\mref{eq9}$^\prime$) hold. In
particular, Eqs.~(\mref{DPSPL1}), (\mref{DPSPL3}) and
(\mref{DPSPL4}) hold. Hence by Lemma~\ref{lem:equ}, $(A,
\rhd,\lhd,[-,-])$  is a \mysymbol algebra.

$(\Longrightarrow)$ Suppose that $(A, \rhd,\lhd,[-,-])$ is a
\mysymbol algebra. By Corollary~\ref{co:aa}, $(A, \circ,[-,-])$ is
a post-Lie algebra. Moreover,
Eqs.~(\mref{eq6}$^\prime$)--(\mref{eq8}$^\prime$) hold and we have
\begin{eqnarray*}
&&\langle \big((L_{\rhd}^{*}-R_{\lhd}^{*})(\{x,y\})-[(L_{\rhd}^{*}-R_{\lhd}^{*})(x),(L_{\rhd}^{*}-R_{\lhd}^{*})(y)]\big)a^*, z \rangle \\
&&\overset{\hphantom{\eqref{pl1}}}{=}  \langle a^*,  -\{x,y\} \circ z + \{x,y\}\lhd z  + x \rhd ( y \bullet z  ) -( y\bullet z  )\lhd  x -y \rhd ( x \bullet z )  +  ( x \bullet z ) \lhd y+ z \lhd \{x,y\}\rangle\\
&&\overset{ \eqref{pl2} }{=}  \langle a^*,  - x\circ (y \circ z)+ y\circ (x \circ z) + \{x,y\}\lhd z  + x \rhd ( y \bullet z  ) -( y\bullet z  )\lhd  x -y \rhd ( x \bullet z )\\
&&\hspace{1.11cm}  +  ( x \bullet z ) \lhd y+ z \lhd \{x,y\}\rangle\\
&&\overset{\eqref{DPSPL4}}{=}\langle  a^*,   [x,z\lhd y] +
[x,y\lhd z]-[y, x \lhd z]-[y, z \lhd x] + z \lhd [y,x] +    [y,x]
\lhd z \rangle \overset{\eqref{pl1},\eqref{DPSPL3}} {=}    0.
\end{eqnarray*}
That is, Eq.~(\ref{eq9}$^\prime$) holds automatically. Hence
$(A^{*}; L_{\rhd}^{*}-R_{\lhd}^{*},-R_{\lhd}^{*},
\mathrm{ad}^{*})$ is a representation of  $(A, \circ,[-,-])$.
\end{proof}

\begin{rmk}
In fact, similarly, if Eqs.~(\mref{eq6}$^\prime$),
(\mref{eq8}$^\prime$) and (\ref{pl1}) hold, then
(\mref{eq9}$^\prime$) hold if and only if Eq.~(\ref{DPSPL5}), or
equivalently, Eq.~(\ref{pl2}) holds. Hence $(A, \rhd,\lhd,[-,-])$
is a \mysymbol algebra if and only if Eqs.~(\ref{pl1}) (or
equivalently, ~(\ref{DPSPL2})) and
(\mref{eq6}$^\prime$)--(\mref{eq9}$^\prime$) hold.
\end{rmk}

There is a theory of splitting the operations of algebras in terms
of the successors of operads \mcite{BBGN}. Explicitly, a binary
operation $\circ$ on a vector space $A$ is split into a sum of
two binary operations $\succ,\prec$ such that
$(A;L_\succ,R_\prec)$ gives a natural module structure
(representation) of the algebra $(A,\circ)$. In this sense, the
operad of the algebras with operations $\succ,\prec$ is the
successor of the operad of the algebras with the operation
$\circ$. For example, the operad of L-dendriform algebras is the
successor of the operad of pre-Lie algebras, that is, the operation of a pre-Lie
algebra is split into the sum of two operations defining an
L-dendriform algebra as above. In this sense, one can also define
the (ordinary) splitting of the operations of a post-Lie algebra $(A, \circ,[-,-])$
to be an algebra $(A, \rhd,\lhd,\star)$ satisfying
$$x\circ y=x\rhd y+x\lhd y,\;\;[x,y]=x\star y-y\star x,\;\;\forall
x,y\in A,$$ and $(A;L_\rhd,R_\lhd, L_\star)$ is a representation
of the post-Lie algebra $(A, \circ,[-,-])$. Moreover, the
operad of such algebras is the successor of the operad of post-Lie
algebras. Such an algebra
$(A, \rhd,\lhd,\star)$ is called a {\bf pre-post-Lie algebra}.
Note that in this case, both the operations $\circ$ and $[-,-]$
are split.

\begin{rmk}
\mlabel{r:pppl}
The \mysymbol algebra has the following special properties that is worth noting.
\begin{enumerate}
\item Proposition~\mref{dpsplrep} demonstrates a new kind of
splitting property of a post-Lie algebra, that is, the split of
$(A, \circ,[-,-])$ into the \mysymbol algebra $(A,
\rhd,\lhd,[-,-])$ is “partial” in the sense that, unlike the pre-post-Lie algebra just defined above, the operation
$\circ$ splits into the sum of the operations $\rhd$ and $\lhd$,
whereas the Lie bracket $[-,-]$ remains unchanged.

\item The characterization of \mysymbol algebras is given in terms
of post-Lie algebras with the specific representations on the dual
spaces instead of the representations on the underlying vector
spaces themselves. Note that if there is a well-constructed dual
representation for any representation of an algebra (such an
algebra is called proper in \mcite{Ku}) such as pre-Lie algebras,
then obviously, the splitting of the operations can be
characterized in terms of the algebras with the specific
representations on the underlying vector spaces themselves as well
as equivalently their dual representations (on the dual spaces).
Unfortunately, there is not a characterization of a  \mysymbol
algebra in terms of a post-Lie algebra with a specific
representations on the underlying vector space itself, due to the
fact that  no  dual representation of a representation of a
post-Lie algebra has been found.
\end{enumerate}
\end{rmk}

Similar to the $\mathcal O$-operators (or relative Rota-Baxter
operators) giving a further interpretation of the usual
splitting of operations \mcite{BBGN}, we give an
interpretation of \mysymbol algebras as the partial splitting of
operations of post-Lie algebras in terms of analogs of
$\mathcal O$-operators. However, note that these operators are
defined on the dual spaces of the representation spaces instead of
the representation spaces themselves.

\begin{defi}
Let  $(V; l,r,\rho)$ be a representation of a post-Lie algebra
$(A, \circ, [-,-])$. A linear map  $T: V^{*} \rightarrow A $ is
called a   {\bf \dpdop} on $(A, \circ, [-,-])$ associated
to $(V;l,r,\rho)$ if
\begin{align}
            T(u^{*})\circ T(v^{*}) &= T\Big((l^{*}-r^{*})\big(T(u^{*})\big)v^{*}-r^{*}\big(T(v^{*})\big)u^{*}\Big),\mlabel{ldlo1}\\
[T(u^{*}), T(v^{*})] &= T \Big(\rho^{*}\big(T(u^{*})\big)v^{*}\Big) = -T\Big(\rho^{*}\big(T(v^{*})\big) u^{*}\Big),\;\forall u^{*},v^{*}\in V^{*}.\mlabel{ldlo2}
\end{align}
In addition, a \dpdop is called {\bf strong} if $T$ satisfies
\begin{align}
& \rho^{*}\big( T(u^{*}) \big) v^{*}  = -\rho^{*}\big( T(v^{*}) \big) u^{*}, \mlabel{sdpo1}\\
&\rho^{*}\big(T(u^{*})\big) \Big( r^{*}(T(v^{*}))w^{*}  +  r^{*}(T(w^{*}))v^{*} \Big)  =   r^{*}\Big( [T(u^{*}), T(w^{*})] \Big)v^{*}  +   r^{*}\big( T(v^{*}) \big) \rho^{*} ( T(u^{*})  ) w^{*}  = 0 , \mlabel{sdpo2} \\
&\rho^{*}\big([T(u^{*}), T(v^{*})]\big) w^{*}  +\rho^{*}\big([T(v^{*}), T(w^{*})]\big) u^{*}+\rho^{*}\big([T(w^{*}), T(u^{*})]\big) v^{*}=0,\; \forall u^{*}, v^{*}, w^{*} \in V^{*}. \mlabel{sdpo3}
\end{align}
\end{defi}

\begin{pro}\mlabel{pro:dual ldl}
Let  $(V; l,r,\rho)$ be a representation of a post-Lie algebra
$(A, \circ, [-,-])$. Let $T:V^{*}\rightarrow A$ be a \dpdop on
$(A, \circ, [-,-])$ associated to  $(V; l,r,\rho)$. Define three
binary operations
$\rhd_{V^{*}},\lhd_{V^{*}},[-,-]_{V^{*}}:V^{*}\otimes
V^{*}\rightarrow V^{*}$ by
\begin{align}\mlabel{eq:dual mul}
    u^{*}\rhd_{V^{*}}v^{*} \!:=\!(l^{*}-r^{*})\big(T(u^{*})\big)v^{*},\;
    u^{*}\lhd_{V^{*}}v^{*} \!:=\! -r^{*}\big(T(v^{*})\big)u^{*},\; [u^{*},v^{*}]_{V^{*}}   \!:=\!   \rho^{*}\big(T(u^{*})\big)
    v^{*},
\end{align}
for all $u^{*},v^{*}\in V^{*}$. Then $(V^{*},\rhd_{V^{*}},\lhd_{V^{*}},[-,-]_{V^*})$ is a
\mysymbol algebra if and only if $T$ is strong.
\end{pro}

\begin{proof}It is obvious that
$(V^*,[-,-]_{V^*})$ is a Lie algebra if and only if $T$ satisfies
Eqs.~\eqref{sdpo1} and~\eqref{sdpo3}.

$(\Longrightarrow)$ Suppose that
$(V^{*},\!\rhd_{V^{*}},\lhd_{V^{*}},[-,-]_{V^*})$ is a  \mysymbol
algebra. Let $u^{*},v^{*},w^{*}\in V^{*}$. Then we have
\begin{eqnarray*}
&&\rho^{*}\big(T(u^{*})\big) \Big( r^{*}(T(v^{*}))w^{*} +
r^{*}(T(w^{*}))v^{*} \Big)
= -[u^*, v^* \lhd w^* + w^* \lhd v^*] \\
&&\overset{\eqref{DPSPL2}}{=} -[u^*,w^*] \lhd v^* -v^* \lhd
[u^*,w^*]= r^{*}\Big( [T(u^{*}), T(w^{*})] \Big)v^{*} + r^{*}\big(
T(v^{*}) \big) \rho^{*} ( T(u^{*})  )
w^{*}\overset{\eqref{DPSPL2}}{=} 0.
\end{eqnarray*}
Thus Eq.~\eqref{sdpo2} holds and hence $T$ is strong.

$(\Longleftarrow)$ Suppose that  $T$ is strong. Let $u^{*},v^{*},w^{*}\in V^{*},a \in V$. Then we
have
\begin{eqnarray*}
&&\langle   u^{*} \lhd  [v^{*},w^{*}] -[u^{*} ,v^{*}]\lhd w^{*} -  [w^{*},u^{*} ] \lhd v^{*}, a\rangle \\
&&\overset{\hphantom{\eqref{ldlo2}}}{=} \langle    -r^*\Big(T\big(\rho^*(T(v^*))w^{*}\big)\Big)u^{*}   +r^*\big(T (w^{*} )\big)\rho^*\big(T(u^*)\big)v^{*}  + r^*\big(T(v^{*} )\big)\rho^*\big(T(w^*)\big)u^{*}, a\rangle \\
&&\overset{\eqref{ldlo2},\eqref{sdpo1}}{=} \langle    -r^*\Big(\big[T(v^*),T(w^*)\big]\Big)u^{*}   -r^*\big(T (w^{*} )\big)\rho^*\big(T(v^*)\big)u^{*}  + r^*\big(T(v^{*} )\big)\rho^*\big(T(w^*)\big)u^{*}, a\rangle \\
&&\overset{\hphantom{\eqref{ldlo2}}}{=} \langle   u^{*}, -\Big(\rho \big(T(v^*)\big)r\big(T (w^{*} )\big) - \rho\big(T(w^*)\big)r\big(T(v^{*} )\big)- r\big([T(v^*),T(w^*)]\big) \Big)a\rangle \overset{\eqref{eq7}}{=} 0.
\end{eqnarray*}
Thus Eq.~\eqref{DPSPL1}
holds. Similarly, Eqs.~\eqref{DPSPL2}-\eqref{DPSPL5} hold. Thus
$(V^{*},\rhd_{V^{*}},\lhd_{V^{*}},[-,-]_{V^*})$  is a \mysymbol
algebra.
\end{proof}

\begin{pro}\mlabel{pro:dual strong}
If a \dpdop on a post-Lie algebra is invertible, then it is strong.
\end{pro}

\begin{proof}
Suppose that $T$ is an invertible \dpdop on a post-Lie algebra $(A,\circ,[-,-])$
associated to a representation $(V;l,r,\rho)$.  Define binary
operations $\rhd_{V^{*}},\lhd_{V^{*}},[-,-]_{V^{*}}:V^{*}\otimes
V^{*}\rightarrow V^{*}$ by Eq.~\eqref{eq:dual mul}. Then we have
\begin{align*}
u^{*} \circ_{V^{*}} v^* :&= u^{*} \rhd_{V^{*}} v^* + u^{*} \lhd_{V^{*}} v^* =   (l^{*}-r^{*})\big(T(u^{*})\big)v^{*}  -r^{*}\big(T(v^{*})\big)u^{*} = T^{-1}(T(u^{*})\circ T(v^{*}) )\\
[u^{*} , v^*]_{V^{*}}  &=\rho^{*}\big(T(u^{*})\big)v^{*} = T^{-1}([T(u^{*}),T(v^{*})]), \quad \forall u^{*}, v^{*} \in A^*.
\end{align*}
Thus $(V^*,\circ_{V^{*}},[-,-]_{V^{*}})$ is a post-Lie algebra.
Moreover, it is straightforward to show that Eqs.~\eqref{DPSPL1},
\eqref{DPSPL3} and  \eqref{DPSPL4} hold. By Lemma~\ref{lem:equ},
$(V^{*},\rhd_{V^{*}},\lhd_{V^{*}},[-,-]_{V^*})$ is a \mysymbol
algebra. Then by Proposition~\ref{pro:dual ldl}, $T$ is strong.
\end{proof}

\begin{thm}\mlabel{thm:dual 2}
    Let $(A, \circ, [-,-])$ be a post-Lie algebra. There is a compatible \mysymbol algebra structure $(A,\rhd,\lhd,[-,-])$ on $(A, \circ, [-,-])$
if and only if there is an invertible \dpdop on $(A, \circ,
[-,-])$ associated to a representation $(V;l,r,\rho)$. In this
case,  $\rhd,\lhd$ are defined by
    \begin{equation}\mlabel{deq:pro:2.14}
        x\rhd y:=T\big( (l^{*}-r^{*})(x)T^{-1}(y)\big),\;
        x\lhd y:=-T\big( r^{*} (y)T^{-1}(x)\big),\; \forall x,y\in A.
    \end{equation}
\end{thm}

\begin{proof} $(\Longrightarrow)$  By Proposition
\mref{dpsplrep}, $(A^{*}; L_{\rhd}^{*}-R_{\lhd}^{*},-R_{\lhd}^{*},
\mathrm{ad}^{*})$ is a representation  of $(A, \circ, [-,-])$. For
all $x,y \in A$, we have
\begin{align*}
 \mathrm{id}(x) \circ \mathrm{id}(y)& = 
x \rhd y + x \lhd y = \mathrm{id}\Big(\big((L_{\rhd}^{*}-R_{\lhd}^{*} )^{*}-(-R_{\lhd}^{*})^{*}\big) ( \mathrm{id}(x) )y-\big(-R_{\lhd}^{*}\big)^{*} ( \mathrm{id}(y) )x\Big),\\
 [\mathrm{id}(x), \mathrm{id}(y)]  &=[x,y] = \mathrm{id} ((\mathrm{ad}^*)^*(\mathrm{id}(x))y) =-[y,x] = -\mathrm{id} ((\mathrm{ad}^*)^*(\mathrm{id}(y))x) .
\end{align*}
Thus $\mathrm{id}: (A^*)^* \cong A \to A$ is an invertible \dpdop
on $(A,\circ,[-,-])$ associated to  the representation $(A^{*};
L_{\rhd}^{*}-R_{\lhd}^{*},-R_{\lhd}^{*}, \mathrm{ad}^{*})$.

 $(\Longleftarrow)$ Suppose that  $T$ is an invertible \dpdop on $(A, \circ,
[-,-])$ associated to a representation $(V;l,r,\rho)$.  Then  we
have
    \begin{align*}
 x \circ y &\overset{\eqref{ldlo1}}{=}  T\big( (l^{*}-r^{*})(x)T^{-1}(y) -  r^{*} (y)T^{-1}(x)\big)  \overset{\eqref{deq:pro:2.14}}{=}     x\rhd y + x\lhd
 y,\;\;
  [x , y] \overset{\eqref{ldlo2}}{=}  T\big( \rho^{*} (x)T^{-1}(y)\big),
\end{align*}
for all $x,y\in A$. By Propositions \mref{pro:dual ldl} and
\mref{pro:dual strong}, $(A,\rhd,\lhd,[-,-])$ is a compatible
\mysymbol algebra structure
 on $(A, \circ, [-,-])$, where the
binary operations $\rhd,\lhd$  are defined by
Eq.~\eqref{deq:pro:2.14}.
\end{proof}

\subsection{\Mysymbol algebras from \gphpl algebras}

\begin{pro}\mlabel{compaDPSPL1}
Let $(A, \circ,[-,-],\mathcal{B})$ be a \gphpl algebra. Then there exists a compatible \mysymbol algebra
structure $(A,\rhd,\lhd,[-,-])$  on  $(A, \circ,[-,-])$ with
$\rhd,\lhd$ defined by
\begin{align}
\mathcal{B}(x \rhd y, z) &= -\mathcal{B}(y, x \circ z- z \circ x),\mlabel{comDpspl1} \\
\mathcal{B}(x \lhd y, z) &=  \mathcal{B}(x, z \circ y), \quad \forall x, y, z \in A. \mlabel{comDpspl2}
\end{align}
Moreover,  the  representations  $(A;L_{\circ}, R_{\circ},
\mathrm{ad})$    and
$(A^*;L_{\rhd}^*-R_{\lhd}^*,-R_{\lhd}^*,\mathrm{ad}^*)$  of the
post-Lie algebra $(A,\circ,[-,-])$   are equivalent.

Conversely, let $(A,  \rhd,\lhd,[-,-])$  be a compatible \mysymbol algebra structure  on a post-Lie algebra $(A,  \circ, [-,-])$.
 If  $(A;L_{\circ}, R_{\circ},
\mathrm{ad})$    and
$(A^*;L_{\rhd}^*-R_{\lhd}^*,-R_{\lhd}^*,\mathrm{ad}^*)$  are equivalent as representations of  $(A,  \circ, [-,-])$, then there exists a nondegenerate  bilinear form  $\mathcal{B}$ on $(A, \circ,[-,-])$ satisfying Eq.~\eqref{lieinv} and the equation
 \begin{align}
\mathcal{B}(x \circ y, z)=-\mathcal{B}(y, x \circ z - z \circ
x)+\mathcal{B}(z \circ y,x),\;\;\forall x,y,z\in A.
\mlabel{quasi2-cocycle}
\end{align}
In particular, if $\mathcal{B}$ is symmetric, then $(A,\circ,[-,-],\mathcal{B})$ is a \gphpl algebra.
\end{pro}

\begin{proof}
Define an invertible linear map $\varphi: A \to A^*$ by
 \begin{align*}
  \langle \varphi(x ), y\rangle &= \mathcal{B}( x ,y),\; \forall x,y \in A.
 \end{align*}
 For all $u^*,v^*,w^* \in A^*$, there exist $x,y,z \in A$ such that $u^* = \varphi(x), v^* = \varphi(y), w^* = \varphi(z)$. Therefore
\begin{eqnarray*}
\langle \varphi(x \circ y), z \rangle &=& \langle (L_{\circ}^*-
R_{\circ}^*)(x)\varphi(y),z\rangle - \langle
R_{\circ}^*(y)\varphi(x), z \rangle,\;\; \langle \varphi([x,y]), z
\rangle  = \langle \mathrm{ad}^{*}(x) \varphi  (y) , z \rangle.
\end{eqnarray*}
Then we obtain {\small \begin{align*}
 \varphi^{-1}(u^*) \circ \varphi^{-1}(v^*)  &= \varphi^{-1}\big((L_{\circ}^*- R_{\circ}^*)(\varphi^{-1}(u^*))v^* -
 R_{\circ}^*(\varphi^{-1}(v^*))u^*\big),\;\;
 [\varphi^{-1}(u^*),\varphi^{-1}(v^*)] = \varphi^{-1}\big( \mathrm{ad}^{*} (\varphi^{-1}(u^*))v^*\big).
\end{align*}}
Hence $\varphi^{-1}: A^* \to A$ is an invertible \dpdop on $(A,
\circ, [-,-])$ associated to the representation $(A;L_{\circ},
R_{\circ}, \mathrm{ad})$. By Theorem \mref{thm:dual 2},   there is
a compatible \mysymbol algebra structure $(A,\rhd,\lhd,[-,-])$ on
$(A, \circ, [-,-])$ with $\rhd,\lhd$    defined by
Eq.~\eqref{deq:pro:2.14} in taking $T := \varphi^{-1}, l :=
L_{\circ}, r := R_{\circ}$. Thus we have
\begin{align*}
\mathcal{B}(x \rhd y, z) &=  \langle \varphi (x \rhd y), z\rangle
= \langle  ( L_{\circ}^{*}-R_{\circ}^{*})(x)\varphi(y), z\rangle =
 -\mathcal{B}(y, x \circ z- z \circ x), \\
\mathcal{B}(x \lhd y, z) &=  \langle \varphi (x \lhd y), z\rangle
= -\langle  R_{\circ}^{*} (y)\varphi (x), z\rangle =
\mathcal{B}(x, z \circ y).
\end{align*}
Moreover,  by the symmetry of $\mathcal{B}$, we obtain
 \begin{align}
\mathcal{B}( x \circ y,z) &=  \mathcal{B}( x,z \lhd y), \mlabel{symin1}\\
\mathcal{B}( x \circ y,z) &=   \mathcal{B}(z,   y \circ
x)-\mathcal{B}(x \rhd z, y) =   \mathcal{B}(z \lhd x,   y
)-\mathcal{B}(y,x \rhd z)=  \mathcal{B}(y ,z \lhd x -x \rhd z).
\mlabel{symin2}
\end{align}
Then we have
 \begin{eqnarray*}
\langle \varphi(L_{\circ }(x)y), z\rangle &=& \mathcal{B}(y ,z \lhd x -x \rhd z) = \langle (L_{\rhd}^*-R_{\lhd}^*)(x)\varphi(y),z\rangle.\\
\langle \varphi(R_{\circ }(y)x), z\rangle &=&  \mathcal{B}( x \circ y,z) = \mathcal{B}( x,z \lhd y)  = \langle - R_{\lhd}^* (y)\varphi(x),z\rangle.\\
\langle \varphi(\mathrm{ad}(x)y), z \rangle &=& \mathcal{B}([x,y],
z) =- \mathcal{B}(y, [x,z])  =  \langle \mathrm{ad}^{*}(x) \varphi
(y) , z \rangle.
  \end{eqnarray*}
 Thus $\varphi$ gives  the equivalence between $(A;L_{\circ}, R_{\circ}, \mathrm{ad})$    and $(A^*;L_{\rhd}^*-R_{\lhd}^*,-R_{\lhd}^*,\mathrm{ad}^*)$ as the representations of  $(A,\circ,[-,-])$. The converse statement can be proved by   reversing the argument.
\end{proof}

\begin{cor}
Let $(A, \circ,[-,-],\mathcal{B})$ be a \gphpl
algebra. Then there exists another post-Lie algebra $(A,
\bullet,[-,-])$  with $\bullet$ defined by
\begin{align}
\mathcal{B}(x \bullet y, z) = -\mathcal{B}( y, x \circ z) ,\quad \forall x,y,z \in A.
\end{align}
Moreover, $(A, \bullet,[-,-],\mathcal{B})$ is also a \gphpl
algebra.
\end{cor}

\begin{proof}
By Proposition \mref{compaDPSPL1}, there is a compatible \mysymbol
algebra structure $(A,\rhd,\lhd,[-,-])$ on the post-Lie algebra
$(A, \circ, [-,-])$ with $\rhd,\lhd$ defined by
Eqs.~\eqref{comDpspl1}--\eqref{comDpspl2} and   $(A, \circ,
[-,-])$  is  the associated horizontal post-Lie algebra of the
\mysymbol algebra $(A,\rhd,\lhd,[-,-])$. It is straightforward to
show that $(A, \bullet,[-,-])$ is the associated   vertical
post-Lie algebra of $(A,\rhd,\lhd,[-,-])$ and $(A,
\bullet,[-,-],\mathcal{B})$ is  a \gphpl algebra.
\end{proof}

\begin{pro}
\label{p:ppplgphpl}
Let $(A,\rhd,\lhd,[-,-])$ be a  \mysymbol algebra and
$(A,\circ,[-,-])$ be the associated  horizontal post-Lie algebra. Then $(A
\ltimes_{L_{\rhd}^{*}-R_{\lhd}^{*}, -R_{\lhd}^{*}, \mathrm{ad}^*}
A^*,\mathcal{B}_d)$ is a \gphpl algebra, where
$\mathcal{B}_d$ is the nondegenerate symmetric  bilinear
form  defined by
\begin{align}\label{mtorbl}
\mathcal{B}_d(x+a^*,y+b^*) :=   \langle  x,b^* \rangle +  \langle  y,a^* \rangle, \quad \forall x,y \in A, a^*,b^* \in A^*.
\end{align}
\end{pro}

\begin{proof}
It is direct to check that  $\mathcal{B}_d$ is  \mysymbolnew on the post-Lie
algebra $A \ltimes_{L_{\rhd}^{*}-R_{\lhd}^{*}, -R_{\lhd}^{*},
    \mathrm{ad}^*} A^*$.
\end{proof}

\section{Pp-post-Lie bialgebras, $\mathcal{O}$-operators and classical Yang-Baxter equation}
\mlabel{s:ldlb}

We introduce the notions of a Manin triple  of post-Lie algebras
associated to an \mysymbolnew bilinear form, and a \mysymbol
bialgebra. The equivalence between them is interpreted in terms of
certain matched pairs of  post-Lie algebras. The study of a
special class of \mysymbol bialgebras leads to the introduction of
the \PCYBE in a  \mysymbol algebra whose antisymmetric solutions
give \mysymbol bialgebras. The \PCYBE is given an operator form as
the $\mathcal{O}$-operators on  \mysymbol algebras and the
construction of skew-symmetric solutions of \PCYBE is given in
terms of $\mathcal{O}$-operators on \mysymbol algebras.

\subsection{Matched pairs, Manin triples of post-Lie algebras and \mysymbol bialgebras}

\begin{pro} \mlabel{mppstl}
Let $(A,\circ_A,[-,-]_A)$  and $(B, \circ_B,[-,-]_B)$ be
post-Lie algebras. Let $l_{A},
r_{A}, \rho_A: A \rightarrow {\rm End}_{\mathbb F}(B)$  and
$l_{B}, r_{B}, \rho_B: B \rightarrow  {\rm End}_{\mathbb F}(A)$ be linear maps. Define binary operations $\circ$ and $[-,-]$ on $A\oplus B$ by
\vsb
\begin{eqnarray}
   & (x+a) \circ (y+b)  := \left(x \circ_A y+l_{B}(a) y+r_{B}(b) x\right)+\left(a \circ_B b+l_{A}(x) b+r_{A}(y) a\right),& \mlabel{r5}\\
    &[x+a, y+b]  := [x, y]+\rho_B(a) y-\rho_B(b) x+[a, b]+\rho_A(x)
    b-\rho_A(y) a, \ \  \forall x, y\in A, a, b\in B.& \label{r6}
\end{eqnarray}
Then $(A \oplus B,\circ,[-,-])$ is a post-Lie algebra if and only
if the quadruples $\left(B;l_{A}, r_{A},\rho_A \right)$ and
$\left(A;l_{B}, r_{B},\rho_B \right)$ are the representations of
the post-Lie algebras $(A,\circ_A,[-,-]_A)$  and $(B,\circ_B,
[-,-]_B)$ respectively, and the following conditions are
satisfied.
\begin{align}
\rho_A(x)([a, b]_B)&=  [\rho_A(x) a, b ]_{B}+ [a, \rho_A(x) b ]_{B}+ \rho_A (\rho_B(b) x ) a -\rho_A (\rho_B(a) x ) b,\mlabel{mplo}  \\
\rho_A(x)(a \circ_B b)&=  a \circ_B \rho_A(x)b  +[b,r_A(x)a]_B - \rho_A(l_B(a)x)b -r_{A}(\rho_B(b)x)a, \mlabel{mppl3} \\
\rho_B(a)([x, y]_A)&=  [\rho_B(a) x, y ]_{A}+ [x, \rho_B(a) y ]_{A}+\rho_B (\rho_A(y) a ) x - \rho_B (\rho_A(x) a ) y,\mlabel{mpl1} \\
\rho_B(a)(x \circ_A y)  &=x \circ_A \rho_B(a)y + [y,r_B(a)x]_A - \rho_B(l_A(x)a)y -r_{B}(\rho_A(y)a)x,  \mlabel{mppl1} \\
l_A(x)([a,b]_B) &= [l_A(x)a,b]_B + [a,l_A(x)b]_B + \rho_A(r_B(a)x)b - \rho_A(r_B(b)x)a, \mlabel{mppl4} \\
l_A(x)(a \circ_B b) &= l_A(x)a\circ_B b + a\circ_B l_A(x)b  - r_A(x)a \circ_B b + \rho_A(x)a \circ_B b  \nonumber \\
 &\quad + r_A(r_B(b)x)a - l_A(l_B(a)x)b + l_A(r_B(a)x)b- l_A(\rho_B(a)x)b, \mlabel{mppl7}\\
l_B(a)([x,y]_A) &= [l_B(a)x,y]_A + [x,l_B(a)y]_A + \rho_B(r_A(x)a)y - \rho_B(r_A(y)a)x, \mlabel{mppl2} \\
l_B(a)(x \circ_A y) &=  l_B(a)x\circ_A y + x\circ_A l_B(a)y  - r_B(a)x \circ_A y + \rho_B(a)x \circ_A y  \nonumber \\
 &\quad + r_B(r_A(y)a)x - l_B(l_A(x)a)y + l_B(r_A(x)a)y - l_B(\rho_A(x)a)y, \mlabel{mppl5} \\
r_A(x)(\{a,b\}_B) &= a \circ_B r_A(x)b - b \circ_B r_A(x)a + r_A(l_B(b)x)a - r_A(l_B(a)x)b, \mlabel{mppl8}\\
r_B(a)(\{x,y\}_{A}) &= x \circ_A r_B(a)y - y \circ_A r_B(a)x+r_B(l_A(y)a)x - r_B(l_A(x)a)y, \mlabel{mppl6}
\end{align}
 for all $x, y \in A$, $a, b \in B$. In this case, we denote this
post-Lie algebra by $A \bowtie_{l_{B}, r_{B},\rho_B}^{l_{A},
r_{A},\rho_A} B$ or simply $A \bowtie B$, and call such an octuple
$\left(A, B, l_{A}, r_{A},\rho_A, l_{B}, r_{B},\rho_B\right)$
a \textbf{matched pair of post-Lie algebras}.
\end{pro}
   \begin{proof}
The checking is straightforward.
   \end{proof}

\begin{defi}  \label{MT:pl}
A {\bf (standard) Manin triple of post-Lie algebras
associated  to
 the \mysymbolnew bilinear form} is a triple  $\big((A \oplus   A^{*},\!\circ_{A \oplus   A^{*}}, \![-,-]_{A \oplus   A^{*}},\mathcal{B}_{d}),
 (A,\circ_A, [-,-]_A)$, $(A^{*}$, $\circ_{A^{*}},[-,-]_{A^{*}})\big)$  for which
\begin{enumerate}
\item as vector spaces,  $A \oplus A^{*}$  is the direct sum of
$A$  and  $A^{*}$; \item  $(A,\circ_A, [-,-]_A)$  and
$(A^{*},\circ_{A^{*}},[-,-]_{A^{*}})$  are post-Lie subalgebras of
the post-Lie algebra $(A \oplus A^{*},\circ_{A \oplus   A^{*}}$,
$[-,-]_{A \oplus A^{*}})$; \item  $(A \oplus   A^{*},
\circ_{A \oplus A^{*}}, [-,-]_{A \oplus A^{*}},\mathcal{B}_{d})$
is a \gphpl algebra with $\mathcal{B}_{d}$ given by Eq.~\eqref{mtorbl}.
\end{enumerate}
\mlabel{d:mtpl}
\end{defi}

\begin{lem}\mlabel{lem:111}
Let $\big((A \oplus   A^{*}, \!\circ_{A \oplus   A^{*}}, \![-,-]_{A
\oplus   A^{*}},\mathcal{B}_{d}),
 (A,\circ_A, [-,-]_A)$, $(A^{*}$, $\circ_{A^{*}},[-,-]_{A^{*}})\big)$
be a Manin triple of post-Lie algebras associated to the
invariant bilinear form $\mathcal{B}_{d}$. Then there exists a compatible \mysymbol
algebra structure $(A\oplus A^*,\rhd_{A\oplus A^*},\lhd_{A\oplus
A^*},[-,-]_{A\oplus A^*})$ with $\rhd_{A\oplus A^*},\lhd_{A\oplus
A^*}$ respectively defined by Eqs.~\eqref{comDpspl1} and
\eqref{comDpspl2} through $\mathcal B_d$. Moreover, $(A,
\rhd_{A}=:\!\rhd_{A\oplus A^*}|_{A\otimes A},
\lhd_{A}=:\!\lhd_{A\oplus A^*}|_{A\otimes A}, [-,-]_A)$ and $(A^*,\!
\rhd_{A^*}=:\rhd_{A\oplus A^*}|_{A^*\otimes A^*},
\lhd_{A^*}=:\lhd_{A\oplus A^*}|_{A^*\otimes A^*}$, $[-,-]_{A^*})$
are \mysymbol subalgebras whose associated horizontal post-Lie
algebras are $(A,\circ_A$, $[-,-]_A)$ and $(A^{*}$,
$\circ_{A^{*}},[-,-]_{A^{*}})$ respectively.
\end{lem}

\begin{proof}
The first conclusion follows from Proposition \mref{compaDPSPL1}. Let  $x, y \in A$. Suppose that  $x \rhd_{A} y=m+a^{*}$, $x \lhd_{A} y=n+b^{*}$, where  $m,n \in A$  and $a^{*},b^* \in A^{*}$. Then we have
\begin{align*}
\left\langle a^{*}, z\right\rangle &=\mathcal{B}_{d}\left(x \rhd_{A} y, z\right) \overset{\eqref{comDpspl1}}{=} -\mathcal{B}_{d}(y, x \circ_{A} z- z \circ_{A} x)=0,   \\
\left\langle b^{*}, z\right\rangle &=\mathcal{B}_{d}\left(x \lhd_{A} y, z\right) \overset{\eqref{comDpspl2}}{=}\mathcal{B}_{d}(x, z \circ_{A} y)=0, \; \forall z \in A.
\end{align*}
Hence  $ a^{*}=0 = b^{*}$ and thus  $ x \rhd_{A} y,\; x \lhd_{A} y  \in A$. So  $(A,
\rhd_{A},
\lhd_{A}, [-,-]_A)$  is a  \mysymbol subalgebra. Similarly,  $ (A^*,
\rhd_{A^*},
\lhd_{A^*}, [-,-]_{A^*}) $   is also a  \mysymbol subalgebra. Obviously, the associated horizontal post-Lie
algebras  of  $(A,
\rhd_{A},
\lhd_{A}, [-,-]_A)$  and  $ (A^*,
\rhd_{A^*},
\lhd_{A^*}, [-,-]_{A^*}) $  are  $ \left(A,\circ_{A} ,[-,-]_{A}\right) $  and  $ \left(A^{*},\circ_{A^{*}},[-,-]_{A^{*}}\right)$  respectively.
\end{proof}

\begin{pro} \mlabel{mtsplb}
Let $(A,\rhd_A,\lhd_{A}, [-,-]_A)$  and
$(A^{*},\rhd_{A^{*}},\lhd_{A^{*}}, [-,-]_{A^{*}})$ be
\mysymbol algebras, and let their associated horizontal post-Lie
algebras be $(A,\circ_A$, $[-,-]_A)$ and $(A^{*}$,
$\circ_{A^{*}}$, $[-,-]_{A^{*}})$ respectively. Then there is a
Manin triple  $\big((A \oplus   A^{*},\!\circ_{A \oplus   A^{*}},
\![-,-]_{A \oplus   A^{*}},\mathcal{B}_{d}),
 (A,\circ_A$, $[-,-]_A)$, $(A^{*}$, $\circ_{A^{*}},[-,-]_{A^{*}})\big)$
 of post-Lie algebras associated to  the invariant
bilinear form $\mathcal{B}_{d}$ such that the compatible \mysymbol algebra structure
$(A\oplus A^*,\rhd_{A\oplus A^*},\lhd_{A\oplus A^*},[-,-]_{A\oplus
A^*})$ with $\rhd_{A\oplus A^*},\lhd_{A\oplus A^*}$ respectively
defined by Eqs.~\eqref{comDpspl1} and \eqref{comDpspl2} through
$\mathcal B_d$ contains $(A,\rhd_A,\lhd_A,[-,-]_A)$  and
$(A^{*},\rhd_{A^{*}},\lhd_{A^{*}},$ $ [-,-]_{A^{*}})$  as \mysymbol
subalgebras if and only if  $(A,
A^{*},L_{\rhd_{A}}^{*}-R_{\lhd_{A}}^{*},-R_{\lhd_{A}}^{*},\mathrm{ad}_{A}^{*},
L_{\rhd_{A^{*}}}^{*}-R_{\lhd_{A^{*}}}^{*},-R_{\lhd_{A^{*}}}^{*},$ $\mathrm{ad}_{A^{*}}^{*})$
is a matched pair of post-Lie algebras.
\end{pro}

\begin{proof}
($\Longrightarrow$) Let $x, y \in A, a^{*}, b^{*} \in A^{*}$.
 By Proposition \mref{compaDPSPL1}, we have
 \vsb
\begin{align*}
\mathcal{B}_{d}(x \circ_{A\oplus A^*} b^{*},a^{*}) &\overset{\eqref{symin1}}{=} \mathcal{B}_{d}(x ,a^{*}\lhd_{A^{*}}  b^{*}) = \left \langle x ,a^{*}\lhd_{A^{*}}   b^{*} \right \rangle  = \left \langle -R_{\lhd_{A^*}}^{*}(b^{*})x , a^{*} \right \rangle, \\
  \mathcal{B}_{d}(x \circ_{A\oplus A^*} b^{*},y) &\overset{\eqref{symin2}}{=} \mathcal{B}_{d}(b^{*} ,  y \lhd_{A} x -x \rhd_{A} y)  = \left \langle (L_{\rhd_{A}}^{*}-R_{\lhd_{A}}^{*})(x)b^{*} , y \right \rangle.
\end{align*}
So $x \circ_{A\oplus A^*} b^{*} =
(L_{\rhd_{A}}^{*}-R_{\lhd_{A}}^{*})(x)b^{*} -
R_{\lhd_{A^*}}^{*}(b^{*})x$. Similarly,
\vsb
\begin{equation*}
a^*\circ_{A\oplus A^*}y= (L_{\rhd_{A^{*}}}^{*}-R_{\lhd_{A^{*}}}^{*})(a^{*}) y - R_{\lhd_{A}}^{*}(y) a^{*},\quad [x,b^{*}]_{A\oplus A^*} =
\mathrm{ad}_{A}^{*}(x) b^{*} -\mathrm{ad}_{A^{*}}^{*}(b^{*}) x.
\end{equation*}
Therefore the post-Lie algebra multiplications $\circ_{A\oplus A^*}$
and $[-,-]_{A\oplus A^*}$  on $A \oplus  A^{*}$ are decided by
\vsb
\begin{small}
\begin{align}
(x+a^{*}) \circ_{A\oplus A^*} (y+b^{*}) &=   x \circ_A y+ (L_{\rhd_{A^{*}}}^{*}-R_{\lhd_{A^{*}}}^{*})(a^{*}) y- R_{\lhd _{A^{*}}}^{*}(b^{*}) x \nonumber\\
&\quad + a^{*} \circ_{A^*} b^{*}+ (L_{\rhd_{A}}^{*}-R_{\lhd_{A}}^{*})(x) b^{*} - R_{\lhd_{A}}^{*}(y) a^{*}, \mlabel{r7}\\
[x+a^{*}, y+b^{*}]_{A\oplus A^*} &= [x,
y]+\mathrm{ad}_{A^{*}}^{*}(a^{*}) y-\mathrm{ad}_{A^{*}}^{*}(b^{*})
x+[a^{*}, b^{*}]+\mathrm{ad}_{A}^{*}(x)
b^{*}-\mathrm{ad}_{A}^{*}(y) a^{*}.\mlabel{r8}
\end{align}
\end{small}
By Proposition \mref{mppstl}, $(A,
A^{*},L_{\rhd_{A}}^{*}-R_{\lhd_{A}}^{*},-R_{\lhd_{A}}^{*},
\mathrm{ad}_{A}^{*},
L_{\rhd_{A^{*}}}^{*}-R_{\lhd_{A^{*}}}^{*},-R_{\lhd_{A^{*}}}^{*}$,
$\mathrm{ad}_{A^{*}}^{*})$ is a matched pair of post-Lie algebras.

($\Longleftarrow$)  By Proposition~\mref{mppstl}, if
$(A, A^{*}\!,L_{\rhd_{A}}^{*}\!-R_{\lhd_{A}}^{*}\!,-R_{\lhd_{A}}^{*},
\mathrm{ad}_{A}^{*},
L_{\rhd_{A^{*}}}^{*}-R_{\lhd_{A^{*}}}^{*},-R_{\lhd_{A^{*}}}^{*}$,
$\mathrm{ad}_{A^{*}}^{*})$ is a matched pair of post-Lie algebras,
then there is a post-Lie algebra $(A \bowtie A^{*},\circ_{A\oplus
A^*}, [-,-]_{A\oplus A^*})$ with $\circ_{A\oplus A^*}$ and
$[-,-]_{A\oplus A^*}$ respectively defined by
Eqs.~(\mref{r7})--(\mref{r8}). It is straightforward to show that the nondegenerate
 symmetric bilinear form $\mathcal{B}_{d}$ defined by Eq.~(\mref{mtorbl}) is
 invariant on $(A \bowtie A^{*},\circ_{A\oplus
A^*}, [-,-]_{A\oplus A^*})$. Thus $\big((A \oplus
A^{*},\!\circ_{A \oplus   A^{*}}, \![-,-]_{A \oplus
A^{*}},\mathcal{B}_{d}),
 (A,\circ_A$, $[-,-]_A)$, $(A^{*}$, $\circ_{A^{*}},[-,-]_{A^{*}})\big)$ is a Manin triple  of post-Lie algebras
associated to the \mysymbolnew bilinear form. By
Lemma~\mref{lem:111}, with the compatible \mysymbol algebra
structure $(A\oplus A^*,\rhd_{A\oplus A^*},\lhd_{A\oplus
A^*},[-,-]_{A\oplus A^*})$ in which $\rhd_{A\oplus
A^*},\lhd_{A\oplus A^*}$ are respectively defined by
Eqs.~\eqref{comDpspl1} and \eqref{comDpspl2} through $\mathcal
B_d$, the two \mysymbol subalgebra structures on $A$ and $A^*$ are
exactly $(A,\rhd_A,\lhd_{A},[-,-]_A)$  and
$(A^{*},\rhd_{A^{*}},\lhd_{A^{*}}, [-,-]_{A^{*}})$ respectively.
\end{proof}

Recall that a {\bf Lie coalgebra} is a pair  $(\mathfrak{g}, \Delta)$,  where $\mathfrak{g}$  is a vector space and  $\Delta: \mathfrak{g} \rightarrow \mathfrak{g} \otimes \mathfrak{g}$  is a linear map satisfying the equations
\vsb
\begin{align*}
\Delta   + \tau \Delta &=0,\quad (\mathrm{id}+\sigma_{123}+\sigma_{123}^{2} )( \Delta \otimes \mathrm{id} ) \Delta   = 0,
\end{align*}
where  $\sigma_{123}(x \otimes y \otimes z) =y \otimes z \otimes x$  for all $x, y, z \in \mathfrak{g}$.
\vsb
\begin{defi}  A  \textbf{\mysymbol coalgebra} is a quadruple $(A,\delta_{\rhd},\delta_{\lhd}, \Delta)$  where  $(A, \Delta)$  is a Lie coalgebra and  $\delta_{\rhd},\delta_{\lhd}: A \rightarrow A \otimes A  $ are linear maps satisfying
the following equations.
\vsb
\begin{align}
(\mathrm{id} \otimes \Delta) \delta_{\lhd}  &= (\Delta \otimes  \mathrm{id}) \delta_{\lhd} +   (  \tau \otimes \mathrm{id}  )(\mathrm{id}  \otimes  \Delta   ) \delta_{\lhd},\mlabel{LDLC1}\\
(   \mathrm{id}\otimes   (\delta_{\lhd}+\tau \delta_{\lhd}))  \Delta  &=  (   \Delta \otimes  \mathrm{id} )(\delta_{\lhd} +\tau\delta_{\lhd} )    = 0,\mlabel{LDLC2}\\
 (\mathrm{id} \otimes \Delta)\delta_{\bullet} &= (\delta_{\circ}\otimes  \mathrm{id}) \Delta+ (\tau \otimes \mathrm{id}) (\mathrm{id} \otimes   \delta_{\bullet})\Delta,\mlabel{LDLC3}\\
(\mathrm{id}  \otimes  \delta_{\lhd} )\delta_{\rhd}  &= (  \delta_{\bullet}\otimes  \mathrm{id} )\delta_{\lhd}   + (\tau \otimes \mathrm{id})(\mathrm{id} \otimes \delta_{\circ} )\delta_{\lhd} -( \mathrm{id} \otimes \delta_{\lhd}) \Delta,\mlabel{LDLC4}\\
(\mathrm{id}  \otimes  \mathrm{id}  -   \tau\otimes \mathrm{id} )(\delta_{\circ}\otimes  \mathrm{id}) \delta_{\rhd} &= (\mathrm{id}  \otimes  \mathrm{id}  -   \tau\otimes \mathrm{id} )(\mathrm{id} \otimes   \delta_{\rhd})\delta_{\rhd} - (\Delta \otimes \mathrm{id}  )\delta_{\circ}\nonumber\\
&\quad+ (   \tau\otimes \mathrm{id} -\mathrm{id}  \otimes  \mathrm{id}  )(\mathrm{id} \otimes   \delta_{\lhd})\Delta,\mlabel{LDLC5}
\vsc
\end{align}
\vsc
where $\delta_{\circ}= \delta_{\rhd} + \delta_{\lhd}$,
$\delta_{\bullet}=  \delta_{\rhd} - \tau\delta_{\lhd} $.
\end{defi}

Let  $A$  be a finite-dimensional vector space and  $\delta: A \rightarrow A \otimes A$  be a linear map. Let $\cdot_{A^{*}}:  A^{*} \otimes A^{*} \rightarrow A^{*}$  be the linear dual of $\delta$, that is, the operation $\cdot_{A^{*}}$ is defined by
\begin{align}
\langle a^* \cdot_{A^{*}} b^*, x \rangle := \langle a^* \otimes b^*, \delta(x) \rangle, \quad \forall x \in A,a^*, b^* \in A^{*}.\mlabel{ldm}
\end{align}

\begin{pro}\mlabel{dpsplco}
Let  $A$  be a finite-dimensional vector space. Let
$\delta_{\rhd},\delta_{\lhd},\Delta: A \rightarrow A \otimes A$ be
linear maps, and  $\rhd_{A^{*}},\lhd_{A^{*}},[-,-]_{A^{*}}: A^{*}
\otimes A^{*} \rightarrow A^{*}$ be their linear duals.
Then  $(A,\delta_{\rhd},\delta_{\lhd},\Delta) $ is a  \mysymbol coalgebra if and only if
 $(A^{*}, \rhd_{A^{*}}, \lhd_{A^{*}}, [-,-]_{A^{*}}) $ is a   \mysymbol algebra.
\end{pro}

\begin{proof}
It is straightforward.
\end{proof}

\begin{defi} A  {\bf \mysymbol bialgebra} is a vector space  $A$  together with six linear maps
\begin{align*}
\rhd,\lhd,[-,-]: A \otimes A \rightarrow A, \quad \delta_{\rhd},\delta_{\lhd},\Delta: A \rightarrow A \otimes A
\end{align*}
satisfying the following conditions.
    \begin{enumerate}
    \item $(A,\rhd,\lhd, [-,-])$  is a  \mysymbol  algebra.
    \item $(A,\delta_{\rhd},\delta_{\lhd}, \Delta)$ is a  \mysymbol coalgebra.
    \item For all $x,y \in A$, the following equalities hold.
\begin{small}
\begin{align}
&\Delta ([x, y])=(\operatorname{ad}(x) \otimes \mathrm{id}+\mathrm{id} \otimes \operatorname{ad}(x)) \Delta (y)-(\operatorname{ad}(y) \otimes \mathrm{id}+\operatorname{id} \otimes \operatorname{ad}(y)) \Delta (x),\mlabel{Liebc}  \\
&\Delta(x \circ y)   =  \big( L_{\circ}(x) \otimes \mathrm{id} + \mathrm{id} \otimes  L_{\bullet}(x)\big)\Delta(y)+(\mathrm{id}\otimes \mathrm{ad}(y) + \mathrm{ad}(y)\otimes \mathrm{id}) \delta_{\lhd}(x), \mlabel{DPSPLB1} \\
&\Delta(x \bullet y)  = \big( L_{\bullet}(x) \otimes \mathrm{id} + \mathrm{id} \otimes L_{\bullet}(x)\big)\Delta(y)-(\mathrm{id} \otimes  \mathrm{ad}(y)) \tau\delta_{\lhd}(x)+ (\mathrm{ad}(y)\otimes \mathrm{id}) \delta_{\lhd}(x), \mlabel{DPSPLB2}\\
&\delta_{\bullet}([x,y]) = \big(\mathrm{id} \otimes \mathrm{ad}(x)\big)\delta_{\bullet}(y) - \big(\mathrm{id} \otimes \mathrm{ad}(y)\big)\delta_{\bullet}(x) + (R_{\lhd}(x) \otimes  \mathrm{id})\Delta(y) - (R_{\lhd}(y) \otimes  \mathrm{id})\Delta(x),\mlabel{DPSPLB3}\\
&\delta_{\circ}([x,y]) = \big(\mathrm{id} \otimes \mathrm{ad}(x)\big)\delta_{\circ}(y) - \big(\mathrm{id} \otimes \mathrm{ad}(y)\big)\delta_{\bullet}(x) + (R_{\lhd}(x) \otimes  \mathrm{id}) \Delta(y) + (L_{\lhd}(y) \otimes  \mathrm{id})\Delta(x),\mlabel{DPSPLB4}\\
&\delta_{\bullet}(x \circ y)  =   \big(\mathrm{id}\otimes L_{\circ}(x) +  (L_{\rhd}+\mathrm{ad})(x) \otimes \mathrm{id}  \big)\delta_{\bullet}(y) -\big(  R_{\lhd}(y)\otimes \mathrm{id}\big)\tau\delta_{\lhd}(x)+\big(\mathrm{id}\otimes R_{\circ} (y)\big)(\delta_{\rhd}+ \Delta)(x), \mlabel{DPSPLB5}\\
&\delta_{\circ}(x \bullet y)  =   \big(\mathrm{id}\otimes L_{\bullet}(x) +  (L_{\rhd}+\mathrm{ad})(x) \otimes \mathrm{id} \big)\delta_{\circ}(y) -\big(L_{\lhd}(y)\otimes \mathrm{id}\big)\delta_{\lhd}(x)+\big(\mathrm{id}\otimes R_{\bullet} (y)\big)(\delta_{\rhd}+ \Delta)(x), \mlabel{DPSPLB6}\\
&\delta_{\lhd}(\{x, y\}) = \big(\mathrm{id} \otimes L_{\bullet}(x) +  L_{\circ}(x) \otimes \mathrm{id} \big)\delta_{\lhd}(y) - \big(\mathrm{id} \otimes L_{\bullet}(y) + L_{\circ}(y) \otimes \mathrm{id} \big)\delta_{\lhd}(x), \mlabel{DPSPLB7}\\
&(\delta_{\circ}\!- \!\tau\delta_{\circ} \!+\! \Delta)(x \lhd y)\! =\!   ( \mathrm{id} \otimes L_{\lhd}(x))\delta_{\bullet}(y)\!+\!(\mathrm{id}\otimes R_{\lhd}(y))\delta_{\circ}(x) \!-\! (L_{\lhd}(x) \otimes
\mathrm{id})\tau\delta_{\bullet}(y)   \!-\! (R_{\lhd}(y) \otimes
\mathrm{id})\tau\delta_{\circ}(x),  \mlabel{DPSPLB8}
\end{align}
\end{small}
    \end{enumerate}
where $\{-,-\}$, $\circ$ and $\bullet$ are defined by
Eqs.~\eqref{eq:curbra},~\eqref{hpla} and \eqref{vpla}
respectively, $\delta_{\circ}:= \delta_{\rhd} + \delta_{\lhd}$,
$\delta_{\bullet}:=  \delta_{\rhd} - \tau\delta_{\lhd} $. We
denote a   \mysymbol bialgebra by $ (A,\rhd,\lhd,
[-,-],\delta_{\rhd},\delta_{\lhd}, \Delta)$.
\mlabel{d:ppplbialg}
    \end{defi}

\begin{rmk}\begin{enumerate}
\item Recall \mcite{CP} that a \textbf{Lie bialgebra}    is a
triple $(\mathfrak{g},[-,-], \Delta)$,  where
$(\mathfrak{g},[-,-])$  is a Lie algebra, $(\mathfrak{g}, \Delta)$
is a Lie coalgebra and Eq.~(\mref{Liebc}) holds. Thus for a
\mysymbol bialgebra  $ (A,\rhd,\lhd,
[-,-],\delta_{\rhd},\delta_{\lhd}, \Delta)$, $(A,[-,-], \Delta)$
is a Lie bialgebra.

\item   Recall \mcite{NB} that an {\bf L-dendriform bialgebra} $(A,\!
\rhd,\! \lhd,\! \delta_{\rhd},\! \delta_{\lhd})$ consists of an
L-dendriform algebra $(A,\rhd,\lhd)$ and an L-dendriform coalgebra
$(A,\delta_{\rhd},\delta_{\lhd})$ satisfying certain compatibility
conditions. Then for a  \mysymbol bialgebra $ (A,\rhd,\lhd,
[-,-],\delta_{\rhd},\delta_{\lhd}, \Delta)$, if $[-,-]=0 =
\Delta$, then $(A,\rhd,\lhd,\delta_{\rhd}, \delta_{\lhd})$ is an
L-dendriform bialgebra.
\end{enumerate}
\end{rmk}

\begin{pro} \mlabel{mpsplb}
 Let $ (A,  \rhd_A,   \lhd_A, [-,-]_A) $   and $ (A^*,  \rhd_{A^*},   \lhd_{A^*},[-,-]_{A^*}) $   be \mysymbol algebras.  Let linear maps $\delta_{\rhd_A}, \delta_{\lhd_A},\Delta_A: A \rightarrow A \otimes A$ be the linear duals of  $\rhd_{A^{*}}$, $\lhd_{A^{*}}$ and $[-,-]_{A^{*}}$ respectively. Then $\left(A, A^{*}, L_{\rhd_{A}}^{*}-R_{\lhd_{A}}^{*},-R_{\lhd_{A}}^{*},\mathrm{ad}_{A}^{*}, L_{\rhd_{A^{*}}}^{*}-R_{\lhd_{A^{*}}}^{*}, -R_{\lhd_{A^{*}}}^{*},\mathrm{ad}_{A^{*}}^{*}\right)$  is a matched pair of post-Lie algebras if and only if  $(A, \rhd_A, \lhd_A, [-,-]_A,\delta_{\rhd_A},\delta_{\lhd_A},\Delta_A)$ is a  \mysymbol bialgebra.
\end{pro}
\begin{proof}
By Proposition \mref{dpsplrep}, $(A^{*};L_{\rhd_{A}}^{*}-
R_{\lhd_{A}}^{*}, R_{\lhd_{A}}^{*},\mathrm{ad}_{A}^{*})$ is a
representation of   $(A, \circ_{A},[-,-]_{A})$ if and only if $(A,
\rhd_{A},\lhd_{A},[-,-]_{A})$ is a  \mysymbol algebra. In addition,
by Proposition \mref{dpsplco},  $( A;L_{\rhd_{A^{*}}}^{*}-
R_{\lhd_{A^{*}}}^{*},
R_{\lhd_{A^{*}}}^{*},\mathrm{ad}_{A^{*}}^{*})$  is a
representation of  $(A^*, \circ_{A^{*}},[-,-]_{A^{*}})$ if and
only if $(A,\delta_{\rhd_A}, \delta_{\lhd_A},\Delta_A)$ is a
\mysymbol coalgebra. Moreover, let $ (A,  \circ_A, [-,-]_A) $ and
$ (A,  \bullet_A, [-,-]_A) $ be the horizontal and vertical
post-Lie algebras of $ (A,  \rhd_A, \lhd_A, [-,-]_A) $
respectively. Then for
 $x, y \in A$, $a^{*}, b^{*} \in A^{*}$, we have
\begin{align*}
 \left\langle \mathrm{ad}_{A^*}^{*}(a^*)(x \circ_A y), b^*\right\rangle & =\left\langle a^* \otimes b^*, -\Delta_A(x \circ_A y  ) \right\rangle ;\\
 \left\langle x \circ_A \mathrm{ad}_{A^*}^{*}(a^*)(y), b^*\right\rangle & =\left\langle a^* \otimes b^*, -\big(  \mathrm{id} \otimes L_{\circ_A} (x)\big )\Delta_A(y) \right\rangle ;\\
\left\langle -[y,R^*_{\lhd_{A^*}}(a^*)x]_A, b^*\right\rangle & =\left\langle a^* \otimes b^*, -(  \mathrm{id} \otimes \mathrm{ad}_{A}(y) )\delta_{\lhd_{A}}(x)\right\rangle;\\
\left\langle -\mathrm{ad}_{A^*}^{*}\big((L^*_{\rhd_{A}}-R^*_{\lhd_{A}})(x)a^* \big)(y), b^*\right\rangle & =\left\langle a^* \otimes b^*, -\big( L_{\bullet_A}  (x)\otimes \mathrm{id}\big) \Delta_A(y)\right\rangle;\\
\left\langle
R_{\lhd_{A^*}}^{*}\big(\mathrm{ad}^*_{A}(y)a^*\big)(x),
b^*\right\rangle & =\left\langle a^* \otimes b^*,  -(
\mathrm{ad}_{A}(y) \otimes\mathrm{id}   )
\delta_{\lhd_{A}}(x)\right\rangle.
\end{align*}
In Eqs.~(\mref{mplo})--(\mref{mppl6}), take
\vsb
 \begin{align*}
 l_A =  L_{\rhd_{A}}^{*}-R_{\lhd_{A}}^{*},\; r_A = -R_{\lhd_{A}}^{*},\;
 \rho_{A}  = \mathrm{ad}_{A}^{*}, \; l_B =
 L_{\rhd_{A^{*}}}^{*}-R_{\lhd_{A^{*}}}^{*},  \;   r_B =
 -R_{\lhd_{A^{*}}}^{*},  \;  \rho_{B} = \mathrm{ad}_{A^{*}}^{*}.
 \end{align*}
Then Eq.~(\mref{mppl1}) holds if and only if  Eq.~(\mref{DPSPLB1}) holds. Similarly, we have
\vsb
\begin{align*}
\text{Eq.~}(\mref{mplo}) &\Longleftrightarrow \text{Eq.~}(\mref{mpl1}) \Longleftrightarrow \text{Eq.~}(\mref{Liebc}), \hspace{2.4cm} \text{Eq.~}(\mref{mppl4})  \Longleftrightarrow \text{Eq.~}(\mref{DPSPLB2}),\\
  \text{Eq.~}(\mref{mppl2})  &\Longleftrightarrow  \text{Eq.~}(\mref{DPSPLB3}),\quad \text{Eq.~}(\mref{mppl3})\Longleftrightarrow \text{Eq.~}(\mref{DPSPLB4}) ,\quad \text{Eq.~}(\mref{mppl5}) \Longleftrightarrow \text{Eq.~}(\mref{DPSPLB5}),\\
 \text{Eq.~}(\mref{mppl7})  &\Longleftrightarrow \text{Eq.~}(\mref{DPSPLB6}),\quad \text{Eq.~}(\mref{mppl6})  \Longleftrightarrow \text{Eq.~}(\mref{DPSPLB7}),\quad \text{Eq.~}(\mref{mppl8})  \Longleftrightarrow \text{Eq.~}(\mref{DPSPLB8}).
\end{align*}
This completes the proof.
\end{proof}
\vsb
Combining Propositions \mref{mtsplb} and \mref{mpsplb}, we obtain the following conclusion.

\begin{thm}\mlabel{djsplb}
Let $(A,\rhd_A,\lhd_{A}, [-,-]_A)$  and
$(A^{*},\rhd_{A^{*}},\lhd_{A^{*}}, [-,-]_{A^{*}})$ be
\mysymbol algebras, and let their associated horizontal post-Lie
algebras be $(A,\circ_A$, $[-,-]_A)$ and $(A^{*}$,
$\circ_{A^{*}}$, $[-,-]_{A^{*}})$ respectively.
 Let linear maps $ \delta_{\rhd_A}, \delta_{\lhd_A},\Delta_A: A \rightarrow A \otimes A$ be the linear duals of  $\rhd_{A^{*}}$,  $\lhd_{A^{*}}$ and $[-,-]_{A^{*}}$ respectively.  Then the following conditions are
 equivalent.
\begin{enumerate}
\item There is a Manin triple  $\big((A \oplus   A^{*},\!\circ_{A
\oplus A^{*}}, \![-,-]_{A \oplus   A^{*}},\mathcal{B}_{d}),
 (A,\circ_A$, $[-,-]_A)$, $(A^{*}$, $\circ_{A^{*}}$, $[-,-]_{A^{*}})\big)$
 of post-Lie algebras associated to  the invariant
bilinear form $\mathcal{B}_{d}$ such that the compatible \mysymbol algebra structure
$(A\oplus A^*\!,\!\rhd_{A\oplus A^*},\lhd_{A\oplus A^*},[-,-]_{A\oplus
A^*})$ with $\rhd_{A\oplus A^*}$, $\lhd_{A\oplus A^*}$ respectively
defined by Eqs.~\eqref{comDpspl1} and \eqref{comDpspl2} through
$\mathcal B_d$ contains $(A,\rhd_A$, $\lhd_{A}$, $[-,-]_A)$  and
$(A^{*},\rhd_{A^{*}},\lhd_{A^{*}}, [-,-]_{A^{*}})$  as \mysymbol
subalgebras.

\item $\left(A, A^{*},
L_{\rhd_{A}}^{*}\!-\!R_{\lhd_{A}}^{*},\!-R_{\lhd_{A}}^{*},\mathrm{ad}_{A}^{*},
L_{\rhd_{A^{*}}}^{*}\!-\!R_{\lhd_{A^{*}}}^{*},
\!-R_{\lhd_{A^{*}}}^{*},\mathrm{ad}_{A^{*}}^{*}\right)$  is a
matched pair of post-Lie algebras;
    \item $(A, \rhd_A, \lhd_A, [-,-]_A,\delta_{\rhd_A},\delta_{\lhd_A},\Delta_A)$ is a  \mysymbol bialgebra.
\end{enumerate}
\end{thm}

\subsection{A class of \mysymbol bialgebras and \pdsCYBE}\

Let $(A, \rhd, \lhd,[-,-])$  be a   \mysymbol algebra and  $r=\sum\limits_{i} a_i \otimes b_i \in A \otimes A$. Set
\vsc
\begin{align}
\mathbf{C}(r) &:= [r_{12}, r_{13}]+[r_{12}, r_{23}]+[r_{13}, r_{23}],\mlabel{CYBE} \\
\mathbf{D}(r) &:=   r_{13} \lhd r_{12} + r_{12}\bullet r_{23} +
r_{13} \circ r_{23},  \mlabel{LDCYBE1}
\end{align}
where  $\circ$ and $\bullet$ are defined by
Eqs.~\eqref{hpla} and \eqref{vpla}
respectively, and
\vsb
\begin{small}
\begin{align*}
\left[r_{12}, r_{13}\right]&:=\sum_{i, j}\left[a_{i}, a_{j}\right]
\!\otimes\! b_{i} \!\otimes\! b_{j},\; \left[r_{12}, r_{23}\right]:=\sum_{i, j} a_{i}
\!\otimes\! \left[b_{i}, a_{j}\right] \!\otimes\! b_{j},\; \left[r_{13}, r_{23}\right]:=\sum_{i, j} a_{i}
\!\otimes\! a_{j} \!\otimes\! \left[b_{i}, b_{j}\right],\\
r_{13} \lhd r_{12} &:=\sum_{i, j} a_{i} \lhd  a_{j}
\!\otimes\! b_{j} \!\otimes\! b_{i},\; r_{12} \bullet  r_{23} :
=\sum_{i, j}  a_{i} \!\otimes\! b_{i}\bullet a_{j} \!\otimes\!
b_{j},\; r_{13} \circ r_{23}:=\sum_{i, j} a_{i} \!\otimes\!
a_{j}\!\otimes\!  b_{i} \circ b_{j}.
\end{align*}
\end{small}
\vse 
\begin{defi}Let $(A, \rhd, \lhd,[-,-])$  be a   \mysymbol algebra and  $r\in A \otimes A$.   $r$ is called a solution of the {\bf \pdsCYBE (\PCYBE)}
    in $(A, \rhd, \lhd,[-,-])$ if $r$ satisfies $\mathbf{C}(r)=
    \mathbf{D}(r)=0$.
\end{defi}
\vsd
\begin{rmk}
    Note that $\mathbf{C}(r)=0$ is precisely the {\bf classical
        Yang-Baxter equation (CYBE)} in the Lie algebra $(A,[-,-])$.
\end{rmk}
\vsb
For convenience, we define a binary operation $\diamond$ in  $(A, \rhd, \lhd,[-,-])$  by
\vsb
\begin{equation}
x \diamond y := x \lhd y + x \rhd y -  y \lhd x - y \rhd x, \quad \forall x,y\in A. \label{new:symb}
\end{equation}
\begin{pro}\mlabel{cldlbiaco}
Let $(A, \rhd, \lhd,[-,-])$  be a  \mysymbol algebra  and $r=\sum\limits_{i} a_i \otimes b_i \in A \otimes A$. Define binary operations $\diamond$, $\circ$ and $\bullet$   by
Eqs.~\eqref{new:symb}, \eqref{hpla} and \eqref{vpla}
respectively and linear maps $ \delta_{\rhd,r},
\delta_{\lhd,r},\Delta_r: A \rightarrow A \otimes A$  by
\vsb
\begin{align}
{\delta_{\rhd,r}}(x) := E(x)r ,\quad  {\delta_{\lhd,r}}(x):=
F(x)(-r),\quad  \Delta_r(x) := G(x)r, \quad \forall x \in
A,\mlabel{EFG}
\end{align}
where the linear maps  $E,F, G: A \rightarrow
\operatorname{End}_{\mathbb F}(A \otimes A)$ are defined
respectively by
\vsb
\begin{align}
E(x)&:= L_{\rhd}(x) \otimes \mathrm{id} + \mathrm{id} \otimes  L_{\diamond}(x), \mlabel{eq:e}\\
F(x)&:= L_{\circ}(x) \otimes \mathrm{id} + \mathrm{id} \otimes L_{\bullet}(x) ,\mlabel{eq:f}\\
G(x)&:= \mathrm{ad}(x) \otimes \mathrm{id} + \mathrm{id} \otimes \mathrm{ad}(x) , \quad \forall x \in A . \mlabel{eq:g}
\end{align}
\begin{enumerate}
    \item \mlabel{it:q1} $(A,\Delta_r)$ is a Lie coalgebra if and only if the following equations hold.
    \vsb
    \begin{align}
&G(x)(r+\tau(r))=0, \mlabel{qclb1}\\
&(\operatorname{ad}(x) \otimes \mathrm{id} \otimes
\mathrm{id}+\mathrm{id} \otimes \operatorname{ad}(x) \otimes
\mathrm{id} +\mathrm{id} \otimes \mathrm{id} \otimes
\operatorname{ad}(x))(\mathbf{C}(r))=0,\;\;\forall x\in
A.\mlabel{qclb2}
    \end{align}
    \item \mlabel{it:q2} Eq.~\eqref{LDLC1} holds if and only if the following equation holds.
    \vsb
\begin{equation}
   \begin{split}
&(L_{\circ}(x) \otimes \mathrm{id} \otimes \mathrm{id}+\mathrm{id} \otimes L_{\circ}(x) \otimes  \mathrm{id}+ \mathrm{id} \otimes \mathrm{id} \otimes L_{\bullet}(x))\left(\mathbf{C}(r)\right)\\
&+\sum_i \big((\operatorname{ad}(a_i) \otimes \mathrm{id})\tau
F(x)(r + \tau(r))\big) \otimes b_i = 0,\;\;\forall x\in A.
\end{split}
\mlabel{qcldl1}
\end{equation}
        \item  Eq.~\eqref{LDLC2} holds if and only if the following
        equations hold.
        \vsb
\begin{eqnarray}
&& (\mathrm{ad}(x) \otimes \mathrm{id}\otimes \mathrm{id}) (\mathrm{id} \otimes \mathrm{id}\otimes \mathrm{id} + \sigma_{23})\Big(\sum_{i}\big((a_i \otimes F(b_i)(r + \tau(r))\big) - \mathbf{D}(r)\Big) \nonumber \\
&&  +\sum_{i}a_i \otimes F([x,b_i])(r + \tau(r)) = 0,\mlabel{qcldl2}\\
&&(  \mathrm{id}\otimes \mathrm{id}\otimes (L_{\lhd}+R_{\lhd})(x)
)(\mathbf{C}(r)) = 0,\;\;\forall x\in A. \mlabel{qcldl3}
   \end{eqnarray}
    \item  Eq.~\eqref{LDLC3} holds if and only if the following equation
    holds.
    \vsb
   \begin{equation}
\mlabel{qcldl4}
    \begin{split}
&(L_{\lhd}(x) \otimes  \mathrm{id} \otimes \mathrm{id})(\mathbf{C}(r))
 + ( \mathrm{id} \otimes \operatorname{ad}(x) \otimes \mathrm{id}  )\Big(  \sigma_{23}(\mathbf{D}(r)) -\sum_{i}a_i \otimes F(b_i)(r + \tau(r)) \Big) \\
&- ( \mathrm{id} \otimes \mathrm{id} \otimes \operatorname{ad}(x)
)(\mathbf{D}(r)) -\sum_i  \big((R_{\lhd}(a_i) \otimes
\mathrm{id})G(x)(r + \tau(r))\big) \otimes b_i = 0,\;\;\forall x\in A.
    \end{split}
    \end{equation}
    \item  Eq.~\eqref{LDLC4} holds if and only if the following equation
    holds.
    \vsb
    \begin{equation}
  \begin{split}
&((\ad+L_{\lhd})(x) \otimes  \mathrm{id} \otimes \mathrm{id} +  \mathrm{id}  \otimes L_{\circ}(x) \otimes \mathrm{id} )\Big( \sum_{i}\big(a_i \otimes F(b_i)(r + \tau(r)) \big) -\sigma_{23}(\mathbf{D}(r)) \Big)\\
&+ (  \mathrm{id}  \otimes \mathrm{id} \otimes L_{\bullet}(x) )\Big( \sum_{i}\big(a_i \otimes F(b_i)(r + \tau(r)) -\tau F(a_i)(r + \tau(r)) \otimes b_i\big) -\sigma_{23}(\mathbf{D}(r)) \Big)\\
&+\!\sum_{i}\Big(\big((R_{\lhd}(a_i) \!\otimes \!\mathrm{id}) \tau F(x)(r +
\tau(r))\big)\!\otimes\! b_i -  \tau F(x \circ a_i)(r + \tau(r))  \!\otimes\!
b_i\Big)= 0,\; \forall x\in A.
\end{split}
\mlabel{qcldl5}
\end{equation}
    \item  Eq.~\eqref{LDLC5} holds if and only if the following equation holds.
    \vsb
{\small
         \begin{equation}
\begin{split}
&(\mathrm{id}  \otimes \mathrm{id} \otimes \mathrm{id} - \tau \otimes \mathrm{id})(\mathrm{ad}(x)\otimes \mathrm{id}  \otimes \mathrm{id})\Big(\sum_{i}\big(a_i \otimes F(b_i)(r + \tau(r))\big) - \sigma_{23}(\mathbf{D}(r))\Big)\\
&+\sum_{i}\Big( \big(\mathrm{id} \otimes R_{\rhd}(a_i) \big) E(x)(r + \tau(r)) \Big)\otimes b_i+\Big( \big(\mathrm{id} \otimes R_{\circ}(a_i) \big) G(x)(r + \tau(r))\Big) \otimes b_i\\
&+(\mathrm{id}  \otimes \mathrm{id} \otimes \mathrm{id} - \tau \otimes \mathrm{id}) \big(\mathrm{id}  \otimes \mathrm{id}  \otimes L_{\diamond}(x) \big) \big(\mathbf{D}(r)\big)- (\mathrm{id}  \otimes \mathrm{id} \otimes R_{\bullet}(x) )(\mathbf{C}(r))\\
&+(\mathrm{id}  \otimes \mathrm{id} \otimes \mathrm{id} - \tau
\otimes \mathrm{id}) (L_{\rhd}(x)\otimes \mathrm{id}  \otimes
\mathrm{id} )  \big(\mathbf{D}(r)-( \tau  \otimes
\mathrm{id}) (\mathbf{D}(r))  \big) =0,\;\;\forall x\in
A,
\end{split}
\mlabel{qcldl6}
    \end{equation}
}
    \end{enumerate}
where $\sigma_{23}(x \otimes y
\otimes z) = x \otimes z \otimes y$ for all $x,y,z \in A$.
    \end{pro}
    \begin{proof}
We only prove (\mref{it:q2}). The others are obtained similarly.  Let $x \in A$. Then we have
      \begin{small}
\begin{eqnarray*}
&&(\mathrm{id} \otimes \Delta_r) \delta_{\lhd,r}(x) -(\Delta_{r} \otimes  \mathrm{id}) \delta_{\lhd,r}(x) -  (\tau \otimes \mathrm{id}) (\mathrm{id} \otimes   \Delta_r)\delta_{\lhd,r}(x)\\
&=&\sum_{i,j} - x \circ a_i \otimes [b_i, a_j] \otimes b_j -  x \circ a_i \otimes a_j \otimes [b_i, b_j] - a_i \otimes [x \bullet b_i, a_j] \otimes b_j - a_i \otimes a_j \otimes  [x \bullet b_i, b_j]\\
& & +  a_j \otimes  [x \circ  a_i,b_j] \otimes  b_i    +   [x \circ  a_i,a_j] \otimes b_j  \otimes  b_i  +   a_j  \otimes  [a_i,b_j]  \otimes x \bullet b_i + [a_i,a_j]   \otimes  b_j     \otimes x \bullet  b_i\\
& & +  [b_i, a_j]  \otimes x \circ a_i\otimes b_j + a_j \otimes   x \circ a_i  \otimes [b_i, b_j] + [x \bullet b_i, a_j] \otimes a_i  \otimes b_j +  a_j \otimes a_i \otimes  [x \bullet b_i, b_j]\\
&=& B_1+B_2+B_3+B_4,
\end{eqnarray*}
  \end{small}
    where
    \begin{small}
    \begin{align*}
B_1 &= \sum_{i,j}-x \circ a_i \otimes [b_i, a_j] \otimes b_j -  x \circ a_i \otimes a_j \otimes [b_i, b_j]   +   [x \circ  a_i,a_j] \otimes b_j  \otimes  b_i  \\
&=\sum_{i,j}-(L_{\circ}(x) \otimes  \mathrm{id} \otimes  \mathrm{id})([r_{12},r_{23}] + [r_{13},r_{23}]) + x \circ [a_i,a_j]  \otimes  b_j  \otimes b_i -  [a_i ,x \circ a_j] \otimes  b_j  \otimes b_i\\
&=\sum_{i,j}-(L_{\circ}(x) \otimes  \mathrm{id} \otimes  \mathrm{id})([r_{12},r_{23}] + [r_{13},r_{23}] + [r_{12},r_{13}] ) -  [a_i ,x \circ a_j] \otimes  b_j  \otimes b_i\\
&=-(L_{\circ }(x) \otimes  \mathrm{id} \otimes  \mathrm{id})(\mathbf{C}(r) )- \sum_{i}\big((\operatorname{ad}(a_i)  \otimes \mathrm{id}) ( L_{\bullet}(x)\otimes \mathrm{id})r\big) \otimes b_i, \\
B_2 &=\sum_{i,j}  a_j  \otimes  [a_i,b_j]  \otimes x \bullet b_i-a_i \otimes a_j \otimes  [x \bullet b_i, b_j]  + [a_i,a_j]   \otimes  b_j     \otimes x \bullet  b_i+  a_j \otimes a_i \otimes  [x \bullet b_i, b_j]\\
&=-(\mathrm{id} \otimes \mathrm{id} \otimes  L_{\bullet}(x))(\mathbf{C}(r) ),\\
B_3 &=\sum_{i,j}  a_j \otimes  [x \circ  a_i,b_j] \otimes  b_i- a_i \otimes [x \bullet b_i, a_j] \otimes b_j  + [b_i, a_j]  \otimes x \circ a_i\otimes b_j + a_j \otimes   x \circ a_i  \otimes [b_i, b_j]\\
&=\sum_{i,j}  a_i \otimes  x \bullet [  a_j,b_i] \otimes  b_j + [b_i, a_j]  \otimes x \circ a_i\otimes b_j + a_j \otimes   x \circ a_i  \otimes [b_i, b_j]\\
&=\sum_{i}  -(\mathrm{id}\otimes  L_{\circ}(x)\otimes \mathrm{id} ) (\mathbf{C}(r) ) -  \big((\operatorname{ad}(a_i)  \otimes \mathrm{id}) ( \mathrm{id} \otimes  L_{\circ}(x))(r+\tau (r))\big) \otimes b_i,\\
B_4 &=\sum_{i,j} [x \bullet b_i, a_j] \otimes a_i  \otimes b_j=- \sum_{i} \big((\operatorname{ad}(a_i)  \otimes \mathrm{id}) ( L_{\bullet}(x)\otimes \mathrm{id})\tau (r)\big) \otimes b_i.
\vsc
\end{align*}
\end{small}
Thus we get
\vsb
\begin{align*}
B_1+B_2+B_3+B_4
&  =-(L_{\circ}(x) \otimes \mathrm{id} \otimes \mathrm{id}+\mathrm{id} \otimes L_{\circ}(x) \otimes  \mathrm{id}+ \mathrm{id} \otimes \mathrm{id} \otimes L_{\bullet}(x))\left(\mathbf{C}(r)\right)\nonumber\\
&\quad-\sum_i \big((\operatorname{ad}(a_i) \otimes \mathrm{id})\tau F(x)(r + \tau(r))\big) \otimes b_i.
    \end{align*}
Hence the conclusion follows.
    \end{proof}

\begin{pro}\mlabel{csplbiaxr}
Let $(A, \rhd, \lhd,[-,-])$  be a   \mysymbol algebra and $r \in A
\otimes A$. Define linear maps $\delta_{\rhd,r},
\delta_{\lhd,r}, \Delta_r : A \rightarrow A \otimes A$   by
Eq.~\eqref{EFG}.
\vsa
 \begin{enumerate}
    \item Eq.~\eqref{DPSPLB1} holds.
    \item \mlabel{it:w2} Eq.~\eqref{DPSPLB2} holds if and only if the following equation
    holds.
    \vsb
        \begin{align}
    (\mathrm{ad}(x) \otimes \mathrm{id})F(y)(r + \tau(r)) = 0, \quad \forall x,y\in A. \mlabel{cldl1}
    \end{align}
    \item Eq.~\eqref{DPSPLB3} holds if and only if the following equation
    holds.
   \vsb
        \begin{align}
\big(F([x,y])+  (\mathrm{ad}(x)  \otimes  \mathrm{id})F(y) -    ( \mathrm{ad}(y)\otimes \mathrm{id}) F(x) \big)(r + \tau(r)) = 0, \quad \forall x,y\in A. \mlabel{cldl2}
    \end{align}
    \item Eq.~\eqref{DPSPLB4}  holds if and only if  Eq.~\eqref{cldl1} holds.
    \item Eq.~\eqref{DPSPLB5} holds if and only if the following equation
    holds.
    \vsb
  \begin{equation}
\begin{split}
\big(F(x \circ y) + ( \mathrm{id} \otimes L_{\circ}(x) & + (\mathrm{ad}+L_{\rhd})   (x)\otimes \mathrm{id} )F(y)\\
&-(R_{\lhd}(y) \otimes \mathrm{id}) \tau  F(x)  \big)(r + \tau(r)) = 0,    \quad \forall x,y\in A,
\end{split}
\mlabel{cldl3}
\end{equation}
where  $\circ$ is defined by Eq.~\eqref{hpla}.
\item Eqs.~\eqref{DPSPLB6}--\eqref{DPSPLB7} hold.
    \item Eq.~\eqref{DPSPLB8} holds if and only if the following equation
    holds.
\begin{equation}
    \begin{split}
\big(E(x \lhd y)-F(x \lhd y) &+ ( \mathrm{id} \otimes \mathrm{id} -  \tau )( L_{\lhd}(x)\otimes \mathrm{id})E(y) + G(x) \\
&+( \mathrm{id}\otimes R_{\lhd}(y))  (F(x) -E(x) )\big)(r + \tau(r)) = 0,    \quad \forall x,y\in A.
\end{split}
\mlabel{cldl4}
\end{equation}
    \end{enumerate}
\end{pro}
    \begin{proof}
We only prove (\mref{it:w2}). The others are obtained similarly. Let $x,y \in A$. Then we have
\vsb
\begin{alignat*}{4}
\Delta_r(x \bullet y) =\big((L_{\bullet}(x)\mathrm{ad}(y)- \mathrm{ad}(y)L_{\bullet}(x))\otimes   \mathrm{id}  
&+  \mathrm{id} \otimes (L_{\bullet}(x)\mathrm{ad}(y) - \mathrm{ad}(y)L_{\bullet}(x) )\big)r,\\
 \big( L_{\bullet}(x) \otimes \mathrm{id} + \mathrm{id} \otimes L_{\bullet}(x)\big)\Delta_r(y) &= (L_{\bullet}(x)\otimes   \mathrm{ad}(y)  +  L_{\bullet}(x)\mathrm{ad}(y) \otimes   \mathrm{id}  \\
 & \quad +  \mathrm{id} \otimes  L_{\bullet}(x)\mathrm{ad}(y) +\mathrm{ad}(y)  \otimes L_{\bullet}(x) ) r,  \\
    -(\mathrm{id} \otimes  \mathrm{ad}(y)) \tau\delta_{\lhd,r}(x) &=( \mathrm{id} \otimes \mathrm{ad}(y) L_{\circ}(x) +  L_{\bullet}(x) \otimes \mathrm{ad}(y)) \tau(r), \\
      (\mathrm{ad}(y)\otimes \mathrm{id}) \delta_{\lhd,r}(x)& = ( \mathrm{ad}(y) L_{\circ}(x) \otimes  \mathrm{id}  +   \mathrm{ad}(y) \otimes L_{\bullet}(x) ) (-r),
\end{alignat*}
where  $\circ$ and $\bullet$ are defined by Eqs.~\eqref{hpla} and \eqref{vpla} respectively. Thus
\vsb
\begin{eqnarray*}
 &&\Delta_r(x \bullet y) - \big( L_{\bullet}(x) \otimes \mathrm{id} + \mathrm{id} \otimes L_{\bullet}(x)\big)\Delta_r(y)+(\mathrm{id} \otimes  \mathrm{ad}(y)) \tau\delta_{\lhd,r}(x) -  (\mathrm{ad}(y)\otimes \mathrm{id}) \delta_{\lhd,r}(x)\\
&&= - (\mathrm{id} \otimes  \mathrm{ad}(y)L_{\bullet}(x)  ) r -(L_{\bullet}(x)\otimes   \mathrm{ad}(y))r - ( \mathrm{id} \otimes \mathrm{ad}(y) L_{\circ}(x) +  L_{\bullet}(x) \otimes \mathrm{ad}(y)) \tau(r)\\
&&= - (\mathrm{id} \otimes  \mathrm{ad}(y)L_{\circ}(x)  ) (r + \tau(r)) -(L_{\bullet}(x)\otimes   \mathrm{ad}(y))(r + \tau(r))\\
&&= -(\mathrm{id} \otimes \mathrm{ad}(y))\tau F(x)(r + \tau(r)).
\end{eqnarray*}
Hence the conclusion follows.
\vsc    \end{proof}

Combining Propositions \mref{cldlbiaco} and \mref{csplbiaxr} yields the following result.
\vsb
\begin{thm}\mlabel{quasildlb}
Let  $(A, \rhd, \lhd,[-,-])$  be a  \mysymbol algebra and $r\in
A\otimes A$. Define linear maps $\delta_{\rhd,r},
\delta_{\lhd,r}, \Delta_r, : A \rightarrow A \otimes A$   by
Eq.~\eqref{EFG}.
\vsb
    \begin{enumerate}
    \item    $(A, \delta_{\rhd,r}, \delta_{\lhd,r}, \Delta_r)$ is a  \mysymbol coalgebra if and only if  Eqs.~\eqref{qclb1}--\eqref{qcldl6} hold.
    \item \mlabel{it:e2}   $(A, \rhd, \lhd,[-,-],\delta_{\rhd,r}, \delta_{\lhd,r}, \Delta_r)$ is a  \mysymbol bialgebra if and only if  Eqs.~\eqref{qclb1}--\eqref{cldl4} hold.
    \end{enumerate}
\end{thm}
As a direct consequnce of Theorem~\mref{quasildlb}, we have
\begin{cor}\mlabel{iartoldlbia}
Let  $(A, \rhd, \lhd,[-,-])$  be a  \mysymbol algebra and $r\in
A\otimes A$. Define linear maps $\delta_{\rhd,r},
\delta_{\lhd,r}, \Delta_r, : A \rightarrow A \otimes A$   by
Eq.~\eqref{EFG}. If $r$ is a solution of the \PCYBE satisfying
\vsb
        \begin{align}
   E(x)(r+\tau(r))=0,\quad F(x)(r+\tau(r))=0,  \quad G(x)(r+\tau(r))=0, \quad \forall x \in A, \mlabel{invldl}
        \end{align}
for the maps $E, F, G$ defined in Eqs.~\meqref{eq:e}--\meqref{eq:g}, then $(A, \rhd, \lhd,[-,-],\delta_{\rhd,r}, \delta_{\lhd,r},
\Delta_r)$ is a  \mysymbol bialgebra. 
In particular, if $r$ is an antisymmetric solution of the \PCYBE,
then $(A, \rhd, \lhd,[-,-],\delta_{\rhd,r}$,
$\delta_{\lhd,r}$, $\Delta_r)$ is a  \mysymbol bialgebra. 
\end{cor}

Let $ V$  be a vector space and $r \in V \otimes V$. Define a linear map  $\tilde{r}: V^{*} \longrightarrow V$  by
\vsb
\begin{align*}
\left\langle \tilde{r}\left(u^{*}\right), v^{*}\right\rangle=\left\langle r, u^{*} \otimes v^{*}\right\rangle,     \quad \forall u^{*}, v^{*} \in V^{*}.
\end{align*}
\begin{pro}\mlabel{qLDCYBEO2}
Let $(A, \rhd, \lhd,[-,-])$  be a   \mysymbol algebra  and $r \in A \otimes A$ be antisymmetric. Then  $r$ is a solution of the
\PCYBE in $(A, \rhd, \lhd,[-,-])$ if and only if  $\tilde{r}$
satisfies the following equations.
\vsb
\begin{align}
\widetilde{r}\left(a^{*}\right) \rhd  \widetilde{r}\left(b^{*}\right)&=\widetilde{r}\big(L_{\diamond }^{*}(\tilde{r}(a^{*})) b^{*}+R_{\rhd }^{*}(\tilde{r}(b^{*})) a^{*}\big)  ,\mlabel{arod1}\\
\widetilde{r}\left(a^{*}\right) \lhd  \widetilde{r}\left(b^{*}\right) &=\widetilde{r}\big(R_{\bullet }^{*}(\tilde{r}(a^{*})) b^{*} - R_{\circ }^{*}(\tilde{r}(b^{*})) a^{*}\big)  ,\mlabel{arod2}\\
[\widetilde{r}\left(a^{*}\right) ,\widetilde{r}\left(b^{*}\right) ]  &=\widetilde{r}\big(\operatorname{ad}^{*}(\tilde{r}(a^{*})) b^{*}-\operatorname{ad}^{*}(\tilde{r}(b^{*})) a^{*}\big),\quad \forall a^*,b^* \in A ^*,\mlabel{arod3}
\end{align}
where  $\diamond$, $\circ$ and $\bullet$ are defined by
Eqs.~\eqref{new:symb}, \eqref{hpla} and \eqref{vpla}
respectively.
\end{pro}

\begin{proof}
Let $r=\sum\limits_{i} a_i \otimes b_i$. For $a^{*},b^{*},c^{*} \in A^*$, we have
\vsb
{\small
\begin{eqnarray*}
\left\langle \tilde{r}\left(a^{*}\right) \rhd \tilde{r}\left(b^{*}\right), c^{*}\right\rangle &= & \Big\langle  a^{*}  \otimes b^{*} \otimes c^{*},\sum_{i,j} a_j \otimes a_i \otimes b_j \rhd  b_i\Big\rangle
=: \left\langle  a^{*}  \otimes b^{*} \otimes c^{*},r_{13}\rhd  r_{23}\right\rangle;\\
\left\langle
\tilde{r}(L_{\diamond}^{*}\left(\tilde{r}\left(a^{*}\right)\right)
b^{*}), c^{*}\right\rangle &=& \Big\langle  a^{*}  \otimes b^{*}
\otimes c^{*},\sum_{i,j} a_j \otimes a_i\diamond b_j \otimes b_i
\Big\rangle
=: \left\langle   a^{*} \otimes b^{*}\otimes c^*  ,  r_{23} \diamond r_{12} \right\rangle; \\
 \left\langle \tilde{r}(R_{\rhd}^{*}\left(\tilde{r} \left(b^{*}\right)\right) a^{*}), c^{*}\right\rangle &=&  \Big\langle  a^{*}
\otimes b^{*} \otimes c^{*},\sum_{i,j} a_i\rhd a_j \otimes b_j
\otimes  b_i\Big\rangle=:  \left\langle   a^{*} \otimes
b^{*}\otimes c^*  , r_{13} \rhd r_{12}\right\rangle.
\end{eqnarray*}}
Thus we have
\vsb
\begin{eqnarray*}
\left\langle \tilde{r}\left(a^{*}\right) \rhd
\tilde{r}\left(b^{*}\right)-\tilde{r}\big(L_{\diamond}^{*}\left(\tilde{r}\left(a^{*}\right)\right)
b^{*}\big)-\tilde{r}\big(R_{\rhd}^{*}(\tilde{r} (b^{*}))
a^{*}\big), c^{*}\right\rangle
= \left\langle  a^{*}  \otimes b^{*} \otimes c^{*},\mathbf{D}(r)
+\sigma_{13}\mathbf{D}(r)\right\rangle.
\end{eqnarray*}
Similarly, we obtain
\vsb
\begin{align*}
 \left\langle \tilde{r}\left(a^{*}\right) \lhd \tilde{r}\left(b^{*}\right)-\tilde{r}(R_{\bullet}^*\left(\tilde{r}\left(a^{*}\right)\right) b^{*})+\tilde{r}\left(R_{\circ}^{*}(\tilde{r}  (b^{*}) ) a^{*}\right), c^{*}\right\rangle &=\left\langle a^{*} \otimes b^{*} \otimes c^{*}, -\sigma_{13}\mathbf{D}(r) \right\rangle,\\
\left\langle [\tilde{r}\left(a^{*}\right)
,\tilde{r}\left(b^{*}\right)]-\tilde{r}(\ad^*\left(\tilde{r}\left(a^{*}\right)\right)
b^{*})+\tilde{r}\left(\ad^{*}(\tilde{r}  (b^{*}) ) a^{*}\right),
c^{*}\right\rangle &=\left\langle a^{*} \otimes b^{*} \otimes
c^{*}, \mathbf{C}(r) \right\rangle,
\vsb
\end{align*}
where  $\sigma_{13}(x \otimes y \otimes z) =z \otimes y \otimes x$
for all $x, y, z \in A$.
Hence $\tilde{r}$ satisfies Eqs.~\eqref{arod1}--\eqref{arod3} if
and only if $\mathbf{C}(r) =\mathbf{D}(r)=0$, giving the conclusion. 
\end{proof}

\vse

\subsection{\texorpdfstring{$\mathcal{O}$}{O}-operators on \mysymbol algebras and pre-\mysymbol algebras}
\begin{defi}\mlabel{ldlrepres}
Let  $(A, \rhd,\lhd,[-,-])$ be a  \mysymbol algebra and
$V$  be a vector space. Let  $l_{\rhd}, r_{\rhd},l_{\lhd}, r_{\lhd}, \rho: A \rightarrow {\rm End}_{\mathbb{F}}(V)$  be linear maps.
A \textbf{representation} of $(A, \rhd,\lhd,[-,-])$  is a sextuple
$(V;l_{\rhd}, r_{\rhd},l_{\lhd}, r_{\lhd}, \rho)$
where
$(V;\rho)$ is a representation of the Lie algebra $(A,[-,-])$
and
$l_{\rhd}, r_{\rhd},l_{\lhd}, r_{\lhd}$ satisfy the following
equations:
\vsb
\begin{align}
r_{\lhd}([x , y])&= r_{\lhd}(x)\rho (y) - r_{\lhd}(y)\rho (x), \mlabel{ldlrep1}\\
l_{\lhd}(x)\rho (y)&= l_{\lhd}([x , y]) - r_{\lhd}(y)\rho (x),\mlabel{ldlrep2}\\
\rho (x)(l_{\lhd}+r_{\lhd})(y) &= (l_{\lhd}+r_{\lhd})([x , y]) = (l_{\lhd}+r_{\lhd})(x) \rho (y)= \rho(x \lhd y + y \lhd x) = 0,\mlabel{ldlrep3}\\
(l_{\rhd}-r_{\lhd})(x)\rho(y) &= \rho(x \circ y) + \rho(y)(l_{\rhd}-r_{\lhd})(x),\mlabel{ldlrep4}\\
(r_{\rhd}-l_{\lhd})([x,y]) &= \rho(x)(r_{\rhd}-l_{\lhd})(y) - \rho(y)(r_{\rhd}-l_{\lhd})(x),\mlabel{ldlrep5}\\
(l_{\rhd}+\rho)(x)l_{\lhd}(y) &= l_{\lhd}(x \bullet y)+ l_{\lhd}(y)(l_{\rhd}+l_{\lhd})(x),\mlabel{ldlrep6}\\
(l_{\rhd}+\rho)(x)r_{\lhd}(y) &= r_{\lhd}(x \circ y)+ r_{\lhd}(y)(l_{\rhd}-r_{\lhd})(x),\mlabel{ldlrep7}\\
r_{\rhd}(x \lhd y) &= r_{\lhd}(y)(r_{\rhd}-l_{\lhd})(x) + l_{\lhd}(x)(r_{\rhd}+r_{\lhd})(y)+\rho(x \lhd y),\mlabel{ldlrep8}\\
r_{\rhd}(x \rhd y) &=  l_{\rhd} (x) r_{\rhd} (y)- r_{\rhd}(y)(l_{\rhd}+l_{\lhd}-r_{\rhd}-r_{\lhd}+\rho)(x) \nonumber\\
&\quad -\rho(x)r_{\lhd}(y)- r_{\lhd}(y)\rho(x)-\rho(x\lhd y),\mlabel{ldlrep9}\\
l_{\rhd}(\{x,y\}) &= l_{\rhd}(x)l_{\rhd}(y) -
l_{\rhd}(y)l_{\rhd}(x)  +\rho(y)l_{\lhd}(x)- \rho(x)l_{\lhd}(y)-
l_{\lhd}([x,y]), \; \forall x,y\in A, \mlabel{ldlrep10}
\end{align}
where $\{-,-\}$, $\circ$ and $\bullet$ are defined by
Eqs.~\eqref{eq:curbra},~\eqref{hpla} and \eqref{vpla}
respectively.
\vsb
\end{defi}

\begin{lem}\mlabel{dualrepdpspl}
Let $(V;l_{\rhd}, r_{\rhd},l_{\lhd}, r_{\lhd}, \rho)$ be a
representation of a \mysymbol algebra $(A,\rhd,\lhd$, $[-,-])$.
Then $(V^*;l_{\rhd}^*-r_{\rhd}^*+l_{\lhd}^*-r_{\lhd}^*,r_{\rhd}^*
,r_{\rhd}^*-l_{\lhd}^*,-(r_{\rhd}^*+r_{\lhd}^*),\rho^*)$ is a
representation of  $(A,\rhd,\lhd,[-,-])$.
\end{lem}
\begin{proof}
For all $x,y \in A$, $v \in V$, $w^* \in V^*$, we have
\begin{eqnarray*}
 &&\big\langle \big(-(r_{\rhd}^*+r_{\lhd}^*)([x , y])+(r_{\rhd}^*+r_{\lhd}^*)(x)\rho^* (y) -(r_{\rhd}^*+r_{\lhd}^*)(y)\rho^* (x)\big)w^* ,v  \big\rangle \\
 &&\overset{\hphantom{\eqref{ldlrep3}}}{=}  \big\langle w^* ,\big((r_{\rhd}+r_{\lhd})([x , y])+\rho (y)(r_{\rhd}+r_{\lhd})(x) -\rho  (x)(r_{\rhd} +r_{\lhd} )(y)\big)v  \big\rangle \\
 &&\overset{\eqref{ldlrep3}}{=} \big\langle w^* ,\big((r_{\rhd}-l_{\lhd})([x , y])+\rho (y)(r_{\rhd}-l_{\lhd})(x) -\rho  (x)(r_{\rhd} -l_{\lhd} )(y)\big)v  \big\rangle \overset{\eqref{ldlrep5}}{=} 0.
\end{eqnarray*}
Thus Eq.~\eqref{ldlrep1} holds. Similarly, Eqs.~\eqref{ldlrep2}--\eqref{ldlrep10} hold. This completes the proof.
\vsb
\end{proof}

\begin{ex}
Let  $(A, \rhd,\lhd,[-,-])$ be a \mysymbol algebra. Then
$(A;L_{\rhd}, R_{\rhd},L_{\lhd},  R_{\lhd},  \mathrm{ad})$  is a
representation of  $(A, \rhd,\lhd,[-,-])$, which is  called the
\textbf{adjoint representation} of  $(A, \rhd,\lhd,[-,-])$.
Moreover,
$(A^*;L_{\diamond}^*,R_{\rhd}^* ,R_{\bullet}^*,-R_{\circ}^*,
\mathrm{ad}^*):=(A^*;L_{\rhd}^*-R_{\rhd}^*+L_{\lhd}^*-R_{\lhd}^*,R_{\rhd}^*
,R_{\rhd}^*-L_{\lhd}^*,-(R_{\rhd}^*+R_{\lhd}^*), \mathrm{ad}^*) $
is also a representation  of $(A, \rhd,\lhd,[-,-])$, which is
called the {\bf coadjoint representation}.
\vsb
\end{ex}

\begin{pro}
Let  $(A, \rhd,\lhd,[-,-])$  be a  \mysymbol algebra, $V$ be a
vector space  and  $l_{\rhd}, r_{\rhd}$, $l_{\lhd}$,
$r_{\lhd},\rho: A\rightarrow \mathrm{End}_{\mathbb F}(V)$ be
linear maps. Then $(V;l_{\rhd}, r_{\rhd},l_{\lhd}, r_{\lhd},
\rho)$ is a representation of $(A$, $\rhd$, $\lhd$, $[-,-])$  if
and only if there is a \mysymbol algebra structure on $A \oplus V$
with $\rhd_{A\oplus V},\lhd_{A\oplus V},$ $[-,-]_{A\oplus V}$
respectively defined by
\vsb
\begin{align}
\left(x_{1}+v_{1}\right) \rhd_{A \oplus V} \left(x_{2}+v_{2}\right) &= x_{1} \rhd x_{2}+l_{\rhd}\left(x_{1}\right) v_{2}+r_{\rhd}\left(x_{2}\right) v_{1}, \mlabel{ldlr1}\\
\left(x_{1}+v_{1}\right) \lhd_{A \oplus V} \left(x_{2}+v_{2}\right) &= x_{1}\lhd x_{2}+l_{\lhd}\left(x_{1}\right) v_{2}+r_{\lhd}\left(x_{2}\right) v_{1}, \mlabel{ldlr2}\\
\left[x_{1}+v_{1}, x_{2}+v_{2}\right]_{A \oplus V} &=
[x_{1},x_{2}]+\rho(x_1)v_2 - \rho(x_2)v_1,\; \forall x_{1},
x_{2} \in A, v_{1}, v_{2} \in V. \mlabel{ldlr3}
\end{align}
We denote it by $A \ltimes_{l_{\rhd}, r_{\rhd},l_{\lhd},
r_{\lhd},\rho}V$.
\vsb
 \end{pro}

\begin{proof}
 It can be checked directly by Definitions \mref{ldldef} and \mref{ldlrepres}.
\vsb
\end{proof}

 \begin{defi}
 Let $(V;l_{\rhd}, r_{\rhd},l_{\lhd},  r_{\lhd}, \rho)$ be a representation of  a   \mysymbol algebra   $(A,\rhd$, $\lhd$, $[-,-])$. A linear map $T: V \to A$ is called an {\bf $\mathcal{O}$-operator
  on $(A,\rhd,\lhd,[-,-])$ associated to $(V;l_{\rhd}, r_{\rhd},l_{\lhd}, r_{\lhd}, \rho)$} if $T$
  satisfies the following conditions.
\vsb
\begin{align}
T(u) \rhd  T(v) &= T \big(l_{\rhd}(T(u)) v + r_{\rhd}(T(v)) u \big), \mlabel{LDLOP1}\\
T(u) \lhd  T(v) &= T \big(l_{\lhd}(T(u)) v + r_{\lhd}(T(v)) u \big),\mlabel{LDLOP2}\\
[T(u), T(v)] &= T \big(\rho(T(u))v-\rho(T(v)) u \big),\quad \forall u, v \in V. \mlabel{LDLOP3}
\end{align}
In particular, if $P$ is an   $\mathcal{O}$-operator on  a
\mysymbol algebra  $(A,\rhd,\lhd,[-,-])$ associated to the adjoint
representation
 $(A;L_{\rhd}, R_{\rhd},L_{\lhd}, R_{\lhd}, \ad)$, then $P$ is called a {\bf Rota-Baxter operator  of weight  zero} on  $(A,\rhd,\lhd,[-,-])$.
\vsc
\end{defi}

Then Proposition~\mref{qLDCYBEO2} can be restated as follows.
\vsb
\begin{cor}\mlabel{LDCYBEO}
Let $(A, \rhd, \lhd,[-,-])$  be a   \mysymbol algebra  and $r \in
A \otimes A$ be  antisymmetric.  Then $r$ is a solution of the
\PCYBE in $(A, \rhd, \lhd,[-,-])$ if and only if $\tilde{r}$ is an
$\mathcal{O}$-operator on $(A, \rhd, \lhd,[-,-])$ associated to
the representation $(A^*;L_{\diamond}^*,R_{\rhd}^*
,R_{\bullet}^*,-R_{\circ}^*, \mathrm{ad}^*)$.
\end{cor}
\vsc

 \begin{thm}\mlabel{oggs}
  Let $(V;l_{\rhd}, r_{\rhd},l_{\lhd},  r_{\lhd}, \rho)$ be a representation of  a   \mysymbol algebra   $(A,\rhd$, $\lhd$, $[-,-])$.
  Set $\hat{A} =  A \ltimes_{l_{\rhd}^{*}-r_{\rhd}^{*}+l_{\lhd}^{*}-r_{\lhd}^{*}, r_{\rhd}^{*}, r_{\rhd}^{*}-l_{\lhd}^{*},-(r_{\rhd}^{*}+r_{\lhd}^{*}),\rho^*} V^{*}$. Let  $T: V \rightarrow A$  be a linear map which is identified as an element in the vector space  $\hat{A}  \otimes \hat{A}$  $($through $\operatorname{Hom}_{\mathbb{F}}(V, A) \cong V^{*} \otimes A \subseteq \hat{A} \otimes \hat{A} $$)$. Then  $r=T-\tau(T)$  is an antisymmetric solution of the \PCYBE in the \mysymbol algebra  $\hat{A}$  if and only if  $T$  is an  $\mathcal{O}$-operator on  $(A, \rhd, \lhd,[-,-])$  associated to  $\left(V;l_{\rhd}, r_{\rhd}, l_{\lhd}, r_{\lhd},\rho\right)$.
  \end{thm}
\vsc
  \begin{proof}
 It is similar to the proof of \cite[Theorem 4.14]{BLN}.
\vsb
  \end{proof}

 \begin{defi}
A  \textbf{pre-\mysymbol algebra} is a sextuple-tuple  $(A ,\searrow , \nearrow,\swarrow ,\nwarrow,\cdot)$  such that  $(A, \cdot)$   is a pre-Lie algebra and $ \searrow, \nearrow ,\swarrow,\nwarrow: A \otimes A \rightarrow A$  are binary operations satisfying
\vsb
{\small 
\begin{align}
x \nwarrow  [y,z]  &=   (z \cdot x) \nwarrow  y  -  (y \cdot x) \nwarrow  z,\mlabel{pldl1}\\
x \swarrow  (y \cdot z)  &=   [x,y] \swarrow z  -  (x \cdot z) \nwarrow  y,\mlabel{pldl2}\\
x \cdot (y \swarrow z + z \nwarrow y) &=   x \swarrow (y \cdot z) + (y \cdot z) \nwarrow y    = 0,\mlabel{pldl3}\\
[x,y]\swarrow z + z \nwarrow[x,y] &= (x \lhd y + y \lhd x) \cdot z = 0,\mlabel{pldl4}\\
x \vee (y \cdot z)   &= (x \circ y) \cdot z + y \cdot (x \vee z  ),\mlabel{pldl5}\\
  x \wedge [y,z]    &=   y \cdot (x \wedge z ) - z \cdot (x \wedge  y),\mlabel{pldl6}\\
x \searrow(y \swarrow z)+ x \cdot (y \swarrow z)  &=y \swarrow(x \vee z)+(x \searrow y+x \nearrow y-y \swarrow x-y \nwarrow x) \swarrow z,\mlabel{pldl7}\\
x \searrow(y \nwarrow z)+ x \cdot (y \nwarrow z)&=y \nwarrow(x \circ z) + (x \searrow y - y \nwarrow x) \nwarrow z , \mlabel{pldl8}\\
x \nearrow(y \lhd z) - (y \lhd z) \cdot x& =y \swarrow(x \wedge z) +(x \nearrow y -y \swarrow x) \nwarrow z,\mlabel{pldl9}\\
x \searrow(y \nearrow z)-y \nearrow(x \rhd z) &= (x \vee y - y \wedge x) \nearrow z + x \cdot (y \nwarrow z)+\! (x \cdot y)\nwarrow z + \!(x \lhd z) \cdot y,\mlabel{pldl10}\\
 \{x, y\}   \searrow z+  [x,y] \swarrow z &=x \searrow(y \searrow z)-y \searrow(x \searrow z)+ y \cdot ( x\swarrow z) -x \cdot ( y\swarrow z), \mlabel{pldl11}
\end{align}
}
for all $x,y,z \in A$, where
{\small
\begin{align*}
[x,y] &:=  x \cdot y - y \cdot x, \hspace{4cm}  x \rhd y  :=x \searrow y+x \nearrow y, \quad x \lhd y:= x \nwarrow y+x \swarrow y, \\
x \circ y &:=x \searrow y+x \nearrow y+x \swarrow y+x \nwarrow y,\hspace{0.32cm}x \vee y : =  x \searrow y+x \swarrow y, \hspace{0.42cm}  x \wedge y := x \nearrow y+x \nwarrow y,\\
\{x, y\}&:=  x \searrow y+x \nearrow y+x \swarrow y +x \nwarrow y- y \searrow x-y \nearrow x-y \swarrow x-y \nwarrow x + x \cdot y - y \cdot x.
\end{align*}
}
\end{defi}
\vsd
\begin{rmk}
When $\cdot = 0$,  a pre-\mysymbol algebra reduces to an L-quadri
algebra. The latter is regarded as the Lie algebraic analog of
a quadri-algebra~\mcite{LNB}. Moreover the operad of L-quadri
algebras is the successor of the operad of L-dendriform algebras
\mcite{BBGN}.
\vsb
\end{rmk}

\begin{pro}  \mlabel{preldltoldl}
Let  $(A ,\searrow , \nearrow ,\swarrow,\nwarrow,\cdot)$  be a pre-\mysymbol algebra. Define
\vsc
\begin{align}
x \rhd y : = x \searrow  y + x  \nearrow y,\hspace{0.2cm}x \lhd y  := x \swarrow y +  x \nwarrow y , \hspace{0.2cm} [x,y]  := x \cdot y - y \cdot x, \hspace{0.2cm}  \forall x,y\in A. \mlabel{pLDL}
\vsb
\end{align}
Then $(A,\rhd,\lhd,[-,-])$ is  a  \mysymbol algebra,  called the
{\bf sub-adjacent \mysymbol algebra} of   $(A ,\searrow$,
$\nearrow,\swarrow,\nwarrow,\cdot)$, and $(A ,\searrow , \nearrow
,\swarrow,\nwarrow,\cdot)$  is   called a {\bf compatible
pre-\mysymbol algebra} structure on the \mysymbol algebra
$(A,\rhd,\lhd,[-,-])$. Moreover, $(A; L_{\searrow },
R_{\nearrow},L_{\swarrow}, R_{\nwarrow},L_{\cdot})$ is a
representation of the \mysymbol algebra $(A,\rhd,\lhd,[-,-])$.
\vsb
\end{pro}
\begin{proof}
Let $x,y,z \in A$. Then we have
\vsb
\begin{eqnarray*}
&&x \lhd [y,z] -   [x,y] \lhd z -   [z,x] \lhd y\\
&&\overset{\eqref{pLDL}}{=} x  \swarrow (y \cdot z - z \cdot y) + x  \nwarrow [y ,z] -  [x,y] \swarrow  z - (x \cdot y - y \cdot x) \nwarrow  z \\
&& \hspace{0.7cm}-  [z,x]  \swarrow  y - (z \cdot x - x \cdot z) \nwarrow  y\\
&&\overset{\eqref{pldl1}}{=} x  \swarrow (y \cdot z - z \cdot y)  - [x,y] \swarrow  z  - (x \cdot y ) \nwarrow  z  +[x,z] \swarrow  y  + (  x \cdot z) \nwarrow  y\\
&&\overset{\eqref{pldl2}}{=}   [x,z] \swarrow  y -x  \swarrow (  z
\cdot y)  - (x \cdot y ) \nwarrow
z\overset{\hphantom{\eqref{ldlrep3}}}{=}  0.
\end{eqnarray*}
Hence Eq.~\eqref{DPSPL1} holds. Similarly,
Eqs.~\eqref{DPSPL2}--\eqref{DPSPL5} hold.  Thus
$(A,\rhd,\lhd,[-,-])$ is a  \mysymbol algebra. Moreover, it is
straightforward to show that $(A; L_{\searrow },
R_{\nearrow},L_{\swarrow}, R_{\nwarrow},L_{\cdot})$ is a
representation of $(A,\rhd,\lhd,[-,-])$.
\vsc
\end{proof}

\begin{thm}\mlabel{ofnagz}
Let $T: V \rightarrow A $ be an  $\mathcal{O}$-operator on a
\mysymbol algebra $(A, \rhd,\lhd,[-,-])$  associated to a
representation  $(V;l_{\rhd}, r_{\rhd},l_{\lhd}, r_{\lhd}, \rho)$.
Define
\vsb
\begin{align*}
u \searrow  v &:= l_{\rhd}(T(u)) v, \quad u \nearrow v:= r_{\rhd}(T(v)) u, \quad u \cdot v := \rho(T(u)) v,\\
 u \swarrow v &:= l_{\lhd}(T(u)) v,\quad u  \nwarrow v := r_{\lhd}(T(v)) u,\quad \forall u,v \in V.
\end{align*}
Then   $(V,\searrow , \nearrow ,\swarrow,\nwarrow,\cdot)$  is a pre-\mysymbol algebra.
\vsb
\end{thm}
\begin{proof}
It is straightforward to verify the axioms of a pre-\mysymbol algebra.
\vsb
\end{proof}

\begin{pro}\mlabel{iogpfna}
 There is a compatible pre-\mysymbol algebra structure on a    \mysymbol algebra $(A, \rhd,\lhd,[-,-])$
if and only if there exists an invertible $\mathcal{O}$-operator
on $(A, \rhd,\lhd,[-,-])$.
\end{pro}

\begin{proof} Suppose that  $T: V \rightarrow A$  is an invertible  $\mathcal{O}$-operator on  $(A, \rhd,\lhd,[-,-])$ associated to a representation  $(V;l_{\rhd}, r_{\rhd},l_{\lhd},  r_{\lhd},
\rho)$. Define the binary operations $\searrow , \nearrow
,\swarrow,\nwarrow,\cdot$ respectively by
{\small
\begin{align*}
x \searrow  y &=T\left(l_{\rhd}(x)\left(T^{-1}(y)\right)\right), \quad x \nearrow  y =T\left(r_{\rhd}(y)\left(T^{-1}(x)\right)\right), \quad x \cdot y := T\left(\rho(x) T^{-1}(y)\right),\\
x \swarrow y &=T\left(l_{\lhd}(x)\left(T^{-1}(y)\right)\right),
\quad x \nwarrow y
=T\left(r_{\lhd}(y)\left(T^{-1}(x)\right)\right), \quad \forall x,
y \in A.
\end{align*}
}
Then applying Theorem \mref{ofnagz}, a direct verification shows that
$(A,\searrow , \nearrow ,\swarrow,\nwarrow,\cdot)$ is a compatible
pre-\mysymbol algebra structure on  $(A, \rhd,\lhd,[-,-])$.

 Conversely, let $(A,\searrow , \nearrow ,\swarrow,\nwarrow,\cdot)$ be a pre-\mysymbol algebra and  $(A, \rhd,\lhd,[-,-])$ be the sub-adjacent \mysymbol algebra.
 By Proposition \mref{preldltoldl}, the identity map $\mathrm{id}:  A \rightarrow A$ is an invertible $\mathcal{O}$-operator on  $(A, \rhd,\lhd,[-,-])$   associated to the representation
 $\left(A; L_{\searrow}, R_{\nearrow},L_{\swarrow}, R_{\nwarrow},L_{\cdot} \right)$.
\end{proof}

\begin{cor} \mlabel{pretoaldlcYBE}
Let   $(A ,\searrow , \nearrow ,\swarrow,\nwarrow,\cdot)$ be a
pre-\mysymbol algebra and $(A, \rhd,\lhd,[-,-])$ be the
sub-adjacent \mysymbol algebra. Let   $\left\{e_{1}, \ldots,
e_{n}\right\}$ be a basis of  $A$  and $\left\{e_{1}^{*}, \ldots,
e_{n}^{*}\right\}$ be the dual basis. Then \vsc
\begin{align}
r:=\sum_{i=1}^{n}\left(e_{i}^{*}  \otimes e_{i}- e_{i}\otimes e_{i}^{*} \right)
\end{align}
is an antisymmetric solution of the \PCYBE in the \mysymbol
algebra
\vsb
\begin{align}\label{bigldl}
\widehat{A}: = A \ltimes_{L_{\searrow}^*-R_{\nearrow}^*+L_{\swarrow}^*-R_{\nwarrow}^*,R_{\nearrow}^* ,R_{\nearrow}^*-L_{\swarrow}^*,-(R_{\nearrow}^*+R_{\nwarrow}^*), L_{\cdot}^*} A^{*}.
\end{align}
\vse \vsc
\end{cor}
\begin{proof} By Proposition \mref{iogpfna}, the identity map $\mathrm{id} :  A \rightarrow A$  is an invertible $\mathcal{O}$-operator on
  $(A, \rhd,\lhd,$ $[-,-])$ associated to the representation   $\left(A; L_{\searrow}, R_{\nearrow},L_{\swarrow}, R_{\nwarrow},L_{\cdot} \right)$. Hence the conclusion follows from Theorem \mref{oggs}.
\end{proof}
\vsc
We end the paper by presenting an example to illustrate the above construction.
\vsb
\begin{ex} In Example \mref{ex:ldl}, we define a  Rota-Baxter operator $P$ of weight  zero on the \mysymbol
algebra $(\mathfrak{s l}(2, \mathbb{C}),\rhd,\lhd,[-,-])$ by
\vsb
\begin{align*}
P(e_1) = e_1, \quad P(e_2) = P(e_3) = 0.
\vsb
\end{align*}
Thus by Theorem \mref{ofnagz}, there is a pre-\mysymbol algebra
structure $(A,\!\searrow , \nearrow , \swarrow ,\nwarrow,\cdot)$
on the vector space $ A:= \mathfrak{s l}(2, \mathbb{C}) $  whose
nonzero products are given by
\vsb
{\small
\begin{align*}
e_1 \searrow e_2 &=   \frac{\mathrm{i} }{2}e_2 +   \frac{1}{2}e_3 , && e_1 \searrow e_3 =   -\frac{1 }{2}e_2 +   \frac{\mathrm{i}}{2}e_3,   && e_2 \nearrow e_1 =   e_3 ,& e_3 \nearrow e_1 =   - e_2 , \\
e_1\swarrow e_2 &=  - \frac{\mathrm{i} }{2}e_2 +   \frac{1}{2}e_3 , &&  e_1 \swarrow e_3 =   -\frac{1 }{2}e_2 -  \frac{\mathrm{i}}{2}e_3,   && \hspace{0.36cm} e_1  \cdot e_2   =   e_3 , & e_1  \cdot e_3 =   - e_2 , \\
 e_2 \nwarrow e_1 &=    \frac{\mathrm{i} }{2}e_2 -   \frac{1}{2}e_3 , &&  e_3 \nwarrow e_1 =     \frac{1 }{2}e_2 +  \frac{\mathrm{i}}{2}e_3.
\end{align*}
}
Then the nonzero products of the  sub-adjacent \mysymbol algebra
$(A,\rhd'\!,\lhd'\!,[-,-]')$ of $(A,\!\searrow$, $\nearrow$,
$\swarrow ,\nwarrow,\cdot)$ are given by
\vsb
{\small
\begin{align}
e_{1} \rhd'  e_{2}  &=  \frac{\mathrm{i}}{2} e_2 + \frac{1}{2}
e_3, &&
e_{1} \rhd'  e_{3}  = -\frac{1}{2} e_2 + \frac{\mathrm{i}}{2}  e_3, &&   e_{2} \rhd'  e_{1}  = e_3,  \quad  e_{3} \rhd'  e_{1}=-e_2, \label{exi}\\
e_{1} \lhd'  e_{2}  &= -\frac{\mathrm{i}}{2} e_2 + \frac{1}{2}  e_3, &&  e_{1} \lhd'  e_{3}  = -\frac{1}{2} e_2 - \frac{\mathrm{i}}{2}  e_3,  && \hspace{0.1cm}[e_{1} , e_{2} ]'   = e_3,  \hspace{0.52cm} [e_{1} , e_{3} ]' =
-e_2, \label{exii}\\
e_{2} \lhd'  e_{1} &=\frac{\mathrm{i}}{2} e_2 - \frac{1}{2}  e_3, &&
  e_{3} \lhd'  e_{1}=\frac{1}{2} e_2 + \frac{\mathrm{i}}{2}  e_3.  \label{exiii}
\end{align}
}
Let $\{e_1^*,e_2^*,e_3^*\}$ be the dual basis.
Then there is a larger \mysymbol algebra $\widehat{A}$ defined by Eq.~\eqref{bigldl}. Explicitly, the nonzero products are given by  Eqs.~\eqref{exi}--\eqref{exiii} and the following equations.\
\vsb
 \begin{align*}
e_{1} \rhd'  e_{2}^*  &=  \frac{\mathrm{i}}{2} e_2^*  + \frac{1}{2}
e_3^* , &&
e_{1} \rhd'  e_{3}^*   = -\frac{1}{2} e_2^*  + \frac{\mathrm{i}}{2}  e_3^* , &&   e_{2}^*  \rhd'  e_{1}  = e_3^* ,  \quad  e_{3}^*  \rhd'  e_{1}=-e_2^* , \\
e_{1} \lhd'  e_{2}^*  &= -\frac{\mathrm{i}}{2} e_2^* + \frac{1}{2}  e_3^*, &&  e_{1} \lhd'  e_{3}^*  = -\frac{1}{2} e_2^* - \frac{\mathrm{i}}{2}  e_3^*,  && [e_{1} , e_{2}^*]'   = e_3^*,  \hspace{0.45cm} [e_{1},e_{3}^*]' =-
e_2^*, \\
e_{2}^* \lhd'  e_{1} &=\frac{\mathrm{i}}{2} e_2^* - \frac{1}{2}  e_3^*, &&
  e_{3}^* \lhd'  e_{1}=\frac{1}{2} e_2^* + \frac{\mathrm{i}}{2}  e_3^*.
\end{align*}
Thus by Corollary \mref{pretoaldlcYBE}, $r=\sum\limits_{i=1}^{3}\left(e_{i}^{*} \otimes e_{i}  - e_{i} \otimes e_{i}^{*}\right)$ is an
 antisymmetric  solution of the \PCYBE in the \mysymbol algebra  $\widehat{A}$. By Corollary \mref{iartoldlbia}, there is a  \mysymbol bialgebra $ (\widehat{A}, \delta_{\rhd'},\delta_{\lhd'},\Delta')$ where $\delta_{\rhd'},\delta_{\lhd'},\Delta'$ are given by
\vsb
{\small \begin{align*}
\delta_{\rhd'}(e_1) &= 0, &&\delta_{\rhd'}(e_1^*)  = 0, \\
\delta_{\rhd'}(e_2) &= e_1^* \otimes \Big(\frac{\mathrm{i}}{2} e_2 - \frac{1}{2} e_3 \Big)- e_3 \otimes e_1^*  , &&\delta_{\rhd'}(e_2^*)  =  e_1^* \otimes \Big(\frac{\mathrm{i}}{2} e_2^* - \frac{1}{2} e_3^* \Big) - e_3^* \otimes e_1^* ,\\
\delta_{\rhd'}(e_3) &=  e_2 \otimes e_1^* + e_1^* \otimes \Big(\frac{1}{2} e_2 + \frac{\mathrm{i}}{2} e_3 \Big), &&\delta_{\rhd'}(e_3^*)  =   e_2^* \otimes e_1^* + e_1^* \otimes \Big(\frac{1}{2} e_2^* + \frac{\mathrm{i}}{2} e_3^* \Big), \\
\delta_{\lhd'}(e_1) &= 0, &&\delta_{\lhd'}(e_1^*)  = 0,
\end{align*}
\vse
\begin{align*}
\delta_{\lhd'}(e_2) &= e_1^* \otimes \Big(\frac{\mathrm{i}}{2} e_2  + \frac{1}{2} e_3  \Big)-\Big(\frac{\mathrm{i}}{2} e_2  + \frac{1}{2} e_3  \Big) \otimes e_1^*  , \\
\delta_{\lhd'}(e_2^*)  &=e_1^* \otimes \Big(\frac{\mathrm{i}}{2} e_2^*  + \frac{1}{2} e_3^*  \Big)- \Big(\frac{\mathrm{i}}{2} e_2^*  + \frac{1}{2} e_3^*  \Big) \otimes e_1^* ,  \\
\delta_{\lhd'}(e_3) &= \Big(  \frac{1}{2} e_2  - \frac{\mathrm{i}}{2} e_3   \Big) \otimes e_1^* + e_1^* \otimes \Big(    \frac{\mathrm{i}}{2} e_3  -\frac{1}{2} e_2  \Big),\\
\delta_{\lhd'}(e_3^*)  &=\Big(  \frac{1}{2} e_2^* - \frac{\mathrm{i}}{2} e_3^*     \Big) \otimes e_1^* + e_1^* \otimes \Big(     \frac{\mathrm{i}}{2} e_3^* -\frac{1}{2} e_2^*  \Big),
\end{align*}
\vse
\begin{align*}
\Delta'(e_1) &= 0, && \Delta'(e_1^*) =0,\\
\Delta'(e_2) &=  e_3 \otimes e_1^* -  e_1^* \otimes e_3, && \Delta'(e_2^*) =  e_3^* \otimes e_1^* -  e_1^* \otimes e_3^*,\\
\Delta'(e_3) &=   e_1^* \otimes e_2-e_2 \otimes e_1^*, && \Delta'(e_3^*) =   e_1^* \otimes e_2^*-e_2^* \otimes e_1^*.
 \end{align*}
}
\end{ex}
\vsb

 \noindent
 {\bf Acknowledgements.} This work is supported by NSFC  (11931009, 12271265, 12261131498, 12326319), Fundamental Research Funds for the Central Universities and Nankai Zhide Foundation.

\smallskip

\noindent
{\bf Declaration of interests. } The authors have no conflicts of interest to disclose.

\smallskip

\noindent
{\bf Data availability. } No new data were created or analyzed in this study.

\end{document}